\documentclass[11pt]{amsart}
\usepackage{amssymb}
\usepackage{amsmath}
\usepackage{amsthm}
\usepackage{mathrsfs}
\usepackage[all]{xy}
\usepackage{hyperref}
\usepackage{amsrefs}
\usepackage{graphicx}
\usepackage{tikz}
\usetikzlibrary{intersections, calc, arrows.meta}

\newtheorem{theorem}{Theorem}[section]
\newtheorem{proposition}[theorem]{Proposition}
\newtheorem{lemma}[theorem]{Lemma}
\newtheorem{corollary}[theorem]{Corollary}
\newtheorem{assumption}[theorem]{Assumption}

\theoremstyle{definition}
\newtheorem{definition}[theorem]{Definition} 
 
\newtheorem{remark}[theorem]{Remark} 
\newtheorem{example}[theorem]{Example}

\newcommand{\AI}{\bfA^I}
\newcommand{\LI}{{\bfL_I}}
\newcommand{\nablaI}{\nabla^I}

\newcommand{\partialaj}{\partial_{\alpha,j}}
\newcommand{\tts}{t,\tau,\sigma} 
\newcommand{\Ral}{R_\lambda}

\newcommand{\ringA}{\cO^{x,\sigma,z,z^{-1}}_{qe^t,\lambda}}
\newcommand{\ringB}{\cO^{x,\sigma,z}_{qe^t,\lambda}}

\newcommand{\C}{{\mathbb C}}
\newcommand{\N}{{\mathbb N}}
\newcommand{\PP}{{\mathbb P}}
\newcommand{\R}{{\mathbb R}}
\newcommand{\Z}{{\mathbb Z}}

\newcommand{\bc}{{\boldsymbol c}}
\newcommand{\bd}{{\boldsymbol d}}

\newcommand{\bm}{{\boldsymbol m}}
\newcommand{\bt}{{\boldsymbol t}}

\newcommand{\bH}{{\boldsymbol H}}
\newcommand{\bL}{{\boldsymbol L}}

\newcommand{\bfA}{\mathbf{A}}

\newcommand{\bfL}{\mathbf{L}}
\newcommand{\bfM}{\mathbf{M}}
\newcommand{\bfN}{\mathbf{N}}
\newcommand{\bfR}{\mathbf{R}}
\newcommand{\bfGamma}{\mathbf{\Gamma}}

\newcommand{\cB}{\mathcal{B}}
\newcommand{\cD}{\mathcal{D}} 
\newcommand{\cE}{\mathcal{E}}
\newcommand{\cF}{\mathcal{F}}
\newcommand{\cH}{\mathcal{H}} 
\newcommand{\cI}{\mathcal{I}} 
 
\newcommand{\cL}{\mathcal{L}} 
\newcommand{\cM}{\mathcal{M}} 
\newcommand{\cN}{\mathcal{N}} 
\newcommand{\cO}{\mathcal{O}}
\newcommand{\cU}{\mathcal{U}}
\newcommand{\cV}{\mathcal{V}}
\newcommand{\cW}{\mathcal{W}}

\newcommand{\chbfR}{\check{\bfR}}
\newcommand{\chcE}{\check{\cE}}
\newcommand{\chtau}{\check{\tau}}
\newcommand{\chx}{\check{x}}

\newcommand{\hcE}{\widehat{\cE}}
\newcommand{\hcI}{\widehat{\cI}}
\newcommand{\hcM}{\widehat{\cM}}

\newcommand{\hsigma}{\hat{\sigma}}

\newcommand{\htau}{\hat{\tau}}

\newcommand{\hy}{\hat{y}}

\newcommand{\hGamma}{\widehat{\Gamma}}

\newcommand{\hR}{\widehat{\bfR}}
\newcommand{\hS}{\widehat{S}}
\newcommand{\hU}{\widehat{U}}
\newcommand{\hV}{\widehat{V}}

\newcommand{\sfa}{\mathsf{a}}
\newcommand{\sfb}{\mathsf{b}}
\newcommand{\sfe}{\mathsf{e}}
\newcommand{\sfp}{\mathsf{p}}
\newcommand{\sfs}{\mathsf{s}}
\newcommand{\sfK}{\mathsf{K}}
\newcommand{\sfS}{\mathsf{S}}

\newcommand{\ta}{\tilde{a}}
\newcommand{\tbfA}{\widetilde{\bfA}}
\newcommand{\tbfL}{\widetilde{\bfL}}
\newcommand{\tbfM}{\widetilde{\bfM}}

\newcommand{\tsfa}{\widetilde{\sfa}}
\newcommand{\tsfb}{\widetilde{\sfb}}
\newcommand{\tI}{\widetilde{I}}

\newcommand{\oanQDM}{\overline{\anQDM}}
\newcommand{\oanQDMT}{\overline{\anQDMT}}

\newcommand{\olambda}{\overline{\lambda}}

\newcommand{\anQDM}{\operatorname{QDM}^{\mathrm{an}}}
\newcommand{\anQDMT}{\operatorname{QDM}^{\mathrm{an}}_T}
\newcommand{\ev}{\operatorname{ev}}
\newcommand{\exc}{{\operatorname{exc}}}
\newcommand{\loc}{\operatorname{loc}}
\newcommand{\pr}{\operatorname{pr}}
\newcommand{\pt}{\operatorname{pt}}
\newcommand{\virt}{\operatorname{virt}}

\newcommand{\Crit}{\operatorname{Crit}}
\newcommand{\Eff}{\operatorname{Eff}}
\newcommand{\End}{\operatorname{End}}
\newcommand{\Exc}{\operatorname{Exc}}
\newcommand{\Frac}{\operatorname{Frac}}
\newcommand{\Gr}{\operatorname{Gr}}
\newcommand{\Hom}{\operatorname{Hom}}
\newcommand{\Id}{\operatorname{Id}}

\newcommand{\Lie}{\operatorname{Lie}}
\newcommand{\Mat}{\operatorname{Mat}}

\newcommand{\QDM}{\operatorname{QDM}}
\newcommand{\QDMT}{\operatorname{QDM}_T}
\newcommand{\Res}{\operatorname{Res}}
\newcommand{\Spec}{\operatorname{Spec}}
\newcommand{\Vol}{\operatorname{Vol}}

\newcommand{\auj}{\alpha^*u_j}
\newcommand{\aUj}{\alpha^*U_j}
\newcommand{\aUjo}{\alpha^*U^\circ_j}
\newcommand{\bigI}{I^{\operatorname{big}}_E}
\newcommand{\bigIE}{I^{\operatorname{big}}_E(\tts,z)}

\newcommand{\corr}[1]{\left\langle#1\right\rangle} 

\newcommand{\pair}[2]{\left\langle#1,#2\right\rangle}
\newcommand{\pairr}[2]{\langle\!\langle#1,#2\rangle\!\rangle}
\newcommand{\parfrac}[2]{\frac{\partial #1}{\partial #2}}

\begin{document}

\title[Convergence of Quantum Cohomology of Toric Bundles]{Convergence and Analytic Decomposition of Quantum Cohomology of Toric Bundles}
\author{Yuki Koto}
\address{Department of Mathematics, Graduate School of Science, Kyoto University, 606-8502, Japan}
\email{y.koutou@math.kyoto-u.ac.jp}

\begin{abstract}
We prove that the equivariant big quantum cohomology $QH^*_T(E)$ of the total space of a toric bundle $E \to B$ converges provided that the big quantum cohomology $QH^*(B)$ converges. 
The proof is based on Brown's mirror theorem for toric bundles \cite{Brown}. 
It has been observed by Coates, Givental and Tseng that the quantum connection of $E$ splits into copies of that of $B$ \cite{Coates-Givental-Tseng}.
Under the assumption that $QH^*(B)$ is convergent, we construct a decomposition of the quantum $D$-module of $E$ into a direct sum of that of $B$, which is analytic with respect to parameters of $QH^*_T(E)$.
In particular, we obtain an analytic decomposition for the equivariant/non-equivariant big quantum cohomology of $E$.
\end{abstract}

\maketitle

\tableofcontents 

\section{Introduction}

Quantum cohomology and mirror symmetry for toric varieties have been investigated in depth by many researchers.
Studies include Givental's mirror theorem \cite{Givental-mirror} (the equality between the $I$-function and the $J$-function), Landau-Ginzburg mirror symmetry \cite{Barannikov, Coates-Corti-Iritani-Tseng-hodge, Givental-homological, Givental-mirror}, and analyticity of quantum cohomology \cite{Iritani-convergence}.
Some of these studies have been extended to a relative setting, namely to toric bundles $E \to B$ which are given as fiberwise GIT quotients of direct sums of line bundles over $B$.
Brown \cite{Brown} introduced the $I$-function for such toric bundles as a hypergeometric modification of the $J$-function of the base $B$, and proved a generalization of Givental's mirror theorem.
Brown's work has been further generalized to higher genera by Coates, Givental and Tseng \cite{Coates-Givental-Tseng}.
Based on Brown's result, Lee, Lin and Wang \cite{Lee-Lin-Wang-invariance2} proved a quantum Leray-Hirsch theorem, which relates quantum differential equations of $E$ to those of $B$.
These results imply that Gromov-Witten theory of $E$ is closely related to that of $B$.

The analyticity of quantum cohomology is significant in Gromov-Witten theory, mirror symmetry and birational geometry.
For instance, Ruan \cite{Ruan} proposed the crepant transformation conjecture, which says that the quantum cohomology of two $K$-equivalent projective manifolds are related via analytic continuation in the quantum parameters.
As mentioned above, it is known that quantum cohomology of toric varieties is convergent.
Thus, it is natural to expect that quantum cohomology of $E$ is convergent if we assume the analyticity of that of $B$.

In this paper, we discuss the convergence and an analytic decomposition of the quantum cohomology (and the quantum connection) of toric bundles $E \to B$.
Our main results are stated as follows.

\begin{theorem}[Corollary \ref{cor:convergence_of_QH(E)}, Theorem \ref{thm:decompQDM}]
\label{thm:main1}
Assume that the big quantum cohomology $QH^*(B)$ of the base $B$ has convergent structure constants.

{\rm (1)}
The big equivariant quantum cohomology $QH^*_T(E)$ of the total space $E$ has convergent structure constants.

{\rm (2)}
The equivariant quantum connection of $E$ is gauge equivalent to a direct sum of the quantum connection of $B$ via an analytic gauge transformation and an analytic variable change.
\end{theorem} 

We explain a more precise meaning of Theorem \ref{thm:main1}.
Let $E \to B$ be a bundle with toric fiber $X$ constructed by a fiberwise GIT quotient of a direct sum of line bundles over $B$.
Note that $E$ and $X$ are naturally endowed with a torus action; we denote this torus by $T$ and the fixed point set of the $T$-action on $X$ by $F$.
\emph{Quantum cohomology} $( QH^*(B), \star_B )$ of $B$ is a deformation of the ring structure of $H^*(B)$ parametrized by $\tau \in H^*(B,\C)$ and the Novikov variables $Q \in H^2(B,\C^\times)$\footnote{This is a redundant parametrization. We will introduce a modified parameter in Section 2.}. 
For $\alpha \in F$, we also denote by $\alpha$ the section $B \to E$ corresponding to $\alpha$.
The structure constants of $\star_B$ are a priori defined as formal power series in $\tau$ and $Q$; we assume that they are convergent. 
Similarly, the $T$-equivariant quantum cohomology $( QH^*_T(E), \star_E )$ is a deformation of the ring structure of $H^*(E)$ parametrized by $\hsigma \in H^*(E,\C)$, $(q,Q) \in H^2(X,\C^\times) \times H^2(B,\C^\times) \cong H^2(E,\C^\times)$\footnote{This isomorphism is not canonical. We need to specify a splitting of the sequence $0 \to H^2(B) \to H^2(E) \to H^2(X) \to 0$.}
and the equivariant parameter $\lambda \in \Lie(T)$.
Here $q$ and $Q$ are the Novikov variables of the fiber $X$ and the base $B$ respectively.
In this notation, Theorem \ref{thm:main1} (1) states that the structure constants of $\star_E$ are convergent as power series of $(\lambda,q,Q,\hsigma)$.

The equivariant quantum product $\star_E$ defines the \emph{equivariant quantum connection} $\nabla^E$, which is, roughly speaking, a family of flat partial connections on the trivial $H^*(E)$-bundle over an open neighborhood of $0 \in H^*(E) \times \Spec \C[q,Q]$ parametrized by $\Lie (T) \times \C^\times_z$.
It is given as follows:
\[
\nabla^E = d + z^{-1} \sum_a (e_a \star_E) d\hsigma_a
\]
where $\{ e_a \}$ is a (homogeneous) basis of the vector space $H^*(E)$ and $\{ \hsigma_a \}$ are linear coordinates dual to $\{ e_a \}$.
Since $\nabla^E$ does not contain the derivatives in the Novikov variables, we call it a partial connection.
We define the \emph{analytic equivariant quantum $D$-module} $\anQDMT (E)$ as the pair of this bundle and $\nabla^E$. 
We also define the (formal) equivariant quantum $D$-module $\QDMT(E)$ by restricting $\anQDMT (E)$ to the formal neighborhood of the origin of $H^*(E) \times \Spec \C[q,Q]$.
See Section \ref{subsect:QDM} for details\footnote{In Section \ref{subsect:QDM}, we include equivariant parameters in the base space of $\QDMT(E)$. In addition, we endow the quantum $D$-module with a pairing and a grading structure.}.
In the same manner, we can define $\anQDMT(B)$ and $\QDMT(B)$. 
(Here we consider the trivial $T$-action on $B$.)
After localizing with respect to $\lambda$, we can construct the following formal decomposition.

\begin{theorem}[Theorem \ref{thm:formaldecomp}]
\label{thm:intro_formal}
There exist a formal change of variables $\chtau_\alpha ( \lambda, \hsigma )$ and a matrix-valued formal power series $\chbfR ( \lambda, \hsigma, z ) = \bigoplus_{\alpha \in F} \chbfR_\alpha$ of the form
\begin{align*}
&\chtau_\alpha - \chtau^\exc_\alpha \in H^*(B) \otimes \C( \lambda ) [\![ q, Q, \hsigma ]\!], & &\chtau_\alpha |_{(q, Q) = 0} = \alpha^* \hsigma + \chtau^\exc_\alpha, \\
&\chbfR_\alpha \in \frac{1}{ \sqrt{e_\alpha} } \cdot \Hom ( H^*(E), H^*(B) ) \otimes \C(\lambda) [\![ q, Q, \hsigma, z ]\!], & 
&\chbfR_\alpha |_{(q,Q,z)=0} (\omega) = \frac{ \alpha^* \omega }{ \sqrt{e_\alpha} }
\end{align*}
which give an isomorphism
\begin{equation}
\label{eq:intro_isom}
\chbfR \colon \QDMT(E) \cong \bigoplus_{\alpha \in F} \chtau_\alpha^* \QDMT(B).
\end{equation}
Here $\chtau^\exc_\alpha$ is a certain cohomology class of $H^*(B)$ with coefficients holomorphic functions in $\lambda$ and $e_\alpha \in H^*_T(B)$ is the equivariant Euler class of the normal bundle of $\alpha$ in $X$.
\end{theorem}

\begin{remark}
\label{rem:CGT}
Coates, Givental and Tseng have shown a decomposition of the fundamental solution of $\nabla^E$ \cite[Proposition 3.1]{Coates-Givental-Tseng}. 
Theorem \ref{thm:intro_formal} can be said to be a reformulation of their decomposition in terms of quantum $D$-modules.
Note that they directly construct the decomposition by the virtual localization formula, while we construct the decomposition via an intermediate connection which is obtained by Brown's $I$-function.
We will also confirm that the decomposition $\chbfR$ is compatible with additional structures of the quantum $D$-modules, pairing and grading (see Section \ref{subsect:QDM} for definition). 
\end{remark}

Theorem \ref{thm:main1} (2) is an analytic refinement of the formal decomposition \eqref{eq:intro_isom}.
We remark that the convergence domains of $\chtau_\alpha$ and $\chbfR$ are not so simple as in Theorem \ref{thm:main1} (1).
Let $\cU \subset \Spec \C [ \lambda, q ]$ and $\cV \subset \Spec \C [ Q,\hsigma ]$ be analytic open neighborhoods of the origin such that $QH^*_T(E)$ converges on $\cU \times \cV$.
Let $\oanQDMT (-)$ be the completion of $\anQDMT (-)$ with respect to $z$.
We prove the following:

\begin{theorem}[Theorem \ref{thm:decompQDM}]
\label{thm:main2}
There exists an analytic hypersurface $(0 \in) \cD \subset \cU$ such that the isomorphism \eqref{eq:intro_isom} can be analytically continued to a multi-valued analytic isomorphism over an open neighborhood $\cW$ of $(\cU \setminus \cD) \times \{0\}$ in $(\cU \setminus \cD) \times \cV$ $:$
\[
\chbfR \colon \oanQDMT(E) \cong \bigoplus_{\alpha \in F} \chtau_\alpha^* \oanQDMT(B)
\]
Here the hypersurface $\cD$ does not entirely contain the space $( \lambda = 0)$, and hence we can take the non-equivariant limit of this isomorphism $:$
\[
\chbfR |_{\lambda = 0} \colon \oanQDM (E) \cong \bigoplus_{\alpha \in F} (\chtau_\alpha |_{\lambda = 0})^* \oanQDM (B)
\]
\end{theorem}

\begin{remark}
As in the formal case, the decomposition $\chbfR$ preserves the additional structures of analytic quantum $D$-modules.
In particular, in the non-equivariant setting, $\chbfR |_{\lambda = 0}$ also preserve the connection in $z$-direction.
\end{remark}

The entries of the matrix $\chbfR$ in this theorem are in general divergent formal power series in $z$ with coefficients multi-valued analytic functions in $(\lambda,q,Q,\hsigma)$. 
This is why we take the $z$-adic completion of the quantum $D$-modules. 
As a corollary, we also obtain a decomposition for equivariant/non-equivariant quantum cohomology, which is given by the Jacobian matrices $J_{\chtau_\alpha}$ of $\chtau_\alpha$.
Let $J_{\chtau} := \bigoplus_{\alpha \in F} J_{\chtau_\alpha}$.

\begin{corollary}[Corollary \ref{cor:decompQH}]
{\rm (1)}
The map $J_{\chtau}$ gives an analytic decomposition of quantum cohomology algebra $QH^*_T(E)_{\lambda,\hsigma}$ into a direct sum $\bigoplus_{\alpha \in F} QH^*_T(B)_{\lambda,\chtau_\alpha}$.

{\rm (2)}
The map $\chtau|_{\lambda=0} \colon H^*(E) \to \bigoplus_{\alpha \in F} H^*(B)$ gives a local isomorphism of quantum cohomology $F$-manifolds with Euler vector fields when we specialize all Novikov variables to $1$.
\end{corollary}

When we take the classical limit $(q,Q)=0$, the decomposition of $QH^*(E)$ coincides with the localization isomorphism for (classical) equivariant cohomology.
Therefore, the map $J_{\chtau}$ can be viewed as a generalization of the localization isomorphism. 

This paper is organized as follows.
In Section 2, we briefly review the genus-zero Gromov-Witten theory and related concepts including (equivariant) quantum cohomology, (equivariant) quantum $D$-module and Givental's Lagrangian cone.
In Section 3, we review Brown's work \cite{Brown}, which introduce the $I$-function for $E$ and relate it to the $J$-function for $B$ via the Landau-Ginzburg mirror of the fiber $X$.
In Section 4, we prove the formal decomposition theorem of $\QDMT(E)$ discussed above.
We also introduce an intermediate $D$-module between $\QDMT(E)$ and $\bigoplus_{\alpha \in F} \QDMT(B)$ by using Brown's $I$-function, and split the problem into two steps.
We deal with these two steps in Section 5 and Section 6 respectively.
The proof of Theorem \ref{thm:main1} (1) is in Section 5, and that of Theorem \ref{thm:main1} (2) (Theorem \ref{thm:main2}) is in Section 6.

\bigskip

\noindent
{\bf Acknowledgments.}
The author gratefully acknowledges many helpful suggestions of Professor Hiroshi Iritani during the preparation of the paper.
This work was supported by JSPS KAKENHI Grant Number JP21J23023.

\section{Genus-Zero Gromov-Witten Theory}

In this section, we review non-equivariant/equivariant quantum cohomology, quantum $D$-module and related objects without proofs. 

\subsection{Quantum Cohomology}

Let $M$ be a smooth projective variety, and $\overline{\cM}_{0,n}(M,d)$ be the moduli space of stable maps from $n$-pointed genus-zero nodal curves to $M$ representing the class $d\in H_2(M,\Z)$.
The moduli space $\overline{\cM}_{0,n}(M,d)$ is constructed as a Deligne-Mumford stack, and is endowed with the evaluation maps $\ev_i \colon \overline{\cM}_{0,n}(M,d) \to M$ $(1 \leq i \leq n)$. 
There exist natural tautological line bundles $\cL_i$ $(1 \leq i \leq n)$ on $\overline{\cM}_{0,n}(M,d)$, whose fiber at a point $[ f \colon (\Sigma,x_1, \dots, x_n) \to M ]$ is given by the cotangent line to $\Sigma$ at $x_i$.
Let $\psi_i = c_1 (\cL_i)$.
For $\alpha_1, \dots, \alpha_n \in H^*(M)$ and $k_1, \dots, k_n \in \Z_{\geq 0}$, we define \emph{genus-zero descendant Gromov-Witten invariants} of $M$ as follows:
\[
\corr{\alpha_1\psi^{k_1},\dots,\alpha_n\psi^{k_n}}^M_{0,n,d}:=\int_{[\overline{\cM}_{0,n}(M,d)]^{\virt}}\prod^n_{i=1}\ev^*_i(\alpha_i)\psi^{k_i}_i
\]
where $[\overline{\cM}_{0,n}(M,d)]^{\virt}$ is the virtual fundamental class.

Choose a homogeneous basis $\phi_0 = 1, \phi_1, \dots, \phi_r, \phi_{r+1}, \dots, \phi_s$ of $H^*(M)$ such that $\phi_1, \dots, \phi_r$ are nef classes in $H^2(M)$ and generate $H^2(M)$.
Let $\tau_0, \dots, \tau_s$ be linear coordinates of $H^*(M)$ dual to the basis $\{ \phi_j \}^s_{j=0}$, and let $Q_1, \dots, Q_r$ be the Novikov variables dual to $\{ \phi_j \}^r_{j=1}$.

Let $\corr{\cdot,\cdot}^M$ be the Poincar\'{e} pairing on $H^*(M)$:
\[
\corr{\alpha,\beta}^M := \int_M\alpha\cup\beta.
\]
We write $Q^D := \prod^r_{j=1} Q^{ \corr{\phi_j,D} }_j$ for $D \in H_2(M,\Z)$.
Then Gromov-Witten invariants define the \emph{quantum product}:
\[
\corr{\alpha\star_M\beta,\gamma}^M := \sum_{n\geq0} \sum_{D \in \Eff(M)} \frac{Q^D}{n!} \corr{\alpha,\beta,\gamma,\tau,\dots,\tau}^M_{0,n+3,D} \in \C[\![Q ]\!] [\![ \tau]\!]
\]
Using the string equation and the divisor equation, we can rewrite this formula as
\[
\corr{\alpha\star_M\beta,\gamma}^M = \sum_{n\geq0} \sum_{D \in \Eff(M)} \frac{ Q^D e^{ \corr{\tau,D} } }{n!} \corr{\alpha,\beta,\gamma,\tau',\dots,\tau'}^M_{0,n+3,D}
\]
where $\tau' = \sum^s_{j = r+1} \tau_j \phi_j$.
This is a redundant parametrization in the sense that $Q_j$ always appears with $e^{\tau_j}$. 
To avoid the redundancy, we also use the following modified variables:
\[
x = (x_0, x_1, \dots, x_r, x_{r+1}, \dots, x_s) := (\tau_0, Q_1 e^{\tau_1}, \dots, Q_r e^{\tau_r}, \tau_{r+1}, \dots, \tau_s).
\]
Writing in these variables, we can see that $\corr{\alpha\star_M\beta,\gamma}^M \in \C [\![x]\!]$.
We define the \emph{quantum cohomology} of $M$ as follows:
\[
QH^*(M) := ( H^*(M)\otimes\C[\![x]\!], \star_M ).
\]
This is an associative and commutative graded ring with extra grading
\[
\deg x_j =
\begin{cases}
\corr{ c_1(M), \phi^j }^M		&1 \leq j \leq r	\\
1 - \frac{1}{2} \deg \phi_j	& \text{otherwise}
\end{cases}
\]
where $\{ \phi^j \}^s_{j=0}$ is the basis of $H^*(M)$ dual to $\{ \phi_j \}^s_{j=0}$ with respect to the Poincar\'{e} pairing of $M$.

We can generalize these constructions to the equivariant setting \cite{Givental-equivariant}.
Let $T_\C = (\C^\times)^N$ be an algebraic torus and $T = (S^1)^N$ be the maximal compact subgroup of $T_\C$.
Suppose that $M$ is endowed with an algebraic $T_\C$-action.
Then the equivariant cohomology $H^*_T(M)$ is a free $H^*_T(\pt)$-module \cite{Goresky} and there exists a non-canonical isomorphism $H^*_T(M) \cong H^*(M) \otimes H^*_T(\pt)$.
The torus $T$ naturally acts on the moduli space $\overline{\cM}_{0,n}(M,d)$, and we can define the \emph{equivariant genus-zero Gromov-Witten invariants} by using the equivariant integrals; for $\alpha_1, \dots, \alpha_n \in H^*(M)$ and $k_1, \dots, k_n \in \Z_{\geq 0}$, define
\[
\corr{\alpha_1\psi^{k_1},\dots,\alpha_n\psi^{k_n}}^{T,M}_{0,n,d}:=\int^T_{[\overline{\cM}_{0,n}(M,d)]^{\virt}}\prod^n_{i=1}\ev^*_i(\alpha_i)\psi^{k_i}_i \in H^*_T(\pt)
\]
Note that these invariants can be computed by the virtual localization formula.

Let $\corr{\cdot,\cdot}^M_T$ be the equivariant Poincar\'{e} pairing on $H^*_T(M)$.
(We will omit the subscript $T$ when no confusion is likely.)
We take a $H^*_T(\pt)$-basis of $H^*_T(M)$ by choosing a homogeneous lift of $\{ \phi_j \}^s_{j=0}$, and denote it by the same symbol.
As in the non-equivariant case, we define $H^*_T(\pt)$-linear coordinates $\tau_0, \dots, \tau_s$ on $H^*_T(M)$ and introduce the modified variables $x$ by the same formula.
We define the \emph{equivariant quantum product} as
\[
\corr{\alpha\star_M\beta,\gamma}^M_T := \sum_{n\geq0} \sum_{D \in \Eff(M)} \frac{Q^D}{n!} \corr{\alpha,\beta,\gamma,\tau,\dots,\tau}^{T,M}_{0,n+3,D} \in H^*_T(\pt) [\![x]\!],
\]
and define the \emph{equivariant quantum cohomology} of $M$ as
\[
QH^*_T(M) := (H^*_T(M) [\![ x]\!], \star_M).
\]
Note that $QH^*_T(M)$ again becomes an associative and commutative graded ring with the same grading.

\subsection{Quantum D-modules}
\label{subsect:QDM}

Let us consider the equivariant situation.
Since $H^*_T(M)$ is a free $H^*_T(\pt)$-module, we can regard it as a trivial vector bundle $\bH_T(M)$ over $\Spec H^*_T(\pt)$ of rank $\dim H^*(M)$.
We consider the vector bundle:
\begin{equation}
\label{eq:bundle}
\bH_T(M) \times_{\Spec H^*_T(\pt)} \Spec ( H^*_T(\pt) [x,z] ) \to \Spec ( H^*_T(\pt) [x,z] )
\end{equation}
The equivariant quantum product defines a meromorphic flat partial connection on this bundle restricted to $\Spec ( H^*_T(\pt) [z] [\![x]\!] )$, which is called \emph{equivariant quantum connection}, as follows: 
\begin{align*}
\nabla^M :=& d + z^{-1} \sum^s_{j=0} (\phi_j \star_M) d \tau_j	\\
=& d + z^{-1} d x_0 + z^{-1} \sum^r_{j = 1} (\phi_j \star_M) \frac{d x_j}{x_j} + z^{-1} \sum^s_{j = r+1} (\phi_j \star_M) d x_j
\end{align*}
where the first term $d$ is the relative differential over $\Spec H^*_T(\pt) [z]$.
The term "partial" implies that $\nabla^M$ does not contain the derivatives with respect to $z$ or the equivariant parameters. 
Let us translate the connection $\nabla^M$ as the operators acting on the module $H^*_T(M) [z] [\![ x ]\!]$, which is the module of the sections of the bundle \eqref{eq:bundle} restricted to $\Spec ( H^*_T(\pt) [z] [\![ x ]\!] )$.
We can regard $\nabla^M_{ z \parfrac{}{\tau_j} }$ ($0 \leq j \leq s$) as the operator on this module:
\[
\nabla^M_{ z \parfrac{}{\tau_j} } \colon H^*_T(M) [z] [\![ x ]\!] \to H^*_T(M) [z] [\![ x ]\!] 
\]

We introduce the \emph{Euler vector field} $\cE^M$: 
\begin{equation}
\label{eq:Eulervf}
\cE^M := C_0(\lambda) \parfrac{}{\tau_0} + \sum^r_{j=1} C_j \parfrac{}{\tau_j} + \sum^s_{j=0} \left( 1 - \frac{1}{2} \deg \phi_j \right) \tau_j \parfrac{}{\tau_j} + \sum^N_{i=1} \lambda_i \parfrac{}{\lambda_i}
\end{equation}
where $\lambda_1,\dots,\lambda_N$ are $T$-equivariant parameters and $C_0(\lambda), C_1, \dots, C_r$ are the coefficients of the $T$-equivariant first Chern class of $M$:
\[
c^T_1(M) = C_0(\lambda) \phi_0 + \sum^r_{j=1} C_j  \phi_j.
\]
This is a vector field on $\bH_T(M)$.
The \emph{grading operator} $\Gr^M \in \End_\C ( H^*_T(M) [z] [\![ x ]\!] )$ is defined by
\begin{equation}
\label{eq:Grading}
\Gr^M := \cE^M + z \parfrac{}{z} + \Gr^M_0 - \frac{1}{2} \dim M
\end{equation}
where $\Gr^M_0$ is the $H^*_T(\pt) [z] [\![x]\!]$-linear operator given by
\[
\Gr^M_0 (\phi_j) = \left( \frac{1}{2} \deg \phi_j \right) \phi_j.
\]

\begin{remark}
{\rm (1)}
Some literature uses a different Euler vector field from $\cE^M$.
In \cite{Coates-Corti-Iritani-Tseng-hodge}, for example, Euler vector field is defined as \eqref{eq:Eulervf} without the first term, and this makes a subtle difference in the grading structures.
For instance, $\tau_0$ is homogeneous with respect to the grading in \cite{Coates-Corti-Iritani-Tseng-hodge}, whereas $\tau_0 - C_0(\lambda) \log C_0(\lambda)$ is homogeneous in our situation.

{\rm (2)}
The grading operator here is also different from an ordinary one.
In many papers, it is defined as \eqref{eq:Grading} without the last term.
Since we need to compare grading operators for manifolds of different dimensions, it turns out that it is natural to subtract the term $\frac{1}{2} \dim M$.
\end{remark}

Finally, we define the $H^*_T(\pt) [\![ x ]\!]$-bilinear pairing on $H^*_T(M) [z] [\![ x ]\!]$, which is denoted by $P^M$, induced by the Poincar\'{e} paring:
\[
P^M( \alpha(z), \beta(z) ) := \corr{\alpha(-z),\beta(z)}^M \in H^*_T(\pt) [z] [\![ x ]\!].
\]
The pairing $P^M$ satisfies the following equations:
\begin{align*}
\parfrac{}{\tau_j} P^M( \alpha(z), \beta(z) ) &= P^M\left( \left(\nabla^M |_{z \to -z}\right)_{\parfrac{}{\tau_j}} \alpha(z), \beta(z) \right) + P^M\left( \alpha(z), \nabla^M_{\parfrac{}{\tau_j}} \beta(z) \right)	\\
\left( \cE^M + z \parfrac{}{z} \right) P^M( \alpha(z), \beta(z) ) &= P^M\left( \Gr^M \alpha(z), \beta(z) \right) + P^M\left( \alpha(z), \Gr^M \beta(z) \right) 
\end{align*}

We define the \emph{equivariant quantum $D$-module} of $M$ as the quadruple
\[
\QDMT(M) := ( H^*_T(M) [z] [\![ x ]\!], \nabla^M, \Gr^M, P^M ).
\]
By taking the non-equivariant limit, we also define the (non-equivariant) \emph{quantum $D$-module} of $M$ as
\[
\QDM(M) := ( H^*(M) [z] [\![ x ]\!], \nabla^M, \Gr^M, P^M ).
\]

In the non-equivariant case, we can extend the connection $\nabla^M$ to the $z$-direction.
By abuse of notation, we denote by $\cE^M$ the non-equivariant limit of \eqref{eq:Eulervf}, and let $E^M$ be the element of $\QDM(M)$ corresponding to $\cE^M$:
\[
E^M = c_1(M) + \sum^s_{j=0} \left( 1 - \frac{1}{2} \deg \phi_j \right) \tau_j \phi_j.
\]
We define the connection in the $z$-direction as
\[
\nabla_{z \parfrac{}{z}} := z \parfrac{}{z} - \frac{1}{z} ( E^M \star_M ) + \Gr^M_0 - \frac{1}{2} \dim M = \Gr^M - \nabla_{\cE^M}.
\]

We consider a fundamental solution $\bfL_M \in \End_{H^*_T(\pt)} (H^*_T(M)) [\![ Q ]\!] [\![ \tau, z^{-1} ]\!]$ of the connection $\nabla^M$, which is uniquely determined by the following conditions \cite{Givental-elliptic}:
\begin{equation}
\label{eq:fund_sol}
\bfL_M (\tau,z) \circ \nabla^M_{\parfrac{}{\tau_j}} = \parfrac{}{\tau_j} \circ \bfL_M(\tau,z) \qquad 0 \leq j \leq s, 
\end{equation}
\[
e^{-\tau/z} \bfL_M \in \End_{H^*_T(\pt)} (H^*_T(M)) [\![ x, z^{-1} ]\!], \qquad (e^{-\tau/z} \bfL_M) |_{Q = 0} = \Id.
\]
This fundamental solution can be given explicitly by Gromov-Witten invariants: 
\begin{align*}
\bfL_M(\tau,z) \phi_i 
:=& \, \phi_i + \sum_{n,D,j} \frac{Q^D}{n!} \corr{\phi_i, \frac{\phi^j}{z - \psi}, \tau, \dots, \tau}^{T,M}_{0,n+2,D} \phi_j \\
=& \, e^{\tau/z} \phi_i + \sum_{D \neq 0} \sum_{n,j} \frac{Q^D e^{ \corr{\tau,D} } }{n!} \corr{\phi_i, \frac{\phi^j}{z - \psi}, \tau', \dots, \tau'}^{T,M}_{0,n+2,D} e^{ (\tau - \tau')/z } \phi_j
\end{align*}
The second equality follows from the string  equation and the divisor equation.
The matrix $\bfL_M$ satisfies the following unitarity condition \cite{Givental-elliptic}.
\begin{proposition}
\label{prop:unitarity}
\[
\pair{\bfL_M(\tau,-z)\alpha}{\bfL_M(\tau,z)\beta}^M = \pair{\alpha}{\beta}^M	\qquad	\alpha,\beta \in H^*_T(M)
\]
\end{proposition}

We define the \emph{$J$-function} of $M$ as:
\begin{align*}
J_M (\tau,z) &= z \bfL_M(\tau,z) \phi_0	\\
&=z+\tau+\sum_{n,D,j} \frac{Q^D}{n!} \corr{\frac{\phi^j}{z-\psi},\tau,\dots,\tau}^{T,M}_{0,n+1,D} \phi_j
\end{align*}
From \eqref{eq:fund_sol}, we have
\[
\bfL_M =
\begin{pmatrix}
|				&		&|				\\
\parfrac{}{\tau_0}J_M	&\cdots	&\parfrac{}{\tau_s}J_M	\\
|				&		&|				\\
\end{pmatrix}
\]

We end this section by introducing an analytic version of $\QDMT(M)$.
Recall that $\lambda = (\lambda_1,\dots,\lambda_N)$ are the $T$-equivariant parameters.
Suppose that the structure constants of $\star_M$ are convergent and analytic over an (analytic) open neighborhood $U$ of the origin in $\Spec \C[\lambda,x]$.
Then $\QDMT(M)$ can be extended to $U \times \Spec \C[z]$.
We define the \emph{analytic equivariant quantum $D$-module} as
\[
\anQDMT(M) := ( H^*_T(M) \otimes \cO^{\mathrm{an}}_U[z], \nabla^M, \Gr^M, P^M ).
\]
Here we regard $H^*_T(M) \otimes \cO^{\mathrm{an}}_U[z]$ as a sheaf of modules on $U$.
We denote by $\oanQDMT(M)$ the $z$-adic completion of $\anQDMT(M)$:
\[
\oanQDMT(M) := ( H^*_T(M) \otimes \cO^{\mathrm{an}}_U[\![z]\!], \nabla^M, \Gr^M, P^M ).
\]
The non-equivariant version of these objects, denoted by $\anQDM(M)$ and $\oanQDM(M)$ respectively, are defined as their non-equivariant limits.

Although we do not give a strict definition of the categories in which each quantum $D$-module lives, we will consider direct sums and pullbacks of them in Section 4 and later.
If there is a morphism $f \colon \cB \to \Spec( H^*_T(\pt)[\![x]\!] )$, we can naturally pullback $\QDM_T(M)$ along this morphism except for the grading structure.
In order to define a pullback of $\Gr^M$, let $\cB$ be endowed with a vector field $\cE^\cB$ and assume that $f_* \cE^\cB = \cE^M$ as global sections of $f^*\Theta$, where $\Theta$ is the tangent sheaf of $\Spec( H^*_T(\pt)[\![x]\!] )$.
Then we can define the pullback of $\Gr^M$ by
\[
f^*\Gr^M = \cE^\cB + z\parfrac{}{z} + \Gr^M_0.
\]
It is easy to see that $f^* \circ \Gr^M = f^*\Gr^M \circ f^*$.
We can also define a direct sum of $D$-modules over a fixed base space with a fixed Euler vector field.

\subsection{Givental Lagrangian Cone}

We describe a symplecto-geometric interpretation of genus-zero Gromov-Witten theory introduced by Givental \cite{Givental-quantization, Givental-symplectic}.
We will follow the formalism introdued in \cite{Coates-Corti-Iritani-Tseng-hodge}.
Let $\cH$ be the vector space
\[
\cH := H^*_T(M) \otimes_{H^*_T(\pt)} \Frac (H^*_T(\pt)) (\!( z^{-1} )\!) [\![ Q ]\!].
\]
Define the symplectic form $\Omega$ on $\cH$ as
\[
\Omega (f,g) = - \Res_{z = \infty} \pair{f(-z)}{g(z)}^M dz.
\]
We consider the following decomposition $\cH = \cH_+ \oplus \cH_-$:
\[
\cH_+ = H^*_T(M) \otimes_{H^*_T(\pt)} \Frac (H^*_T(\pt)) [z] [\![ Q ]\!], \quad \cH_- = z^{-1} \cdot H^*_T(M) \otimes_{H^*_T(\pt)} \Frac (H^*_T(\pt)) [\![ z^{-1} ]\!] [\![ Q ]\!]
\]
Since these are isotropic subspaces, we can identify $\cH$ with the cotangent bundle $T^* \cH_+$.
We define the equivariant Givental cone $\cL_M \subset \cH$ as the graph of the differential of the genus-zero descendant potential $\cF$ defined over the formal neighborhood of $-z \in \cH_+$.
Let $t_1, \dots, t_N$ be formal variables.
A $\Frac(H^*_T(\pt)) [\![ Q ]\!] [\![ t_1,\dots,t_N ]\!]$-valued point on $\cL_M$ is a point of the form:
\begin{equation}
\label{eq:point_cL}
- z + \bt(z) + \sum_{n,D,j} \frac{Q^D}{n!} \corr{ \frac{\phi^j}{-z-\psi}, \bt(\psi), \dots, \bt(\psi) }^{T,M}_{0,n+1,D} \phi_j
\end{equation}
where $\bt(z) \in \cH_+ [\![ t_1, \dots, t_N ]\!]$ with $\bt(z) |_{(Q,t)=(0,0)} = 0$.
For instance, the $J$-function $J_M(\tau,-z)$ (with sign of $z$ flipped) is a $\Frac(H^*_T(\pt)) [\![ Q ]\!] [\![ \tau ]\!]$-valued point on $\cL_M$.
We will refer to $\cL_M$ as Lagrangian cone, and will write, for example, $J_M (\tau,-z)$ lies in $\cL_M$.

The potential $\cF_M$ satisfies certain partial differential equations called the topological recursion relations, the string equation and dilaton equation.
These equations lead to special geometric properties of $\cL_M$ \cite{Givental-symplectic}.
In fact, $\cL_M$ becomes a over-ruled Lagrangian cone, that is, each tangent space $L$ of $\cL_M$ is tangent to $\cL_M$ along $\cL_M \cap L = zL$.
In particular, $\cL_M$ satisfies the following properties.

\begin{lemma}[{\cite{Givental-symplectic}}]
\label{lem:cL}
The tangent space $L$ of $\cL_M$ at an $\Frac(H^*_T(\pt)) [\![ Q ]\!] [\![ t ]\!]$-valued point is freely generated over $\Frac(H^*_T(\pt)) [z] [\![ Q ]\!] [\![ t ]\!]$ by the derivatives of the $J$-function $J_M(\tau,-z)$ for some $\tau \in H^*_T(M) \otimes_{H^*_T(\pt)} \Frac(H^*_T(\pt)) [\![ Q ]\!] [\![ t ]\!]$ with $\tau |_{(Q,t) = 0}$.
\end{lemma}

\section{Toric Bundles}
\label{sect:toricbundles}

In this section we introduce toric bundles and summarize the result of Brown \cite{Brown}.
We refer to \cite{Audin} for toric geometry. 

\subsection{Construction}
\label{subsect:construction}
We first review the basic material on projective toric manifolds as found in \cite{Audin}. 
Then we explain the construction of toric bundles and introduce some important equivariant chomology classes on them introduced by Brown \cite{Brown}.

Let $X$ be a smooth projective toric manifold. 
It is known that $X$ can be obtained as the Geometric Invariant Theory (GIT) quotient of a vector space $\C^N$ by a torus $K_\C=(\C^\times)^k$, where $K_\C$ acts on $\C^N$ via an injective homomorphism $K_\C \to T_\C=(\C^\times)^N$.
That is to say, there exists an open dense subset $\cU_X \subset \C^N$, defined as the complement of a union of certain coodinate subspaces, on which $K_\C$ acts freely such that $X=\cU_X/K_\C$. 
We may assume that $\cU_X$ is 2-connected (i.e. the complement of $\cU_X$ does not contain a hyperplane) by changing the GIT data if necessary.
We denote by $K=(S^1)^k$ and $T=(S^1)^N$ the maximal compact subgroups of $K_\C$ and $T_\C$ respectively.
The tori $T_\C$ and $T$ naturally act on $X$.

Under the assumption that $\cU_X$ is 2-connected, we have an identification of $H^2(X,\R)$ with $\Lie^*(K)$ which induces an isomorphism between 
$H^2(X,\Z)$ and the integral character lattice $\Lie^*_\Z(K)$ of $\Lie^*(K)$. 
This identification is given as follows: 
\begin{equation}
\label{eq:identification}
\Lie^*_\Z(K) \xrightarrow{\cong} H^2(X,\Z) \qquad m \mapsto c_1(\bL_m)
\end{equation}
where $\bL_m$ is the line bundle constructed as the quotient $(\cU_X\times\C)/K_\C\to\cU_X/K_\C=X$ of the space $\cU_X\times\C$ on which $K_\C$ acts by $t\cdot(z,u)=(t\cdot z,m(t)^{-1}\cdot u)$. 
We sometimes regard a class in $H^2(X)$ as an element of $\Lie^*(K)$, and vice versa.

We endow $\bL_m$ with the $T$-action $s \cdot [z,u] = [s \cdot z, u]$.
Let $-p_1,\dots,-p_k \in H^2_T (X)$ denote the $T$-equivariant first Chern classes of the $T$-line bundles $\bL_{e_1},\dots,\bL_{e_k}$ associated with the standard basis of $\Z^k = \Lie^*_\Z (K)$.
We define the integer $k\times N$ matrix $\bc=(c_{ij})$ to be the matrix representing the projection $\R^N=\Lie^*(T)\to\R^k=\Lie^*(K)$.
Then the $T$-equivariant Poincar\'e dual $u_j$ of the toric divisor $\{ [z_1,\dots,z_N] \mid z_j=0 \} \subset X$ is given by
\[
u_j=\sum^k_{i=1}c_{ij}p_i-\lambda_j\qquad 1\leq j\leq N
\]
where $\lambda_j$ represents the $j$-th equivariant parameter of $H^*_T(\operatorname{pt})=\C[\lambda_1,\dots,\lambda_N]$.
The equivariant first Chern class of $X$ is given by $c^T_1(X) = u_1 + \dots + u_N$. 

Let $F$ be the set of fixed points for the torus action on $X$.
Each fixed point $[z_1,\dots,z_N] \in F$ determines a $k$-element subset $\alpha = \{ i \mid z_i \neq 0 \}$ of $\{ 1,\dots,N \}$.
By this correspondence, we identify an element of $F$ with the corresponding $k$-element subset $\alpha$. 

Let $\cN$ be the cone in $\Lie^*(K)$ corresponding to the nef cone of $X$ via the identification \eqref{eq:identification}. 
We can describe $\cN$ in terms of $\bc$ and $F$ as 
\begin{align*}
\cN&=\bigcap_{\alpha\in F}\bc_\alpha(\R^k_+)
\end{align*}
where $\bc_\alpha$ is the $k\times k$ submatrix $(c_{ij}\,|\,1\leq i\leq k, j\in\alpha)$ of $\bc$. 
Changing the identification of $K$ with $(S^1)^k$ if necessary, we may assume that the classes $p_i$ are nef. 
This condition can be rephrased in the following way.

\begin{lemma}
\label{lem:nefness}
The classes $p_1,\dots,p_k$ are nef if and only if all entries of $\bc^{-1}_\alpha$ are non-negative for all $\alpha\in F$. 
\end{lemma}

\begin{proof}
The cone $\cN$ contains the class $p_i$ if and only if $\bc_\alpha(\R^k_+)$ does for all $\alpha\in F$.
The latter conditon is equivalent to $\bc^{-1}_\alpha(p_i)$ being contained in $\R^k_+$ for all $\alpha\in F$.
Since $\bc^{-1}_\alpha(p_i)$ is the $i$-th column of the matrix $\bc^{-1}_\alpha$, this proves the equivalence. 
\end{proof}

Let $B$ be a smooth projective manifold and $L_1,\dots,L_N$ be line bundles over $B$. 
The torus $T_\C$ naturally acts fiberwise on the vector bundle $\bigoplus^N_{i=1} L_i$. 
Performing a fiberwise GIT quotient by the torus $K_\C$, we obtain a \emph{toric bundle} $E\rightarrow B$ with fiber $X$. 
More concretely, there exists an open dense subset $\cU_E$ of $\bigoplus^N_{i=1} L_i$, whose fibers can be identified with $\cU_X$, such that $E=\cU_E/K_\C$.
Note that $E$ is endowed with a fiberwise action of $T$. 

We can construct $T$-line bundles over $E$ from elements of $\Lie^*_\Z(K)$ as in the case of the toric manifold $X$.
For $m \in \Lie^*_\Z (K)$, we define the line bundle $\bL^E_m$ to be the quotient $(\cU_E \times \C) / K_\C \to \cU_E / K_\C = E$, with $K_\C$-action on $\cU_E \times \C$ given by $t \cdot (z,u) = (t \cdot z, m(t)^{-1} \cdot u)$.
We endow this line bundle with the $T$-action $s \cdot [z,u] = [s \cdot z, u]$. 

We denote by $-P_1,\dots ,-P_k\in H^2_T(E)$ the $T$-equivariant first Chern classes of the $T$-line bundles $\bL^E_{e_1},\dots,\bL^E_{e_k}$ corresponding to the standard basis of $\Z^k = \Lie^*_\Z (K)$.
They restrict to $-p_1,\dots ,-p_k$ on the fibers. 
Define
\[
\Lambda_j := c_1 (L^\vee_j) \in H^2(B), \quad U_j := \sum^k_{i=1} c_{ij} P_i - \Lambda_j - \lambda_j \in H^2_T(E) \qquad 1 \leq j \leq N.
\]
We remark that the $T$-equivariant first Chern class of $L^\vee_j$ is $\Lambda_j+\lambda_j$, and that $U_j$ is the $T$-equivariant first Chern class of the $T$-line bundle $( \cU_E \times_B L_j ) / K_\C \to \cU_E / K_\C$.
In other words, $U_j$ is the $T$-equivariant Poincar\'e dual of the divisor $( \cU_E \cap \bigoplus_{i\neq j} L_i ) / K_\C$ of $E$.
This divisor restricts to a toric divisor of the fibers, and hence we call it a \emph{toric divisor} of $E$.

From the construction each fixed locus is isomorphic to the base $B$ and is labeled by $\alpha\in F$, a fixed point of the fiber $X$. 
Let $\alpha\colon B\to E$ denote the embeddings of the fixed component labeled by $\alpha$ and $B^\alpha$ be the image.
We will describe the restrictions $P^\alpha_i:=\alpha^*P_i$ and $\aUj$ of the classes $P_i$ and $U_j$ to the fixed locus $B^\alpha$. 
Since $B^\alpha$ is disjoint from the $j$-th toric divisor for $j\in\alpha$, it follows that:
\[
\aUj=\sum^k_{i=1}c_{ij}P^\alpha_i-\Lambda_j-\lambda_j=0\qquad j\in\alpha.
\]
These equations determine $P^\alpha_i$ for $1\leq i\leq k$ in terms of $\Lambda_j$ and $\lambda_j$, and thus $\aUj$ for $j \notin \alpha$. 
A straightforward calculation shows the following:

\begin{lemma}
\label{lem:computePU}
\begin{alignat*}{2}
P^\alpha_i&=\sum_{n\in\alpha}(\bc^{-1}_\alpha)_{ni}(\Lambda_n+\lambda_n)&&1\leq i\leq k,\\
\aUj&=\sum_{n\in\alpha}(\bc^{-1}_\alpha\bc)_{nj}(\Lambda_n+\lambda_n)-\Lambda_j-\lambda_j&\qquad&j\notin\alpha.
\intertext{In particular, when $B$ is a point, }
p^\alpha_i:=\alpha^*p_i&=\sum_{n\in\alpha}(\bc^{-1}_\alpha)_{ni}\lambda_n&&1\leq i\leq k,\\
\auj&=\sum_{n\in\alpha}(\bc^{-1}_\alpha\bc)_{nj}\lambda_n-\lambda_j&&j\notin\alpha. 
\end{alignat*}
\end{lemma}

We shall use the classes $P_i$ to measure the degree of curves.
For this purpose, it is convenient to choose a GIT presentation for $E$ such that $P_i$'s are nef.
The set of classes of effective curves on $E$ is given by
\[
\Eff (E) = \sum_{\alpha \in F} \alpha_* \Eff (B) + \Eff (X).
\] 
We have already assumed that $P_i |_X = p_i$ is nef, and therefore $P_i$ is nef if and only if $P_i$ pairs non-negatively with curves in $B^\alpha$, or equivalently, $P^\alpha_i$ is nef for all $\alpha \in F$.

It is easy to see that, for any line bundle $A$ over the base $B$ and for any $x \in \Z^k = \Lie_\Z (K)$, the different set of line bundles over $B$
\[
L'_i = L_i \otimes A^{\otimes \left( \bc^* (x)_i \right)} \qquad 1 \leq i \leq N
\]
gives rise to the same toric bundle $E$, where $\bc^* \colon \Z^k = \Lie_\Z (K) \to \Z^N = \Lie_\Z (T)$ is the dual map of $\bc$.
The classes $P_i$ are then transformed into
\[
P'_i = P_i - x_i \pi^* c_1 (A), \qquad 1\leq i\leq k
\]
where $\pi \colon E \to B$ is the projection.
By choosing $A$ to be ample and $x_i \in \Z$ to be sufficiently negative, we can ensure that $\alpha^* P'_i$ is nef for all $\alpha \in F$.
Hence we may assume the following:

\begin{assumption}
\label{assump:ampleP}
The classes $P_1,\dots,P_k$ are all nef on $E$.
\end{assumption}

Applying the localization theorem \cite{Atiyah-Bott,Berline-Vergne} (see also \cite[Theorem VI.3.16]{Audin}) for equivariant cohomology to the toric bundle $E$, we obtain the following:

\begin{proposition}
\label{prop:localization}
The $\Frac(H^*_T(\pt))$-module homomorphism 
\[
\Xi \colon H^*_T (E) \otimes \Frac (H^*_T (\pt)) \to \bigoplus_{\alpha \in F} H^* (B) \otimes \Frac (H^*_T (\pt)) \qquad x \mapsto (\alpha^* x)_{\alpha \in F} 
\]
is an isomorphism.
The inverse map is given by
\[
(x_\alpha)\mapsto\sum_{\alpha\in F}\alpha_*\left(\frac{x_\alpha}{e_T\left(N^\alpha\right)}\right)
\]
where $e_T\left(N^\alpha\right)=\prod_{j\notin\alpha}\aUj$ is the $T$-equivariant Euler calss of the normal bundle $N^\alpha$ 
to $B^\alpha$ in $E$. 
\end{proposition}

\subsection{Gromov-Witten Theory for Toric Bundles}
We choose a homogeneous $\C$-basis $\phi_0 = 1, \phi_1, \dots, \phi_r, \phi_{r+1}, \dots, \phi_s$ of $H^*(B)$ such that $\phi_1, \dots, \phi_r$ are nef classes in $H^2 (B)$, and $\phi_{r+1}, \dots, \phi_s$ are basis of $H^{>2} (B)$.
Let $\tau_0, \dots, \tau_s$ be linear coordinates of $H^*(B)$ dual to the basis $\{ \phi_j \}^s_{j=0}$, and let $Q_1, \dots, Q_r$ be the Novikov variables dual to $\{ \phi_j \}^r_{j=1}$.
We decompose the $J$-function of $B$ as
\[
J_B(\tau,z)=\sum_{D\in\Eff(B)}J_D(\tau,z)Q^D.
\] 
The \emph{$I$-function} of $E$ is defined to be:
\[
I_E(t,\tau,z) := e^{Pt/z} \sum_{d \in \Z^k, D \in \Eff(B)} \frac{ \prod^N_{j=1} \prod^{0}_{m=-\infty} (U_j+mz) }{ \prod^N_{j=1} \prod^{U_j(\cD)}_{m=-\infty} (U_j+mz) } J_D(z,\tau) Q^D q^d e^{dt}
\]
where:
\begin{align*}
&q:=(q_1,\dots,q_k),			&&t:=(t_1,\dots,t_k),	&&U_j(\cD):=\sum^k_{i=1}c_{ij}d_i-\Lambda_j(D), 	\\
&q^d:=\prod^k_{i=1}q^{d_i}_i, 	&&Pt:=\sum^k_{i=1}P_it_i, 	&&dt:=\sum^k_{i=1}d_it_i.
\end{align*}

\begin{remark}
We consider that $(q,Q)$ represent Novikov variables for $E$ by associating $q_1^{P_1(\cD)} \cdots q_k^{P_k(\cD)} \cdot Q_1^{\pi^*\phi_1(\cD)} \cdots Q_r^{\pi^*\phi_r(\cD)}$ to $\cD \in \Eff(E)$.
We also remark that this representation depends on the choice of $P_1, \dots, P_k$.
From the nef assumption on $P$, it follows that $P(\cD) \in \Z^k_{\geq 0}$ and the summation range of $I_E$ reduces to $d \in \Z^k_{\geq 0}$. 
\end{remark}

\begin{remark}
\label{rem:IE}

We will regard $I_E$ as an element of several different spaces depending on the context.
By the localization isomorphism in Proposition \ref{prop:localization}, we see that
\[
I_E \in H^*_{T \times \C^\times_z} (E)_{\loc} [\![q,Q,t,\tau]\!] := \left( H^*_{T \times \C^\times_z} (E) \otimes_{H^*_{T \times \C^\times_z} (\pt)} \Frac ( H^*_{T \times \C^\times_z} (\pt) ) \right) [\![q,Q,t,\tau]\!]
\]
since $\alpha^* I_E$ belongs to $H^*(B) \otimes \Frac ( H^*_{T \times \C^\times_z} (\pt) ) [\![q,Q,t,\tau]\!]$ for all $\alpha \in F$.
Expanding the term $1/(U_j+mz)$ as $\sum^\infty_{n=0}(-U_j)^n(mz)^{-n-1}$, or equivalently, taking the Laurent expansion at $z = \infty$, we can regard $I_E$ as an element of the space $H^*_T(E) (\!(z^{-1})\!) [\![q,Q,t,\tau]\!]$.
On the other hand, by taking the Laurent expansion at $z=0$, we can regard $I_E$ as an element of the space $H^*_T(E)_{\loc} (\!(z)\!) [\![q,Q,t,\tau]\!]$.
In the statement of the following theorem, we take the first viewpoint.

\end{remark}

\begin{theorem}[{\cite[Theorem 1]{Brown}}]
The family $(t,\tau) \mapsto I_E(t,\tau,-z) \in H^*_T(E) (\!(z^{-1})\!) [\![q,Q,t,\tau]\!]$ lies in the Lagrangian cone $\cL_E$.
\end{theorem}

This theorem was established by showing that $I_E$ satisfies certain conditions that characterize elements lying on the cone $\cL_E$. 
The conditions require that the localizations $\alpha^*I_E$ must lie in the Lagrangian cone of $B$ twisted by the normal bundle of $B^\alpha$ (in the sense of Coates-Givental \cite{Coates-Givental}) and that their residues at simple poles in $z$ satisfy certain recursion relations. 
Brown checked the former condition by using Theorem \ref{thm:Brown} below, which gives a relationship between $\alpha^*I_E$ and $J_B$.
This theorem and its variants play an important role in our proof of convergence.

\subsection{Equivariant Mirror of Toric Manifold}
\label{subsect:equivmirror}

We introduce a Landau-Ginzburg model associated to the toric manifold $X$.
We set $\cM = (\C^\times)^k_q \times \C^N_\lambda$ and $\hcM = \C^k_q \times \C^N_\lambda$.
Let $Y$ be the subvariety of $\cM \times (\C^\times)^N_\xi$ defined by the equations
\begin{equation}
\label{eq:subvar}
\prod^N_{j=1} \xi^{c_{ij}}_j = q_i \qquad 1 \leq i \leq k	
\end{equation}
where $\bc = (c_{ij})$ is the matrix appearing in the subsection \ref{subsect:construction}.
We denote by $Y_{q,\lambda}$ the fibers of the projection $Y \hookrightarrow \cM \times (\C^\times)^N_\xi \to \cM$.
Define the equivariant phase function $W(q, \lambda, \xi)$ on $Y$ to be
\[
W(q, \lambda, \xi):=\sum^N_{j=1} ( \xi_j + \lambda_j \log \xi_j ).
\]
This is a multi-valued analytic function on $Y$.
Let $W_{q,\lambda}$ denote the restriction of $W$ to the fiber $Y_{q,\lambda}$.

We introduce a critical branch $\rho^\alpha(q,\lambda)$ of the phase function $W_{q,\lambda}$ corresponding to a fixed point $\alpha \in F$. 
Let $\cD_\alpha \subset \C^N_\lambda$ be the subset defined by the equation $\prod_{j \notin \alpha} \alpha^* u_j = 0$, and $\cD$ be the union of $\{ \cD_\alpha \}_{\alpha \in F}$.
We take an open neighborhood $\cU$ of $\{ 0 \} \times (\C^N_\lambda \setminus \cD)$ in $\hcM$, and set $\cU^* = \cU \cap \cM$.

\begin{proposition}
\label{prop:rhoalpha}
Shrinking $\cU$ if necessary, we have a unique analytic family of critical points $\rho^\alpha (q,\lambda)$ of $W_{q,\lambda}$ parametrized by $(q,\lambda) \in \cU^*$ such that 
\[
\lim_{q \to 0} \rho^\alpha_j (q,\lambda) = \auj \qquad 1 \leq j \leq N.
\]
Moreover, $\rho^\alpha \colon \cU^* \to \C^N$ can be extended to an analytic function on $\cU$. 
\end{proposition}

Before proceeding to the proof, we introduce a coordinate chart of $Y$ associated with the fixed point $\alpha \in F$. 
Let $T_\alpha \cong (\C^\times)^{N-k}$ denote the algebraic torus with coordinates $\{\xi_j\}_{j \notin \alpha}$. 
Then the projection $Y \hookrightarrow \cM \times (\C^\times)^N_\xi \to \cM \times T_\alpha$ is an isomorphism and gives a coordinate chart of $Y$. 
In fact, we can express the coordinate $\xi_n$ with $n \in \alpha$ in terms of $\{ \xi_j \}_{j \notin \alpha}$ by solving the equations \eqref{eq:subvar} as follows. 
\begin{equation} 
\label{eq:x_n} 
\xi_n =  \prod^k_{i=1} \left( q_i \right)^{(\bc^{-1}_\alpha)_{ni}} \prod_{j \notin \alpha} \xi_j^{-(\bc^{-1}_\alpha\bc)_{nj}}
\end{equation} 
Here we note that the exponent $(\bc^{-1}_\alpha)_{ni}$ of $q_i$ in this expression is non-negative by Lemma \ref{lem:nefness} and that at least one of $(\bc^{-1}_\alpha)_{n1}, \dots, (\bc^{-1}_\alpha)_{nk}$ is positive.  
We also introduce a modified phase function $W^\alpha$ on $V$ adapted to this coordinate chart. 
\begin{align*}
W^\alpha(q,\lambda,\xi) 
:=& W(q,\lambda,\xi) - \sum_{i=1}^k p^\alpha_i \cdot \log q_i	\\
=& \sum_{j \notin \alpha} (\xi_j - \auj \cdot \log \xi_j) + \sum_{n\in\alpha} \xi_n
\end{align*}
When written in the coordinates $(q,\lambda, \{ \xi_j \}_{j \notin \alpha})$, $W^\alpha$ has a well-defined limit as $q_i \to 0$; this follows from the fact that $\xi_n$ with $n \in \alpha$ contains only non-negative powers of $q_i$ and the fact that the terms $\log q_i$ in $W^\alpha$ cancel by Lemma \ref{lem:computePU}. We also note that $W^\alpha|_{V_{q,\lambda}}$ and $W_{q,\lambda}$ have the same critical points in common. 

\begin{proof}[Proof of Proposition \ref{prop:rhoalpha}] 
We work with the coordinate chart $(q,\lambda,\{ \xi_j \}_{j \notin \alpha})$ of $V$. The critical points of $W_{q,\lambda}$ are given by the equations 
\[
0 = \xi_j \parfrac{W^\alpha}{\xi_j} = \xi_j - \auj - \sum_{n \in \alpha}  \left( \bc^{-1}_\alpha \bc \right)_{nj} \xi_n     \qquad j \notin \alpha 
\]
where $\xi_n$ with $n \in \alpha$ are expressed in terms of $\{\xi_j\}_{j \notin \alpha}$ as in \eqref{eq:x_n}. These equations make sense at the limit point $q=0$ and have a unique solution $(\rho^\alpha_j(0,\lambda))_{j \notin \alpha} \in T_\alpha$ when $\lambda \notin \cD_\alpha$: 
\[
\rho^\alpha_j(0,\lambda) = \auj. 
\]
The condition $\lambda \notin \cD_\alpha$ ensures that $\rho^\alpha_j (0,\lambda) = \auj$ is non-zero. 
We also have $\rho^\alpha_n (0,\lambda) = 0 = \alpha^*u_n$ for $n \in \alpha$. 
We can easily see that this critical point is non-degenerate by computing the Hessian matrix of $W^\alpha|_{q=0}$. 
Therefore, we can find a unique branch of (non-degenerate) critical points 
$\rho^\alpha(q,\lambda)$ which converges to $\rho^\alpha(0,\lambda)$ as $q \to 0$ and is defined over an open neighborhood $\cU$ of  $\{0\} \times ( \C^N_\lambda \setminus \cD )$ in $\hcM$.
\end{proof} 

Since $\{ \rho^\alpha (0,\lambda) \}_{\alpha \in F}$ differ from each other, the critical branches $\{ \rho^\alpha(q,\lambda) \}_{\alpha \in F}$ also do. 
For later use, we expand $\rho^\alpha(q,\lambda)$ with respect to $q$ and describe properties of its coefficients.

\begin{definition}
We define $\Ral$ to be the ring $\C [ \lambda, \, \{ (\auj)^\pm \}_{j \notin \alpha,\alpha \in F} \, ] \subset \C(\lambda)$.
\end{definition}

\begin{corollary}
\label{cor:expandrho}
The critical branch $\rho^\alpha_j (q,\lambda)$ admits a Taylor expansion of the form
\[
\rho^\alpha_j (q,\lambda) = \auj + \sum_{d\neq0} \rho^\alpha_{j,d} (\lambda) q^d
\]
with $\rho^\alpha_{j,d} (\lambda) \in \Ral$.
It is homogeneous of degree
1 with respect to the grading $\deg \lambda_j = 1$ and $\deg q_i = \int_{p_i} c_1(X)$. 
\end{corollary}

\begin{proof}
This follows immediately by substituting the above expansion into the equations $(\xi_j \partial W^\alpha / \partial \xi_j=0)_{j \notin \alpha}$ and solving it inductively on $d$.
\end{proof}

We introduce the following oscillating integrals:
\begin{align*}
\Gamma(\nu,z)
&:= \int_C e^{ (- \xi + \nu \log \xi)/z } d\log \xi \\
\cI_\alpha (q,\lambda,z)
&:= \int_{C^\alpha_{q,\lambda}\subset Y_{q,\lambda}} e^{W^\alpha(q,\lambda,\xi)/z} \bigwedge_{j \notin \alpha} d \log \xi_j
\end{align*}
where $C$ is a descending Morse cycle of $\Re ( ( -\xi + \nu \log \xi ) / z )$ from the unique critical point $\xi=\nu$, and $C^\alpha_{q,\lambda}$ is that of $\Re(W^\alpha/z)$ from the critical point $\rho^\alpha (q,\lambda)$. 
Note that $\Gamma(\nu,z)$ is related to the Euler's Gamma function, and that $\cI_\alpha (q,\lambda,z)$ is the equivariant mirror of the toric manifold $X$ introduced by Givental \cite{Givental-mirror}. 

We then consider their \emph{stationary phase asymptotics}. 
In general, oscillatory integrals admit asymptotic expansions of the form: 
\[
\int_C e^{f(x)/z}g(x)dx_1 \cdots dx_n \sim (-2\pi z)^{n/2} \frac{1}{ \sqrt{ \det ( f_{ij} (p) ) } } e^{f(p)/z} \sum_{n \geq 0} a_n z^n
\]
as $z \to 0$ along some angular sector, see \cite{Wasow}.
Here $C$ is a descending Morse cycle of $\Re(f)$ from a critical point $p$ and $( f_{ij}(p) ) = ( \partial_{x_i} \partial_{x_j} f(p) )$ is the Hessian matrix at $p$.
The coefficients $a_n$ can be determined algebraically from the Taylor coefficients of the functions $f(x), g(x)$ at $p$, see \cite[Section 6.2]{Coates-Corti-Iritani-Tseng-hodge} for details.
In general, the power series $\sum_{n \geq 0} a_n z^n$ does not converge and the product $e^{f(p)/z} \sum_{n \geq 0} a_n z^n$ does not make sense even as a Laurent series of $z$; it should be interpreted as $e^{f(p)/z} ( \sum^m_{n=0} a_n z^n + O(z^{m+1}) )$ for all $m \geq 0$.

Define $(-2 \pi z)^{(N-k)/2} \cdot \hcI_\alpha (q,\lambda,z)$ and $\sqrt{2 \pi z} \cdot \hGamma(\nu,z)$ to be the stationary phase asymptotics of the oscillating integrals $\cI_\alpha(q,\lambda,z)$ and $\Gamma(\nu,z)$ respectively. 
By explicit computation, one obtains that 
\[
\hGamma (\nu,z) = \frac{1}{ \sqrt{\nu} } e^{ ( \nu \log \nu - \nu ) / z } \exp \left[ \sum^\infty_{m=1} \frac{ B_{2m} }{ 2m(2m-1) } \left( \frac{z}{\nu} \right)^{2m-1} \right]
\]
where $B_{2m}$ are the Bernoulli numbers.
We set $\hGamma_\alpha(\lambda,z)$ to be $\prod_{j \notin \alpha} \hGamma(\aUj,-z)$.
Finally, we define $\partial_{\Lambda_j}$ to be the vector field on the $\tau$-space $H^*(B)$ satisfying $\partial_{\Lambda_j} \tau = \Lambda_j$. 
Brown showed that:

\begin{theorem}[{\cite[Corollary 10]{Brown}}]
\label{thm:Brown}
\[
\hcI_\alpha (qe^t, \lambda + z \partial_\Lambda, z) J_B (Q_\alpha, \tau_\alpha, z) = \alpha^*I_E (t,\tau,z) \hGamma_\alpha (\lambda,z)
\]
where the variables $Q_\alpha$ and $\tau_\alpha$ are related to $Q$ and $\tau$ by
\[
(Q_\alpha)^D = q^{P^\alpha(D)} Q^D \quad \text{for } D \in \Eff(B), \quad \tau_\alpha = \tau + \alpha^* t.
\]
\end{theorem}

\begin{remark}
\label{rem:Brown}
{\rm (1)}
The above equation is slightly different from the one given by Brown.
This is due to the difference in the definition of $\hcI_\alpha (q,\lambda,z)$.
If we set $\hcI^{\text{Br}}_\alpha (q,\lambda,z)$ to be the stationary phase asymptotics appearing in \cite{Brown}, we have $\hcI^{\text{Br}}_\alpha (q,\lambda,z) = (-2 \pi z)^{(N-k)/2} e^{p^\alpha \log q / z} \hcI_\alpha (q,\lambda,z)$ and
\[
q^{-P^\alpha / z} \hcI^{\text{Br}}_\alpha (qe^t,\lambda + z \partial_\Lambda,z) J_B(Q,\tau,z) = (-2 \pi z)^{(N-k)/2} \hcI_\alpha (qe^t,\lambda + z \partial_\Lambda,z) J_B(Q_\alpha,\tau_\alpha,z).
\]

{\rm (2)}
Since both sides contain the terms $\log \lambda, \sqrt{\lambda}$ and $\sqrt{z}$, we should be careful where this equation is considered.
We should interpret them as elements of the space $\Exc_\alpha (\lambda + \Lambda,z) \cdot H^*(B) \otimes \Ral (\!(z)\!) [\![q,Q,t,\tau]\!]$, where
\[
\Exc_\alpha (\lambda,z) = \prod_{j \notin \alpha} \frac{1}{ \sqrt{\auj} } e^{ ( \auj - \auj \log \auj ) / z }. 
\]
Since $\auj |_{\lambda \to \lambda + \Lambda} = \aUj$ (Lemma \ref{lem:computePU}), we have
\[
\Exc_\alpha (\lambda + \Lambda,z) = \prod_{j \notin \alpha} \frac{1}{ \sqrt{\aUj} } e^{ ( \aUj - \aUj \log \aUj ) / z }.
\]
Although the product of $\Exc_\alpha (\lambda + \Lambda,z)$ and an element of $H^*(B) \otimes \Ral (\!(z)\!) [\![q,Q,t,\tau]\!]$ does not make sense as a formal power series, we can naturally factor out $\Exc_\alpha (\lambda + \Lambda,z)$ from both sides and the identity makes sense in the ring $H^*(B) \otimes \Ral (\!(z)\!) [\![q,Q,t,\tau]\!]$. 

We explain this interpretation for each side.
For the right-hand side, we regard $\alpha^*I_E$ as an element of $H^*(B) \otimes \Ral (\!(z)\!) [\![q,Q,t,\tau]\!]$ by taking the Laurent expansion at $z=0$ (see Remark \ref{rem:IE}), and $\Exc(\lambda + \Lambda,z)^{-1} \hGamma_\alpha(\lambda,z)$ as an element of $H^*(B) \otimes \Ral(\!(z)\!)$ naturally.
For the left-hand side, we give a detailed computation, which we also need in Section \ref{subsect:estimateM}. 
Write the stationary phase asymptotics $\hcI_\alpha(q,\lambda,z)$ as
\begin{align*}
\hcI_\alpha(q,\lambda,z) 
&= \frac{1}{\sqrt{\det(W_{ij}(\rho^\alpha))}} e^{W^\alpha(q,\lambda,\rho^\alpha)/z} \sum_{n\geq0} a^\alpha_n(q,\lambda) z^n	\\
&= \Exc_\alpha(\lambda,z) e^{F^\alpha(q,\lambda)/z} \sum_{n\geq0} A^\alpha_n(q,\lambda) z^n
\end{align*}
where we set
\[
F^\alpha(q,\lambda) = W^\alpha (q,\lambda,\rho^\alpha) - \sum_{j \notin \alpha} (\auj - \auj \log \auj), \quad A^\alpha_n(q,\lambda) = \sqrt{ \frac{ \prod_{j \notin \alpha} \auj }{ \det(W_{ij}(\rho^\alpha)) } } \cdot a^\alpha_n(q,\lambda).
\]
By a direct calculation along with Corollary \ref{cor:expandrho}, we can see that $F^\alpha(q,\lambda)$ and $A^\alpha_n(q,\lambda)$ are holomorphic functions on $U$, and the coefficients of thier Taylor expansion at $q=0$ belong to the ring $\Ral$.
We also note that they have the following asymptotics:
\begin{equation}
\label{eq:asmypFA}
F^\alpha(0,\lambda) = 0, \qquad A^\alpha_0 (0,\lambda) = 1.
\end{equation}
Consequently, the only thing we need to check is that $\Exc_\alpha(\lambda + z \partial_\Lambda,z) J_B(Q,\tau,z)$ belongs to the space $\Exc_\alpha (\lambda + \Lambda,z) \cdot H^*(B) \otimes \Ral (\!(z)\!) [\![q,Q,t,\tau]\!]$.

We decompose $\Exc_\alpha(\lambda + z \partial_\Lambda,z)$ into two parts:
\begin{multline*}
\Exc_\alpha(\lambda + z \partial_\Lambda,z) =
\biggl[ \Exc_\alpha(\lambda,z) \cdot \prod_{j \notin \alpha} e^{-\partialaj \log \auj} \biggr] 		\\
\cdot  \left[ \prod_{j\notin\alpha} \left( 1 + \frac{z \partialaj}{\auj} \right)^{-\frac{1}{2}} \cdot e^{\left.\left\{ z \partialaj - ( \auj + z \partialaj ) \log \left( 1 + \frac{z \partialaj}{\auj} \right) \right\} \right/ z} \right]
\end{multline*}
where $\partialaj$ is the vector field on $H^*(B)$ such that $\partialaj \tau = \aUj - \auj =: \aUjo$.
The former part transforms $J_B(\tau,z)$ as
\begin{align*}
&\Exc_\alpha(\lambda,z) e^{- \sum_{j\notin\alpha} \partialaj \log \auj} J_B(\tau,z) \\
=& \Exc_\alpha(\lambda,z) J_B (\tau - \sum_{j \notin \alpha} \aUjo \log \auj, z)	\\
=& \Exc_\alpha(\lambda,z) \biggl[ \prod_{j \notin \alpha} e^{-\aUjo \log \auj / z} \biggr] \cdot J_B (\tau,z) \left|_{Q^D \to (\auj)^{-\aUj(D)} Q^D} \right. \\
=& \biggl[ \prod_{j \notin \alpha} \frac{1}{ \sqrt{\auj} } e^{ ( \auj - \aUj \log \auj ) / z } \biggr] \cdot J_B (\tau,z) \left|_{Q^D \to (\auj)^{-\aUj(D)} Q^D} \right. .
\end{align*}
This belongs to the space
\[
\biggl[ \prod_{j \notin \alpha} \frac{1}{ \sqrt{\auj} } e^{ ( \auj - \aUj \log \auj ) / z } \biggr] \cdot H^*(B) \otimes \Ral (\!(z)\!) [\![q,Q,t,\tau]\!].
\]
We note that, since $\aUj - \auj \in H^2(B)$ is nilpotent, this is the same as the space $\Exc_\alpha (\lambda + \Lambda,z) \cdot H^*(B) \otimes \Ral (\!(z)\!) [\![q,Q,t,\tau]\!]$.
On the other hand, $J_B(\tau,z)$ transformed by the latter part belongs to the ring $H^*(B) \otimes \Ral (\!(z)\!) [\![q,Q,t,\tau]\!]$.

\end{remark}

\section{Formal Decomposition of $\QDM(E)$}
\label{sect:formal}

Let $E \rightarrow B$ be a toric bundle with toric fiber $X$ as in Section \ref{subsect:construction}.
We first introduce the big version of $I_E$ and use it to constract a \emph{formal decomposition} of $\QDM(E)$.
As we remarked in Remark \ref{rem:CGT}, this decomposition is implicitly observed in \cite[Proposition 3.1]{Coates-Givental-Tseng}.
In addition to their work, we will confirm that this decomposition is compatible with the pairings and the grading operators.

\subsection{Big $I$-function of $E$}
\label{subsect:bigI}

Note that the $I$-function $I_E$ has $\dim H^*(B) + \dim H^2(X)$ number of parameters $(t,\tau)$ whereas the rank of $H^*(E)$ is 
$\dim H^*(B)\cdot\dim H^*(X)$. 
We add parameters $\sigma_{i,j}$ to $I_E$ so that the number of parameters equals the rank of $H^*(E)$.
We note that similar big $I$-functions have been introduced in \cite[Example 4.14]{Iritani-quantum}, \cite{Ciocan-Kim}.

As $H^*(X)$ is generated by $H^2(X)$, one can take monomials $T_{k+1},\dots,T_{l}$ in $k$-variables such that
\[p_0=1,p_1,\dots,p_k,p_{k+1}=T_{k+1}(p_1,\dots,p_{k}),\dots,p_l=T_{l}(p_1,\dots,p_{k})\]
form a basis of $H^*(X)$. 
We also set $T_0=1$ and $T_i(p_1,\dots,p_k)=p_i$ for $1\le i\le k$ so that $T_0(p),T_1(p),\dots,T_l(p)$ give a basis of $H^*(X)$. 
We give a $H^*_T(\pt)$-basis of $H^*_T(E)$ as follows:
\[
\{ e_{i,j} = T_i(P_1, \dots, P_k) \phi_j \,|\,0\leq i\leq l,0\leq j\leq s \}.
\]
Define the \emph{big $I$-function} \cite[Example 4.14]{Iritani-quantum}, \cite{Ciocan-Kim} of $E$ as follows:
\[
\bigIE:=\Delta(\sigma) I_E(t,\tau,z),
\]	
where $\Delta(\sigma)$ is the differential operator
\[
\Delta(\sigma) = \left[ \prod_{(i,j)\in \sfS} \exp \left( \sigma_{i,j} \parfrac{}{\tau_j} T_i \left( z\parfrac{}{t_1}, \dots, z\parfrac{}{t_k} \right) \right) \right], 
\]
\[
\sfS = \{ (i,j) \mid1\leq i\leq l, 1\leq j\leq s \} \cup \{ (i,0)\mid k+1\leq i\leq l \}.
\]
We use the following coordinates:
\[
x = (x_0, x_1, \dots, x_r, x_{r+1}, \dots, x_s) := (\tau_0, Q_1 e^{\tau_1}, \dots, Q_r e^{\tau_r}, \tau_{r+1}, \dots, \tau_s)
\]
\begin{equation}
\label{eq:Definition_y}
y = ( y_{i,j} \mid 0\leq i\leq l, 0\leq j\leq s), \qquad y_{i,j}:=
\begin{cases}
q_ie^{t_i}			&1\leq i\leq k,j=0	\\
x_j				&i=0,0\leq j\leq s	\\
\sigma_{i,j}			&\text{otherwise}	
\end{cases}
\end{equation}

As in Remark \ref{rem:IE}, the function $\bigIE$ can be interpreted in two ways; an element of $H^*_T (E) (\!(z^{-1})\!) [\![q,Q,\tts]\!]$ or that of $H^*_T (E)_{\loc} (\!(z)\!) [\![q,Q,\tts]\!]$.
We note that $\bigI(\tts,-z)$ interpreted in the former manner lies in the Lagrangian cone $\cL_E$. 
This follows immediately from the next lemma due to Coates and Givental.

\begin{lemma}[{\cite[]{Coates-Givental}; see also \cite[in the proof of Theorem 4.6]{Coates-Corti-Iritani-Tseng-computing}}]
\label{lem:CG}
Let $Y$ be a smooth projective variety, $I(\tau_1,\dots,\tau_n,z)$ be a family of elements lying in the Lagrangian cone $\cL_Y$ of $Y$, and $\Phi(x,z)$ be any formal power series in $z$ with coefficients polynomials in $x_1,\dots,x_n$. 
Then the family
\[
\tI (\nu,\tau,z) = \exp \left( \left. \nu \Phi \left( z\partial_{\tau}, z \right) \right/ z \right) I(\tau,z)
\]
also lies in $\cL_Y$.
\end{lemma}

We have an equation for $\bigIE$ similar to Theorem \ref{thm:Brown}. 

\begin{proposition}
\label{prop:bigBrown}
\[
\alpha^*\bigIE \hGamma_\alpha (\lambda,z) = \Delta (\sigma) \hcI_\alpha (qe^t, \lambda + z \partial_\Lambda, z) J_B(Q_\alpha,\tau_\alpha,z)
\]
in the space $\Exc_\alpha(\lambda + \Lambda,z) \cdot H^*(B) \otimes \Ral (\!(z)\!) [\![q,Q,\tts]\!]$.
\end{proposition}

\begin{proof}
Apply the differential operator $\Delta(\sigma)$ to the equation in Theorem \ref{thm:Brown}. 
Note that $\hGamma_\alpha(\lambda,z)$ is independent of $t$ and $\tau$, hence commutes with $\Delta(\sigma)$. 
\end{proof}

\subsection{Formal Decomposition}
\label{subsect:formaldecomposition}

In this subsection, we construct a formal decomposition of the quantum connection $\nabla^E$ of the total space $E$.

In order to distinguish parameters appearing in each connection, we will use the parameters $\tts$ for the connection obtained by $\bigI$, and the parameter $\hsigma$ for the quantum connection $\nabla^E$.
These parameters are related by some mirror transformation.

Define the matrix $\LI(\tts,z)$ to be
\begin{equation*}
\LI(\tts,z) e_{i,j} :=
\begin{cases}
\parfrac{}{t_i}\bigIE		&1\leq i\leq k,j=0	\\
\parfrac{}{\tau_j}\bigIE		&i=0,0\leq j\leq s	\\
\parfrac{}{\sigma_{i,j}}\bigIE	&\text{otherwise}	
\end{cases}
\end{equation*}

As in the case of $\bigI$, the matrix $\LI$ can be interpreted as an element of the rings 
\[
\End_{\C[\lambda]} ( H^*_T(E) )(\!(z^{-1})\!) [\![q,Q,\tts]\!] 
\quad \text{or} \quad 
\End_{\C(\lambda)} ( H^*_T(E)_{\loc} ) (\!(z)\!) [\![q,Q,\tts]\!].
\]
Similarly, the matrix $e^{(-t-\tau)/z} \LI$ can be interpreted as an element of the rings
\[
\End_{\C[\lambda]} ( H^*_T(E) ) (\!(z^{-1})\!) [\![y]\!]
\quad \text{or} \quad
\End_{\C(\lambda)} ( H^*_T(E)_{\loc} ) (\!(z)\!) [\![y]\!].
\]

\begin{proposition}
\label{prop:BFforLI}
There exists a unique factorization of the form
\[
\LI (\tts,z) = \bfL_E (\hsigma,z) \hR (\lambda,y,z)
\]
where 
\[
\hsigma (\lambda,\tts) - \sigma \in H^*_T(E) [\![y]\!], \qquad \hsigma |_{(q,Q)=0} = \sigma,
\]
\[
\hR (\lambda,y,z) \in \End_{\C[\lambda]} ( H^*_T(E) ) [z][\![y]\!], \qquad \hR (\lambda,y,z)|_{(q,Q)=0} = \Id.
\]

\end{proposition}

\begin{proof}
Performing the Birkhoff factorization of $\LI (\tts,z)$ in the $(q,Q)$-adic topology, we have a unique decomposition of the form
\[
\LI(\tts,z) = \bfL_- (\tts,z) \hR (\lambda,y,z)
\]
where $\bfL_- (\tts,z) \in \End_{\C[\lambda]} ( H^*_T(E) ) [\![q,Q,\tts,z^{-1}]\!]$ with $\bfL_- \mid_{z^{-1} = 0} = \Id$, and $\hR (\lambda,y,z)$ is a matrix satisfying the desired condition.
As $\bigIE$ lies in the Lagrangian cone $\cL_E$ of $E$, the columns of $\bfL_- (\tts,z)$ span the tangent space of $\cL_E$.
Hence $\bfL_- (\tts,z)$ is in fact a fundamental solution of the quantum connection $\nabla^E$.
Since $\hsigma$ is obtained as the coefficient of the $z^{-1}$-term of $\bfL_- (\tts,z) e_{0,0}$, the conditon imposed on $\hsigma$ also holds.
\end{proof}

We may interpret $\hsigma = \hsigma (\lambda,\tts)$ as the formal coordinate change of $\Spec ( H^*_T(\pt) [\![x]\!] )$, which is the base space of $\QDMT(E)$. 
Since $e^{-\sigma/z} \bigIE$ is homogeneous with respect to $\Gr^E$, this coordinate change is compatible with the Euler vector fields and the matrix $\hR (\lambda,y,z)$ is also homogeneous.

\begin{proposition}
\label{prop:homogeneity1}
{\rm (1)}
The map $\hsigma$ is homogeneous$:$
\[
\hsigma_* \cE^E = \hat{\partial}_{c_1^T(E)} + \sum_{i,j} \left( 1 - \frac{1}{2} \deg e_{i,j} \right) \hsigma_{i,j} \parfrac{}{\hsigma_{i,j}} + \sum^N_{j=1} \lambda_j \parfrac{}{\lambda_j}
\]
where $\hat{\partial}_{c_1^T(E)}$ is the vector field satisfying $\hat{\partial}_{c_1^T(E)} \hsigma = c_1^T(E)$ and $\hat{\partial}_{c_1^T(E)} \lambda_j = 0$.

{\rm (2)}
The matrix $\hR$ is homogeneous of degree zero with respect to $\Gr^E:$
\[
\Gr^E \circ \hR = \hR \circ \Gr^E.
\]
\end{proposition}

The matrix $\LI$ defines the connection $\nablaI$ by
\[
\nablaI = d + z^{-1} \AI , \quad \AI (y,z) = z \LI (\tts,z)^{-1} d \LI (\tts,z).
\]
Then, the matrix $\hR$ gives a formal gauge transformation between $\nabla^E$ and $\nablaI$.
We shall write the connection matrix of $\AI$ as
\[
\AI = \sum^k_{i=1} \AI_{(i,0)} d t_i + \sum^s_{j=0} \AI_{(0,j)} d \tau_j + \sum_{(i,j) \in \sfS} \AI_{(i,j)} d \sigma_{i,j}.
\]
\begin{proposition}
\label{prop:AI}
All entries of the connection matrix $\AI_{(i,j)}$ belong to the ring $\C[\lambda,z] [\![y]\!]$.
\end{proposition}

\begin{proof}
By Proposition \ref{prop:BFforLI}, $\AI$ can be written as
\begin{equation}
\label{eq:AI}
\AI = \hR^{-1} \bfA_E \hR + z \hR^{-1} d \hR.
\end{equation}
Since $\bfA_{E,(i,j)} \in \End_{\C[\lambda]} ( H^*_T(E) ) [\![y]\!]$ and $\hR, \hR^{-1} \in \End_{\C[\lambda]} ( H^*_T(E) ) [z][\![y]\!]$, entries of $\AI_{(i,j)}$ are elements of $\C[\lambda,z][\![y]\!]$.
\end{proof}

We will relate $\LI$ and $\bfL_B^{\oplus | F |}$ as mentioned at the beginning of this section. 
For this purpose, we introduce a new matrix $\bfM$.
This also plays an important role in giving estimates of connection matrices in Section \ref{sect:convergence1} and Section \ref{sect:convergence2}. 
Fix a total order on the set $F$, and denote its elements by $\alpha_0,\dots,\alpha_l$.
Let $\bfM$ be the block matrix:
\[
\bfM=
\begin{pmatrix}
\bfM_{\alpha_0,0}		&\cdots	&\bfM_{\alpha_0,l}	\\
\vdots			&\ddots	&\vdots		\\
\bfM_{\alpha_l,0}		&\cdots	&\bfM_{\alpha_l,l}	\\
\end{pmatrix}
=
\begin{pmatrix}
-		&\bfM_{\alpha_0}	&-	\\
		&\vdots		&	\\
-		&\bfM_{\alpha_l}	&-	\\
\end{pmatrix}
\]
where $\bfM_{\alpha,i} \colon H^*(B) \to \Exc_\alpha(\lambda + \Lambda,z) \cdot H^*(B) \otimes \Ral (\!(z)\!) [\![q,Q,\tts]\!]$ is given by  
\begin{equation}
\label{eq:bfMai}
\bfM_{\alpha,i} = T_i \left( z \parfrac{}{t} \right) \Delta (\sigma) \hcI_\alpha ( qe^t, \lambda + z \partial_\Lambda, z ) \bfL_B ( Q_\alpha, \tau_\alpha, z ). 
\end{equation}
We regard $\bfM = (\bfM_\alpha)_{\alpha\in F}$ as the matrix representing the following $H^*_T(\pt)$-linear morphism:
\begin{align*}
\bfM \colon H^*_T(E)													&\longrightarrow
\bigoplus_{\alpha \in F} \Exc_\alpha(\lambda + \Lambda,z) \cdot H^*(B) \otimes \Ral (\!(z)\!) [\![q,Q,\tts]\!]	\\
e_{i,j} 				\phantom{A} 										&\longmapsto	
\phantom{AAAA}		(\bfM_\alpha e_{i,j})_{\alpha\in F}							\\
																\\
\bfM_\alpha \colon H^*_T(E)												&\longrightarrow 
\Exc_\alpha(\lambda + \Lambda,z) \cdot H^*(B) \otimes \Ral (\!(z)\!) [\![q,Q,\tts]\!]					\\
e_{i,j}				\phantom{A}											&\longmapsto	
\phantom{AAAA}		\bfM_{\alpha,i} \phi_j
\end{align*}
From Proposition \ref{prop:bigBrown}, we have
\[
\bfM = \bfGamma \Xi \LI
\]
where $\Xi$ is the isomorphism introduced in Proposition \ref{prop:localization}, and $\bfGamma$ is the block matrix: 
\[
\bfGamma=
\begin{pmatrix}
\hGamma_{\alpha_0}	&		&O				\\
				&\ddots	&				\\
O				&		&\hGamma_{\alpha_l}	\\
\end{pmatrix}
.\]
Note that $\bfM$ determines the same connection as $\LI$ does, that is,
\[
\AI = z \bfM^{-1} d \bfM. 
\]

The following proposition follows from the argument similar to Proposition \ref{prop:BFforLI} along with the computation in Remark \ref{rem:Brown}.

\begin{proposition}
\label{prop:BFformodifiedJB}
There exists a unique factorization of the form
\[
\Delta(\sigma) \hcI_\alpha (qe^t, \lambda + z\partial_\Lambda, z) \bfL_B(Q_\alpha,\tau_\alpha,z) = \bfL_B(Q_\alpha, \tau^\star_\alpha, z) \bfR^\star_\alpha (\lambda, y, z)
\]
where 
\begin{align*}
&\tau^\star_\alpha - \biggl[ \tau_\alpha + \sum_{j \notin \alpha} (\auj - \aUj \log \auj) \biggr] \in H^*(B) \otimes \Ral[\![ y ]\!], \\
&\tau^\star_\alpha |_{(q,Q)=0} = \alpha^* \sigma + \sum_{j \notin \alpha} (\auj - \aUj \log \auj), \\
&\bfR^\star_\alpha (\lambda, y, z) \in \biggl( \prod _{j \notin \alpha} \frac{1}{ \sqrt{\auj} } \biggr) \cdot \End (H^*(B)) \otimes \Ral[\![ y, z ]\!], \\
&\bfR^\star_\alpha (\lambda, y, z) |_{(q,Q) = 0} = \biggl( \prod _{j \notin \alpha} \frac{1}{ \sqrt{\auj} } \biggr) \cdot \Id.
\end{align*}
\end{proposition}

\begin{proof}
From Remark \ref{rem:Brown} and Lemma \ref{lem:CG}, we can see that the family
\[
\biggl[ \prod_{j \notin \alpha} \frac{1}{ \sqrt{\auj} } e^{ ( \auj - \aUj \log \auj ) / z } \biggr]^{-1} \Delta(\sigma) \hcI_\alpha (qe^t, \lambda + z\partial_\Lambda, z) J_B(Q,\tau,z)\left|_{Q^D \to (\auj)^{\aUj(D)} Q^D} \right.
\]
lies in the Lagrangian cone $\cL_B$. 
The argument similar to Proposition \ref{prop:BFforLI} gives the desired factorization. 
\end{proof}

As in the case of $\hsigma$ and $\hR$, we have homogeneity results for them.

\begin{proposition}
\label{prop:homogeneity2}
{\rm (1)}
The morphism $\tau^\star_\alpha$ is homogeneous$:$
\[
(\tau^\star_\alpha)_* \cE^E = \partial^\star_{c_1(B)} + \sum_j \left( 1 - \frac{1}{2} \deg \phi_j \right) \tau^\star_{\alpha,j} \parfrac{}{\tau^\star_{\alpha,j}} + \sum^N_{j=1} \lambda_j \parfrac{}{\lambda_j}
\]
where $\partial^\star_{c_1(B)}$ is the vector field satisfying $\partial^\star_{c_1(B)} \tau^\star_\alpha = c_1(B)$ and $\partial^\star_{c_1(B)} \lambda_j = 0$.

{\rm (2)}
The matrix $\bfR^\star_\alpha$ is homogeneous with respect to $(\tau^\star_\alpha)^* \Gr^B:$
\[
(\tau^\star_\alpha)^* \Gr^B \circ \bfR^\star_\alpha = \bfR^\star_\alpha \circ \left( (\tau^\star_\alpha)^* \Gr^B + \frac{1}{2} \dim B - \frac{1}{2} \dim E \right).
\]
\end{proposition}

\begin{proof}

Since $\deg^E ( (\auj)^{-\aUj(D)} Q_\alpha^D e^{ \corr{\tau_\alpha,D} } ) = \deg^B ( Q^D e^{\corr{\tau,D}} )$, the matrix
\[
e^{-\tau_\alpha/z} \biggl[ \prod_{j \notin \alpha} \frac{1}{ \sqrt{\auj} } e^{ ( \auj - \aUj \log \auj ) / z } \biggr]^{-1} \Delta(\sigma) \hcI_\alpha (qe^t, \lambda + z\partial_\Lambda, z) \bfL_B(Q_\alpha,\tau_\alpha,z)
\]
is homogeneous of degree zero with respect to the grading $\cE^E + z \parfrac{}{z} + \Gr^B_0$, and hence the matrices obtained by the Birkhoff decomposition are also homogeneous of degree zero.
Therefore, we only need to calculate $\cE^E(\tau^\exc_\alpha)$ where $\tau^\exc_\alpha$ is the exceptional part of $\tau^\star_\alpha$:
\[
\tau^\exc_\alpha := \sum^r_{j=0} \tau_j \phi_j + \alpha^*t + \sum_{j \notin \alpha} (\auj - \aUj \log \auj).
\]
Since $c_1^T(E) = c_1(B) + \sum^N_{j=1} U_j$, we have
\[
\cE^E(\tau^\exc_\alpha) = \tau_0 + c_1(B) + \sum^k_{i=1} p^\alpha_i t_i + \sum_{j \notin \alpha} (\auj - \auj \log \auj)
\]
and hence
\[
\cE^E(\tau^\exc_{\alpha,0}) = \tau^\exc_{\alpha,0}, \qquad \sum^r_{j=1} \cE^E(\tau^\exc_{\alpha,j}) \phi_j = c_1(B).
\]

\end{proof}

\begin{proposition}
\label{prop:BFforM}
The Birkhoff factorization of the matrix $\bfM$ is of the form$:$

\[
\bfM=
\begin{pmatrix}
\bfL_B(Q_{\alpha_0}, \tau^\star_{\alpha_0},z)	&		&O							\\
							&\ddots	&							\\
O							&		&\bfL_B(Q_{\alpha_l}, \tau^\star_{\alpha_l},z)	\\
\end{pmatrix}
\bfR^\star
\]

\end{proposition}

\begin{proof}
From Proposition \ref{prop:BFformodifiedJB}, we have a factorization of $\bfM_{\alpha,0}$:
\[
\bfM_{\alpha,0} (\tts,z) = \bfL_B (Q_\alpha,\tau^\star_\alpha,z) \bfR^\star_{\alpha}.
\]
Since $\bfL_B$ is a fundamental solution for the connection $\nabla^B = d + z^{-1} \bfA_B$, we have
\[
z \parfrac{}{t_m} \circ \bfL_B (Q_\alpha, \tau^\star_\alpha,z) = \bfL_B (Q_\alpha, \tau^\star_\alpha,z) \circ \left( z \parfrac{}{t_m} + \sum_{0\leq n\leq s} \parfrac{\tau^\star_{\alpha,n}}{t_m} (\bfA_B)_n (\tau^\star_\alpha) \right)
\]
Hence $\bfM_{\alpha,j}$ is decomposed as
\begin{align*}
\bfM_{\alpha,j} (\tts,z)
&= T_j \left( z \parfrac{}{t} \right) \bfM_{\alpha,0} (\tts,z)		\\
&= T_j \left( z \parfrac{}{t} \right) \bfL_B (Q_\alpha,\tau^\star_\alpha,z) \bfR^\star_{\alpha}		\\
&= \bfL_B(Q_\alpha,\tau^\star_\alpha,z) T_j \left( z \parfrac{}{t} + \sum_{0\leq n\leq s} \parfrac{\tau^\star_{\alpha,n}}{t} (\bfA_B)_n (\tau^\star_{\alpha}) \right) \bfR^\star_{\alpha}.
\end{align*}
Let $\bfR^\star$ be
\begin{equation}
\label{eq:R*}
\bfR^\star =
\begin{pmatrix}
\bfR^\star_{0,0}	&\cdots	&\bfR^\star_{0,l}	\\
\vdots	&\ddots	&\vdots	\\
\bfR^\star_{l,0}	&\cdots	&\bfR^\star_{l,l}	\\
\end{pmatrix}
, \qquad \bfR^\star_{i,j} := T_j \left( z \parfrac{}{t} + \sum_{0\leq n\leq s} \parfrac{\tau^\star_{\alpha_i,n}}{t} (\bfA_B)_n (\tau^\star_{\alpha_i}) \right) \bfR^\star_{\alpha_i}.
\end{equation}
We then obtain the desired formula.
\end{proof}

In summary, the relations among the matrices and the connections appearing in this subsection are as shown in the following diagram:
\[
\xymatrix{
\bfL_E(\hsigma,z)\ar@{.>}[d]		&\LI(\tts,z)	\ar@{.>}[d]	&\bfM (\tts,z)		\ar@{.>}[d]		
&\bigoplus_{\alpha \in F} \bfL_B(Q_\alpha, \tau^\star_\alpha,z)		\ar@{.>}[d]\\
\nabla^E		&\nablaI	\ar[l]_{\hR}	\ar@{=}[r]	&\nablaI	\ar[r]^(.4){\bfR^\star}		&(\nabla^B)^{\oplus l}	
}
\]
We define the compositions of these transformations:
\begin{equation}
\label{eq:chR}
\chtau_\alpha := \tau^\star_\alpha \circ \hsigma^{-1}, \qquad \chbfR := \bfR^\star \circ \hR^{-1}
\end{equation}
Combining Proposition \ref{prop:BFforLI}, Proposition \ref{prop:BFformodifiedJB} and Proposition \ref{prop:BFforM}, we obtain their asymptotics:
\begin{align*}
&\chtau_\alpha - \biggl[ \alpha^* \htau + \sum_{j \notin \alpha} (\auj - \aUj \log \auj) \biggr] \in H^*(B) \otimes \Ral[\![ \hy ]\!], \\
&\chtau_\alpha |_{(q,Q)=0} = \alpha^* \hsigma + \sum_{j \notin \alpha} (\auj - \aUj \log \auj), \\
&\chbfR (\lambda, y, z) \in \bigoplus_{\alpha \in F} \biggl( \prod _{j \notin \alpha} \frac{1}{ \sqrt{\auj} } \biggr) \cdot \End_{\C [\lambda]} ( H^*_T(E), H^*_T(B) ) \otimes_{\C [\lambda]} \Ral[\![ \hy, z ]\!], \\
&\chbfR (\lambda, y, 0) |_{(q,Q) = 0} (e) = \biggl( \prod_{j \notin \alpha} \frac{1}{ \sqrt{ \auj } } \cdot \alpha^* e \biggr)_{\alpha \in F} \qquad e \in H^*_T(E).
\end{align*}

We write the Euler vector fields for $\Spec( H^*_T(\pt) [\![ \hy ]\!] )$ and $\Spec( H^*_T(\pt) [\![ \chx_\alpha ]\!] )$
\begin{align*}
\hcE^E &:= \hat{\partial}_{c_1^T(E)} + \sum_{i,j} \left( 1 - \frac{1}{2} \deg e_{i,j} \right) \hsigma_{i,j} \parfrac{}{\hsigma_{i,j}} + \sum^N_{j=1} \lambda_j \parfrac{}{\lambda_j}, \\
\chcE^B &:= \check{\partial}_{c_1(B)} + \sum_j \left( 1 - \frac{1}{2} \deg \phi_j \right) \chtau_{\alpha,j} \parfrac{}{\chtau_{\alpha,j}} + \sum^N_{j=1} \lambda_j \parfrac{}{\lambda_j}
\end{align*}
where $\check{\partial}_{c_1(B)}$ is the vector field satisfying $\check{\partial}_{c_1(B)} \chtau_\alpha = c_1(B)$ and $\check{\partial}_{c_1(B)} \lambda_j = 0$.
From Proposition \ref{prop:homogeneity1} and Proposition \ref{prop:homogeneity2}, we can see that $\chtau_\alpha$ and $\chbfR$ are homogeneous.

\begin{proposition}
{\rm (1)}
The morphism $\chtau_\alpha$ is homonegeneous$:$
\[
\chtau_{\alpha *} \hcE^E = \chcE^B.
\]

{\rm (2)}
The matrix $\chbfR$ is homogeneous$:$
\[
\chbfR \circ \Gr^E = \bigoplus_{\alpha \in F} \chtau_\alpha^* \Gr^B \circ \chbfR.
\]
\end{proposition}

We also have the following formal decomposition of $\QDM_T(E)$. 

\begin{theorem}
\label{thm:formaldecomp}
The matrix $\chbfR$ gives the following isomorphism of quantum $D$-modules which is compatible with pairing and grading$:$
\[
\chbfR \colon \QDMT(E) \cong \bigoplus_{\alpha \in F} \chtau^*_\alpha \QDM(B).
\]
\end{theorem}

\begin{proof}
It is enough to check that $\chbfR$ preserves the pairings.
More precisely, we will check that
\[
\pair{\omega_1(-z)}{\omega_2(z)}^E = \sum_{\alpha \in F} \pair{\chbfR_\alpha(-z) \omega_1(-z)}{\chbfR_\alpha(z) \omega_2(z)}^B
\]
for $\omega_1, \omega_2 \in H^*_T(E)_{\loc} [z] [\![\hy]\!]$.
For $\eta_1=(\eta^\alpha_1)_{\alpha\in F},\eta_2=(\eta^\alpha_2)_{\alpha\in F}\in \bigoplus_{\alpha\in F}H^*(B)$, we have
\begin{align*}
\pair{\Xi^{-1} \bfGamma(-z)^{-1} \eta_1}{\Xi^{-1} \bfGamma(z)^{-1} \eta_2}^E
&=\sum_{\alpha \in F} \pair{ \alpha_* \left( \frac{\hGamma_\alpha(-z)^{-1} \eta^\alpha_1}{e_T(N^\alpha)} \right) }{ \alpha_* \left( \frac{\hGamma_\alpha(z)^{-1} \eta^\alpha_2}{e_T(N^\alpha)} \right) }^E						\\
&=\sum_{\alpha\in F} \int_E\alpha_*(\eta^\alpha_1\cup\eta^\alpha_2)								\\
&=\sum_{\alpha\in F}\pair{\eta^\alpha_1}{\eta^\alpha_2}^B.										
\end{align*}
Here we used $(\hGamma_\alpha(-z)\hGamma_\alpha(z))^{-1} = e_T(N^\alpha)$ and the projection formula:
\[
\alpha_* \omega_1 \cup \alpha_* \omega_2 
= \alpha_* ( \omega_1 \cup \alpha^* \alpha_* \omega_2 )
= \alpha_* ( \omega_1 \cup \omega_2 \cup e_T(N^\alpha) ).
\]
Hence, from Proposition \ref{prop:unitarity}, 
\begin{align*}
\pair{\omega_1(-z)}{\omega_2(z)}^E
&=\pair{\bfL_E(\hsigma,-z) \omega_1(-z)}{\bfL_E(\hsigma,z) \omega_1(z)}^E	\\
&=\sum_{\alpha \in F} \pair{\bfL_B(Q_\alpha, \chtau_\alpha, -z) \chbfR_\alpha(-z) \omega_1(-z)}{\bfL_B(Q_\alpha, \chtau_\alpha, z) \chbfR_\alpha(z) \omega_2(z)}^B	\\
&=\sum_{\alpha \in F} \pair{\chbfR_\alpha(-z) \omega_1(-z)}{\chbfR_\alpha(z) \omega_2(z)}^B.
\end{align*}
\end{proof}

Under the assumption that the big quantum cohomology $QH^*(B)$ is analytic, we will prove the analyticity of $\hsigma$ and $\hR$ in Section \ref{sect:convergence1}, and that of $\tau^\star$ and $\bfR^\star$ in Section \ref{sect:convergence2}.

\section{Convergence of $\hR$}
\label{sect:convergence1}

Henceforth we assume that the big quantum cohomology $QH^*(B)$ of the base $B$ converges. 
Under this assumption, we will prove that the map $\hsigma(q,Q,\lambda,t,\tau,\sigma)$ and the gauge transformation $\hR(\lambda,y,z)$, defined only formally in section \ref{subsect:formaldecomposition}, are in fact analytic.
The key ingredient is the \emph{gauge fixing theorem}, which says that a logarithmic flat connection with nilpotent residue can be transformed to a $z$-independent form by a unique gauge transformation with \emph{good} estimate if its connection matrix satisfies \emph{good} estimate.  
Our plan is as follows.
We first estimate the matrix $\bfM$ by explicit calculation, and then estimate the matrix $\bfA^I$.
Finally, we apply the gauge fixing theorem to $\bfA^I$ and prove the analyticity of $\hsigma$ and $\hR$.

\subsection{Gauge Fixing}

Denote by $\cO_{\C^N_\lambda}$ the analytic structure sheaf of the complex manifold $\C^N_\lambda$.
For any multi index $\bm = (m_1, \dots, m_a) \in \Z^a_{\geq 0}$, we define $| \bm | := \sum^a_{i=1} m_i$.

\begin{definition}[Compare {\cite[Section 4]{Iritani-convergence}}, {\cite[Definition 7.1]{Coates-Corti-Iritani-Tseng-hodge}}]
\label{def:rings1}

Let $U$ be an open subset of $\C^N_\lambda$.
Define $\cO^{y,z,z^{-1}}_\lambda(U)$ to be the subspace of the ring $\cO_{\C^N_\lambda}(U)(\!(z)\!)[\![y]\!]$ consisting of series $\sum_{\bm,n}a_{\bm,n}(\lambda)y^\bm z^n$ such that there exist a positive continuous function $C \colon U \to \R$ and a positive constant $C'$ satisfying
\begin{itemize} 
\item[(1)] 
$a_{\bm,n}(\lambda)=0$ for $n < - C' ( |\bm| + 1 )$ and,
\item[(2)] 
the following estimate:
\[
|a_{\bm,n}(\lambda)|\leq C(\lambda)^{|\bm|+|n|+1}
\begin{cases}
n!			&n\geq0	\\
1/(-n)! \quad	&n<0		
\end{cases}
\]
for any $\lambda \in U$.
\end{itemize} 
Similarly, define $\cO^{y,z}_\lambda(U)$ to be the subspace of the ring $\cO_{\C^N_\lambda}(U)[\![y,z]\!]$ consisting of series $\sum_{\bm,n}a_{\bm,n}(\lambda)y^\bm z^n$ such that there exists a continuous function $C \colon U \to \R$ satisfying the estimate
\[
|a_{\bm,n}(\lambda)|\leq C(\lambda)^{|\bm|+n+1}|\bm|^n
\]
for any $\lambda \in U$.
Here we define $|\bm|^n = 1$ if $|\bm| = n = 0$.
\end{definition}

\begin{remark}
(1)
Take $a\in\cO^{y,z,z^{-1}}_\lambda(U)$ and consider a constant $C'$ and a function $C \colon U \to \R$ which satisfies the condition for $a$ in Definition \ref{def:rings1}.
It is clear that we can assume that $C'\geq1$ and $C(\lambda)\geq1$ for $\lambda\in U$.
When $C(\lambda) \geq 1$, it follows that
\[
C(\lambda)^n \leq C(\lambda)^m
\]
if $n\leq m$.
Since we will use this kind of inequality many times, we often assume that positive constants or continuous functions are not less than one to avoid any confusion.

(2)
In order to avoid complicated notation, we use the following notation as the extension of the factorial; for $n \in \Z$, we define
\[
n! =
\begin{cases}
n!			&n\geq0,	\\
1/(-n)! \quad	&n<0.		
\end{cases}
\]
This is not a standard notation, but it is useful in some situations.
For instance, it follows that, for any $n,m \in \Z$,
\[
n!m! \leq 2^{|n| + |m|} (n+m)!.
\]
\end{remark}

\begin{lemma}
\label{lem:both_rings}
Let $U$ be an open subset of $\C^N_\lambda$.
Then the spaces $\cO^{y,z,z^{-1}}_\lambda(U)$ and $\cO^{y,z}_\lambda(U)$ are both rings. 
\end{lemma}

\begin{proof}
Take two elements $a = \sum_{\bm,n}a_{\bm,n}(\lambda)y^\bm z^n$ and $b = \sum_{\bm,n}b_{\bm,n}(\lambda)y^\bm z^n$ in $\cO^{y,z,z^{-1}}_\lambda(U)$.
Let $C \colon U \to [1,+\infty)$ and $C'\geq1$ be a continuous function and a positive constant respectively that satisfy the condition (1) and (2) in Definition \ref{def:rings1} for $a$ and $b$ simultaneously. 
We have 
\[
ab = \sum_{\bm,n} c_{\bm,n}(\lambda) y^\bm z^n, \qquad c_{\bm,n} = \sum_{\bm_1, \bm_2; \bm_1 + \bm_2 = \bm} \sum_{n_1, n_2; n_1 + n_2 = n} a_{\bm_1,n_1}(\lambda) b_{\bm_2,n_2}(\lambda) 
\]
and $| c_{\bm,n}(\lambda) |$ can be estimated as follows:
\begin{align*}
| c_{\bm,n}(\lambda) | 
&\leq C(\lambda)^{| \bm | + | n | + 2} \sum_{\bm_1, \bm_2} \sum_{n_1, n_2} n_1 ! n_2 !	\\
&\leq C(\lambda)^{| \bm | + | n | + 2} \cdot 2^{| \bm |} ( n + C'  | \bm | + 2C' ) \cdot 2^{| n |} n!
\end{align*}
Hence $ab$ belongs to $\cO^{y,z,z^{-1}}_\lambda(U)$, which shows that $\cO^{y,z,z^{-1}}_\lambda(U)$ is a ring.

The proof that $\cO^{y,z}_\lambda(U)$ is a ring is almost the same as above.
\end{proof}

\begin{lemma}[{\cite[Lemma 4.5]{Iritani-convergence}}, {\cite[Lemma 7.6]{Coates-Corti-Iritani-Tseng-hodge}}]
\label{lem:invertible}
Let $U$ be an open subset of $\C^N_\lambda$, and $a(\lambda,y,z) = \sum_{\bm,n}a_{\bm,n}(\lambda)y^\bm z^n$ be an element of $\cO^{y,z}_\lambda(U)$.
If $a_{0,0}(\lambda) \in \cO^\times_{\C^N_\lambda} (U)$, then $a(\lambda,y,z)$ is invertible in $\cO^{y,z}_\lambda(U)$.
\end{lemma}

\begin{proof}

Let $a = \sum_{\bm,n} a_{\bm,n}(\lambda) y^\bm z^n$ be an element of $\cO^{y,z}_\lambda(U)$ with $a_{0,0}(\lambda) \in \cO^\times_{\C^N_\lambda} (U)$, and let $C(\lambda)$ be a continuous function on $U$ which is not less than one such that $|a_{\bm,n}(\lambda)|\leq C(\lambda)^{|\bm|+n+1}|\bm|^n$ for any $\lambda \in U$.
Since $a_{0,0}(\lambda)$ is invertible, we can take its inverse $a^{-1}$ in the ring $\cO_{\C^N_\lambda}(U)[\![y,z]\!]$.
We may assume $a_{0,0}(\lambda) = 1$.
Then we have
\[
a^{-1} = 1 + \sum_{\bm \neq 0,n} y^\bm z^n \sum^{| \bm |}_{l = 1} (-1)^l {\sum}^\prime_{l,\bm,n} \prod^l_{i=1} a_{\bm_i,n_i} (\lambda)
\]
where the symbol $\sum'_{l,\bm,n}$ means the sum over all $\bm_1, \dots, \bm_l$ and $n_1, \dots, n_l$ such that $\sum^l_{i=1} \bm_i = \bm$ and $\sum^l_{i=1} n_i = n$.
The coefficient of $y^\bm z^n$ can be estimated as
\begin{align*}
\left| \sum^{| \bm |}_{l = 1} (-1)^l {\sum}^\prime_{l,\bm,n} \prod^l_{i=1} a_{\bm_i,n_i} (\lambda) \right| 
&\leq \sum^{| \bm |}_{l = 1} C(\lambda)^{ | \bm | + n + l } {\sum}^\prime_{l,\bm,n} \prod^l_{i=1} | \bm_i |^{n_i}	\\
&\leq C(\lambda)^{2 | \bm | + n} | \bm |^n \sum^{| \bm |}_{l = 1} {\sum}^\prime_{l,\bm,n} 1.				  
\end{align*}
It is easy to see that $\sum^{| \bm |}_{l = 1} {\sum}^\prime_{l,\bm,n} 1 \leq C'^{ | \bm | + n}$ for some constant $C'\geq1$.
This implies that $a^{-1}$ belongs to $\cO^{y,z}_\lambda(U)$.

\end{proof}

We can now state the gauge fixing theorem.

\begin{theorem}[gauge fixing theorem, {\cite[Proposition 4.8]{Iritani-convergence}}, {\cite[Theorem 7.13]{Coates-Corti-Iritani-Tseng-hodge}}]
\label{thm:gaugefixing}
Let $U$ be a neighborhood of $0\in\C^N_\lambda$, and let $\nabla$ be a logarithmic flat connection of the form
\[
\nabla = d + z^{-1} \sum_a \bfA_a \frac{d y_a}{y_a}
\]
such that $\bfA_a (\lambda,y,z) \in \Mat_n ( \cO^{y,z}_{\lambda}(U) )$ is a square matrix of size $n$ with all entries in $\cO^{y,z}_{\lambda}(U)$.
Suppose that the residue matrices $\bfA_a (0,0,z)$ are nilpotent.
Then, after shrinking $U$ if necessary, there exists a unique gauge transformation $\bfR (\lambda,y,z) \in \Mat_n ( \cO^{y,z}_{\lambda}(U) )$ with entries in $\cO^{y,z}_{\lambda}(U)$ such that $\bfR (\lambda,0,z) = \Id$ and $\bfA^\prime_a := \bfR^{-1} z y_a \parfrac{\bfR}{y_a} + \bfR^{-1} \bfA_a \bfR$ is independent of $z$ for all $a$.
In particular, $\bfA^\prime_a$ is convergent and analytic on $U \times V$ for some open neighborhood $V$ of $y = 0$.
\end{theorem}

\begin{remark}
\label{rem:gaugefixing}
Since this theorem is stated and proved in a general setting, it may be applicable to other situations as well.
We will apply this to a connection over the $(x,\sigma)$-space instead of the $y$-space in Section 6.
\end{remark}

\subsection{Estimate for $\bfM$}
\label{subsect:estimateM}
In this subsection, we estimate the matrix $\bfM$.
Recall that $\bfM$ consists of  the matrices
\[
\bfM_{\alpha,i} = T_i \left( z \parfrac{}{t} \right) \Delta (\sigma) \hcI_\alpha ( qe^t, \lambda + z \partial_\Lambda, z ) \bfL_B ( Q_\alpha, \tau_\alpha, z ), 
\]
whose entries are elements of $\Exc_\alpha(\lambda + \Lambda,z) \cdot \Ral (\!(z)\!) [\![q,Q,\tts]\!]$.
Here the exceptional term $\Exc_\alpha(\lambda,z)$, defined in Remark \ref{rem:Brown}, is 
\[
\Exc_\alpha(\lambda,z) = \prod_{j \notin \alpha} \frac{1}{ \sqrt{\auj} } e^{ ( \auj - \auj \log \auj ) / z }.
\]
Define $\tbfM_{\alpha,i}$ to be the matrix $( \Exc_\alpha(\lambda + \Lambda,z) e^{\tau_\alpha / z} )^{-1} \cdot \bfM_{\alpha,i}$.
It is easy to see that
\[
\tbfM_{\alpha,i} \in \End(H^*(B)) \otimes \cO_{\C^N_\lambda} (\cD^c) (\!(z)\!) [\![y]\!]
\]
where $\cD^c := \C^N_\lambda \setminus \cD$ and $\cD \subset \C^N_\lambda$ is the union of subspaces defined in Section \ref{subsect:equivmirror}. 
We will show that:

\begin{proposition}
\label{prop:estimateM}
All entries of $\tbfM_{\alpha,i}$ belong to the ring $\cO^{y,z,z^{-1}}_\lambda(\cD^c)$.	
\end{proposition}

First we decompose $\bfL_B$ as:
\[
\bfL_B = e^{\tau/z} \sum_{\bm,n\geq0} \bfL^B_{\bm,n} x^\bm z^{-n}.
\]
From the convergence assumption of $QH^*(B)$, we have an estimate for $\bfL^B_{\bm,n}$.

\begin{lemma}[{\cite[Lemma 4.1]{Iritani-convergence}}]
\label{lem:estimateLB}
There exist a constant $C_1>1$ such that
\[
\| \bfL^B_{\bm,n} \| \leq C^{|\bm|+n+1}_1 \frac{1}{n!}.
\]
\end{lemma}

Secondly, we estimate $\hcI_\alpha(q,\lambda,z)$.
In Section \ref{subsect:equivmirror}, $\hcI_\alpha(q,\lambda,z)$ is defined as the asymptotic expansion of the oscillatory integral $\cI_\alpha(q,\lambda,z)$ divided by $(-2 \pi z)^{(N-k)/2}$.
Write
\[
\hcI_\alpha(q,\lambda,z) = \Exc_\alpha(\lambda,z) e^{ F^\alpha (q,\lambda) / z } \sum_{n \geq 0} A^\alpha_n (q,\lambda) z^n.
\]
Note that $F^\alpha(q,\lambda)$ and $A^\alpha_n(q,\lambda)$ are holomorphic functions on the domain $\cU \subset \hcM$, which is introduced in Proposition \ref{prop:rhoalpha}.

\begin{lemma}[{\cite[Lemma 7.10]{Coates-Corti-Iritani-Tseng-hodge}}]
\label{lem:estimateSPA}
Let $f(\xi)$ and $g(\xi)$ be holomorphic functions on a domain $U \subset \C^m$, and $p \in U$ be a non-degenerate critical point of $f$.
Let
\[
f(p+\xi) = \sum_I a^f_I \xi^I, \quad g(p+\xi) = \sum_I a^g_I \xi^I
\]
be the Taylor expansions of $f$ and $g$, and
\[
\int e^{f(\xi)/z} g(\xi) d\xi_1 \cdots d\xi_n \sim (-2\pi z)^{n/2} e^{f(p)/z} \sum_{n\geq0} a_n z^n
\]
be the oscillatory integral and its formal expansion at $p$.
Then, if $C'$ is the positive constant satisfying $\max ( |a^f_I|, |a^g_I| ) \leq (C')^{|I|+1}$, the coefficients $a_n$ satisfy the following estimate
\[
|a_n| \leq C^{n+1} n!
\]
for some positive constant $C$ depending continuously on $C'$.
\end{lemma}

\begin{lemma}
\label{lem:estimateAn}
There exists a continuous function $C_2(q,\lambda)$ on $\cU$ such that
\[
| A^\alpha_n (q,\lambda) | \leq C_2 (q,\lambda)^{n + 1} n! \qquad \forall (q,\lambda) \in \cU.
\]
\end{lemma}

\begin{proof}

Apply the $(q,\lambda)$-parametrized version of Lemma \ref{lem:estimateSPA} to the oscillatory integral $\cI_\alpha (q,\lambda,z)$.

\end{proof}

\begin{lemma}
\label{lem:preservering}
Let $U$ be an open subset of $\C^N_\lambda$.
The operators $z \parfrac{}{t_i}$ and $\Delta(\sigma)$ preserve the ring $\cO^{y,z,z^{-1}}_{\lambda}(U)$.
\end{lemma}

\begin{proof}

Take $a(\lambda,y,z) = \sum_{\bm,n} a_{\bm,n}(\lambda) y^\bm z^n \in \cO^{y,z,z^{-1}}_{\lambda}(U)$.
Let $C \colon U \to [1,+\infty)$ be a continuous function satisfying $|a_{\bm,n}(\lambda)| \leq C(\lambda)^{|\bm| + |n| + 1} n!$, and let $C'\geq1$ be a constant such that $a_{\bm,n}(\lambda)=0$ if $n<-C'(|\bm| + 1)$.
It is easy to see that 
\[
z \parfrac{}{t_i} a(\lambda,y,z) 
= \sum_{\bm,n} a_{\bm,n-1}(\lambda) \bm_{i,0} y^\bm z^n
\]
again belongs to $\cO^{y,z,z^{-1}}_{\lambda}(U)$.
We will show that $\exp ( \sigma_{i,j} \partial/\partial \tau_j T_i ( z \partial/\partial t ) ) \cdot a(\lambda,y,z) \in \cO^{y,z,z^{-1}}_{\lambda}(U)$.
This is sufficient since $\Delta(\sigma)$ is the product of exponentials of this form.
Let $d = \deg T_i$ be the degree of the monomial $T_i$.
We first assume that $j=0$ or $j > r$.
Then we have
\begin{align*}
&\exp \left( \sigma_{i,j} \parfrac{}{\tau_j} T_i \left( z \parfrac{}{t} \right) \right) \sum_{\bm,n} a_{\bm,n}(\lambda) y^\bm z^n		\\
=& \sum_{\bm,n} \left[ \sum_{l\geq0} \frac{1}{l!} \frac{(\bm_{0,j}+l)!}{\bm_{0,j}!} T_i (\bm_{\bullet,0})^l a_{ \bm + l \cdot \sfe_{0,j} - l \cdot \sfe_{i,j}, n-dl } (\lambda) \right] y^\bm z^n
\end{align*}
The coefficient of $y^\bm z^n$ can be estimated as follows:
\begin{align*}
&\left| \sum_{0 \leq l \leq d^{-1} (n + C' ( | \bm | + 1 ))} \frac{1}{l!} \frac{(\bm_{0,j}+l)!}{\bm_{0,j}!} T_i (\bm_{\bullet,0})^l a_{ \bm + l \cdot \sfe_{0,j} - l \cdot \sfe_{i,j}, n-dl } (\lambda) \right|		\\
\leq& 2^{\bm_{0,j} + d^{-1} (n + C' ( | \bm | + 1 ))} 2^{1 + n + C' ( | \bm | + 1 )} \cdot \max_{l \geq 0} | \bm |^{dl} | a_{ \bm + l \cdot \sfe_{0,j} - l \cdot \sfe_{i,j}, n-dl } (\lambda) |	\\
\leq& 2^{\bm_{0,j} + d^{-1} (n + C' ( | \bm | + 1 ))} 2^{1 + n + C' ( | \bm | + 1 )} \cdot C''(\lambda)^{ | \bm | + n + 1 } n!
\end{align*}
for some continuous function $C''(\lambda)\geq1$ on $U$.
Here the last inequality follows from the following estimate:
\begin{align*}
| \bm |^{dl} | a_{ \bm + l \cdot \sfe_{0,j} - l \cdot \sfe_{i,j}, n-dl } (\lambda) |
&\leq C(\lambda)^{ | \bm | + n + 1 } dl! (n-dl)! \\
&\leq C''(\lambda)^{ | \bm | + n + 1 } n!
\end{align*}
The case $1 \leq j \leq r$ can be shown in a similar way.

\end{proof}

\begin{proof}[Proof of Proposition \ref{prop:estimateM}]

We have
\begin{align*}
\tbfM_{\alpha,i} 
&= ( \Exc_\alpha(\lambda + \Lambda,z) e^{\tau_\alpha / z} )^{-1} T_i \left( z \parfrac{}{t} \right) \Delta (\sigma) \hcI_\alpha ( qe^t, \lambda + z \partial_\Lambda, z ) \bfL_B ( Q_\alpha, \tau_\alpha, z )	\\
&= T_i \left( z \parfrac{}{t} + P^\alpha \right) \Delta' (\sigma) ( \Exc_\alpha(\lambda + \Lambda,z) e^{\tau_\alpha / z} )^{-1} \hcI_\alpha ( qe^t, \lambda + z \partial_\Lambda, z ) \bfL_B ( Q_\alpha, \tau_\alpha, z )	\\
&= T_i \left( z \parfrac{}{t} + P^\alpha \right) \Delta' (\sigma) \left[ ( \Exc_\alpha(\lambda + \Lambda,z) e^{\tau / z} )^{-1} \hcI_\alpha ( qe^t, \lambda + z \partial_\Lambda, z ) \bfL_B ( Q, \tau, z ) \right]_{Q \to Q_\alpha, \tau \to \tau_\alpha}.
\end{align*}
Here $\Delta'(\sigma)$ is defined to be $\Delta(\sigma)$ with $z\partial/\partial t_i$ and $z\partial/\partial \tau_j$ replaced by $z\partial/\partial t_i + P^\alpha_i$ and $z\partial/\partial \tau_j + \phi_j$ respectively.
Due to Lemma \ref{lem:preservering}, it is sufficient to show that all entries of the matrix
\begin{align*}
\bfN^\alpha 
&= ( \Exc_\alpha(\lambda + \Lambda,z) e^{\tau / z} )^{-1} \hcI_\alpha ( qe^t, \lambda + z \partial_\Lambda, z ) \bfL_B ( Q, \tau, z ) \\
&= e^{-\tau / z} \frac{ \Exc_\alpha(\lambda + z \partial_\Lambda,z) }{ \Exc_\alpha(\lambda + \Lambda,z) } e^{F^\alpha(qe^t,\lambda + z \partial_\Lambda)/z} \biggl[ \sum_{a\geq 0} A^\alpha_a (qe^t,\lambda + z \partial_\Lambda) z^a \biggr] \bfL_B ( Q, \tau, z )
\end{align*}
belong to the ring $\cO^{y,z,z^{-1}}_{\lambda} (\cD^c)$.
The operator $z\partial_{\Lambda_j}$ acts on $e^{\tau/z}x^\bm z^{-n}$ as
\[
z \partial_{\Lambda_j} e^{\tau/z} x^\bm z^{-n} = \left( \Lambda_j + z \Lambda_j(\bm^\prime) \right) e^{\tau/z} x^\bm z^{-n}
\]
where $\bm^\prime \in \Z^r \cong H_2(B)$ is the degree of $x^\bm z^{-n}$ in $Q$.
Therefore, $\bfN^\alpha$ is computed as 
\begin{align*}
\bfN^\alpha
&= e^{-\tau / z} \frac{ \Exc_\alpha(\lambda + z \partial_\Lambda,z) }{ \Exc_\alpha(\lambda + \Lambda,z) } e^{F^\alpha(qe^t,\lambda + z \partial_\Lambda)/z} \biggl[ \sum_{a \geq 0} A^\alpha_a (qe^t,\lambda + z \partial_\Lambda) z^a \biggr] e^{\tau/z} \sum_{\bm,n\geq0} \bfL^B_{\bm,n} x^\bm z^{-n}	\\
&= \sum_{\bm,n\geq0} \frac{ \Exc_\alpha(\lambda + \Lambda + z \Lambda(\bm^\prime), z) }{ \Exc_\alpha(\lambda + \Lambda,z) } e^{F^\alpha(qe^t,\lambda + \Lambda + z \Lambda(\bm^\prime))/z} \biggl[ \sum_{a \geq 0} A^\alpha_a (qe^t,\lambda + \Lambda + z \Lambda(\bm^\prime)) z^a \biggr] \bfL^B_{\bm,n} x^\bm z^{-n}	\\
&= \biggl( \sum_{a,\bd_A,n_A} \sum_{\bd_F,n_F} \sum_{n_E} \sum_{\bm,n} A_{a,\bd_A,n_A} (\lambda,\bm^\prime) \cdot F_{\bd_F,n_F} (\lambda,\bm^\prime) \cdot E_{n_E} (\lambda,\bm^\prime) \cdot \bfL^B_{\bm,n} \\
& \phantom{AAAAAAAAAAAAAAAAAAAAAAAAAAAA} \cdot \left. (qe^t)^{\bd_A + \bd_F} x^\bm z^{ a + n_A + n_F + n_E - n} \biggr) \right|_{\lambda \to \lambda + \Lambda}
\end{align*}
where we set
\begin{align*}
A^\alpha_a (q,\lambda + z \Lambda(\bm^\prime)) 
&= \sum_{\bd,n\geq0} A_{a,\bd,n} (\lambda,\bm^\prime) q^\bd z^n, \\
e^{F^\alpha (q,\lambda + z \Lambda(\bm^\prime)) / z} 
&= \sum_{\bd,n\geq-|\bd|} F_{\bd,n} (\lambda,\bm^\prime) q^\bd z^n, \\
\frac{ \Exc_\alpha(\lambda + z \Lambda(\bm^\prime), z) }{ \Exc_\alpha(\lambda,z) }
&= \sum_{n\geq0} E_n (\lambda,\bm^\prime) z^n.
\end{align*}
We claim that the coefficients satisfy the following estimates:
\begin{align}
\label{eq:estimateAadn}
\left\| A_{a,\bd,n} (\lambda + \Lambda,\bm^\prime) \right\| 
&\leq C_3(\lambda)^{a + |\bd| + n + | \bm^\prime | + 1} (a+n)!, \\
\label{eq:estimateFdn}
\left\| F_{\bd,n} (\lambda + \Lambda,\bm^\prime) \right\| 
&\leq C_3(\lambda)^{|\bd| + | n | + | \bm^\prime | + 1} n!, \\
\label{eq:estimateEn}
\left\| E_n (\lambda + \Lambda,\bm^\prime) \right\| 
&\leq C_3(\lambda)^{n + | \bm^\prime | + 1} n!
\end{align}
for some continuous function $C_3(\lambda) \colon \cD^c \to [1,+\infty)$.
Note that it is sufficient to show the same estimates for $| A_{a,\bd,n} (\lambda,\bm^\prime) |, | F_{\bd,n} (\lambda,\bm^\prime) |, | E_n (\lambda,\bm^\prime) |$ instead of $\| A_{a,\bd,n} (\lambda + \Lambda,\bm^\prime) \|, \| F_{\bd,n} (\lambda + \Lambda,\bm^\prime) \|, \| E_n (\lambda + \Lambda,\bm^\prime) \|$ since $\Lambda_j \in H^2(B)$ is nilpotent.

We first confirm the estimate for $A_{a,\bd,n} (\lambda,\bm^\prime)$.
From Lemma \ref{lem:estimateAn} along with the Cauchy integral formula, there exists a real continuous function $C_4(\lambda)\geq1$ such that
\[
\left| \partial^\bd_q \partial^{I}_\lambda A^\alpha_a (0,\lambda) \right| \leq C_4(\lambda)^{a + |\bd| + |I| + 1} a! \bd! I!.
\]
Taking the Taylor expansion for $A^\alpha_a$, we have
\begin{align*}
A^\alpha_a (q,\lambda + z \Lambda(\bm^\prime))
&= \sum_{\bd \in \N^k,I,I' \in \N^N} \partial^\bd_q \partial^I_\lambda A^\alpha_a (0,\lambda) \cdot \frac{q^\bd}{\bd!} \cdot \frac{\Lambda(\bm^\prime)^I}{I!} z^{|I|} \\
&= \sum_{\bd,n\geq0} q^\bd z^n \sum_{I \in \N^N ; |I| = n} \partial^\bd_q \partial^{I}_\lambda A^\alpha_a (0,\lambda) \cdot \frac{\Lambda(\bm^\prime)^I}{\bd!I!}.
\end{align*}
Let $C_5\geq1$ be a constant such that $\max_j | \Lambda_j (D) | \leq C_5 |D|$ for any $D \in \Eff (B)$.
Then we have
\begin{align*}
\left| A_{a,\bd,n} (\lambda,\bm^\prime) \right| 
&\leq 2^{N+n} \max_{I \in \N^N ; |I| = n} \left( \left| \partial^\bd_q \partial^{I}_\lambda A^\alpha_a (0,\lambda) \right| \cdot \frac{\left| \Lambda(\bm^\prime) \right|^I}{\bd!I!} \right) \\
&\leq 2^{N+n} \cdot C_4(\lambda)^{a + |\bd| + n + 1} a! \cdot C^n_5 | \bm^\prime |^n 
\end{align*}
This estimate along with the inequality $| \bm^\prime |^n \leq e^{| \bm^\prime |} n! $ shows \eqref{eq:estimateAadn}.

Similarly, by choosing a larger $C_4(\lambda)$ if necessary, we have the following estimates for the coefficients of the expansion of $F^\alpha$:
\begin{equation} 
\label{eq:estimateF}
F^\alpha (q,\lambda + z \Lambda(\bm^\prime)) = \sum_{\bd\neq0,n\geq0} F'_{\bd,n} (\lambda,\bm^\prime) q^\bd z^n, \qquad \left| F'_{\bd,n} (\lambda,\bm^\prime) \right| \leq C_4(\lambda)^{|\bd| + n + 1} | \bm^\prime |^n.
\end{equation}
Here we used $F^\alpha(0,\lambda) = 0$.
We decompose $e^{F^\alpha/z}$ into three parts as
\begin{align*}
e^{ F^\alpha(q,\lambda + z \Lambda(\bm^\prime)) / z } 
=& \exp \biggl[ \sum_{\bd\neq0,n\geq0} F'_{\bd,n} (\lambda,\bm^\prime) q^\bd z^{n-1} \biggr] \\
=& \exp \biggl[ \sum_{\bd\neq0} F'_{\bd,0} (\lambda) q^\bd z^{-1} \biggr] \cdot \exp \biggl[ \sum_{\bd\neq0} F'_{\bd,1} (\lambda,\bm^\prime) q^\bd \biggr] \\
&\cdot \exp \biggl[ \sum_{\bd\neq0,n\geq1} F'_{\bd,n+1} (\lambda,\bm^\prime) q^\bd z^n \biggr]
\end{align*}
The third exponential is expanded as 
\begin{align*}
\exp \biggl[ \sum_{\bd\neq0,n\geq1} F'_{\bd,n+1} (\lambda,\bm^\prime) q^\bd z^n \biggr]
&= \sum_{m\geq0} \frac{1}{m!} \biggl[ \sum_{\bd\neq0,n\geq1} F'_{\bd,n+1} (\lambda,\bm^\prime) q^\bd z^n \biggr]^m \\
&= 1 + \sum_{\bd\neq0,n\geq1} q^\bd z^n \left[ \sum^n_{m=1} \frac{1}{m!} {\sum}^\prime_{\bd,n,m} \prod^{m}_{i=1} F'_{\bd_i,n_i+1} (\lambda,\bm^\prime) \right]
\end{align*}
where the symbol ${\sum}^\prime_{\bd,n,m}$ means the sum over all $\bd_1,\dots,\bd_m \in \N^k$ and $n_1,\dots,n_m \in \N$ such that $n_1 \geq 1,\dots,n_m \geq 1$, $\sum^m_{i=1} \bd_i = \bd$ and $\sum^m_{i=1} n_i = n$. 
We estimate the coefficient of $q^\bd z^n$ for $\bd\neq0$ and $n \geq 1$ by using \eqref{eq:estimateF} as
\begin{align*}
\left| \sum^n_{m=1} \frac{1}{m!} {\sum}^\prime_{\bd,n,m} \prod^{m}_{i=1} F'_{\bd_i,n_i+1} (\lambda,\bm^\prime) \right|
&\leq 2^{|\bd| + (k+1)n } \sum^n_{m=1} \frac{1}{m!} C_4(\lambda)^{| \bd | + n + 2m } | \bm^\prime |^{n + m} \\
&\leq 2^{|\bd| + (k+1)n } C_4(\lambda)^{| \bd | + 3n } | \bm^\prime |^n \sum^n_{m=1} \frac{| \bm^\prime |^m}{m!}	\\
&\leq 2^{|\bd| + (k+1)n } C_4(\lambda)^{| \bd | + 3n } | \bm^\prime |^n e^{| \bm^\prime |}
\end{align*}
Hence the third exponential satisfies the estimate of the form \eqref{eq:estimateFdn}.
In the same way, we can show the same estimates for the remaining two exponentials, which shows the estimate \eqref{eq:estimateFdn}.

Finally, we confirm \eqref{eq:estimateEn}. 
We decompose $\Exc_\alpha(\lambda + z \Lambda(\bm^\prime), z) / \Exc_\alpha(\lambda, z)$ into two parts:
\begin{align*}
\frac{ \Exc_\alpha(\lambda + z \Lambda(\bm^\prime), z) }{ \Exc_\alpha(\lambda,z) }
= \left[ \prod_{j \notin \alpha} \frac{ \auj^{- \aUj(\bm^\prime)} }{ \sqrt{1+\frac{z \cdot \aUj(\bm^\prime)}{\auj}} } \right] \cdot e^{ E' (\lambda,z\Lambda(\bm^\prime)) / z }
\end{align*}
where
\begin{align*}
E' (\lambda,z\Lambda(\bm^\prime)) = \sum_{j \notin \alpha} \left[ z \cdot \aUj(\bm^\prime) - \left( \auj + z \cdot \aUj(\bm^\prime) \right) \log \left( 1 + \frac{z \cdot \aUj(\bm^\prime)}{\auj} \right) \right].
\end{align*}
The former part satisfies the estimate of the form \eqref{eq:estimateEn}.
On the other hand, we have the following expansion of $E' (\lambda,z \Lambda(\bm^\prime))$ and estimate of its coefficients:
\begin{align*}
z^{-1} E' (\lambda,z\Lambda(\bm^\prime)) = \sum_{n \geq 1} E'_n (\lambda,\bm^\prime) z^n, \qquad | E'_n (\lambda,\bm^\prime) | \leq C_4 (\lambda)^{n + 1} | \bm^\prime |^n.
\end{align*}
As in the case of $\exp (\sum_{\bd\neq0,n\geq1} F'_{\bd,n+1} (\lambda,\bm^\prime) q^\bd z^n)$, we can estimate this exponential and show \eqref{eq:estimateEn}.

Finally, we estimate $\bfN^\alpha$.
Expand $\bfN^\alpha$ as $\bfN^\alpha = \sum_{\bd,\bm,n} \bfN^\alpha_{\bm,n} (\lambda) (qe^t)^\bd x^\bm z^n$.
Then we have
\begin{align*}
\bfN^\alpha_{\bm,n} = \sum_{a\geq0} \sum_{\bd_A,\bd_F ; \bd_A + \bd_F = \bd} \sum_{n_A\geq0} \sum_{n_F \geq - |\bd_F|} \sum_{n_E\geq0} A_{a,\bd_A,n_A} \cdot F_{\bd_F,n_F} \cdot E_{n_E} \cdot \bfL^B_{\bm,a + n_A + n_F + n_E - n}.
\end{align*}
Since $\bfL^B_{\bm,n} = 0$ for $n > 0$, the effective summation range is a finite set with cardinality less than $C^{| \bd | + n + 1}_6$ for some constant $C_6$. 
From the estimates \eqref{eq:estimateAadn} \eqref{eq:estimateFdn} \eqref{eq:estimateEn}, there exist real functions $C(\lambda), C'(\lambda)$ such that
\begin{align*}
\left\| \bfN^\alpha_{\bm,n} \right\| 
&\leq C(\lambda)^{| \bd | + | \bm | + n + 1} (a+n_A)! \cdot n_F! \cdot n_E! \cdot (n - a - n_A - n_F - n_E)! \\
&\leq C'(\lambda)^{| \bd | + | \bm | + n + 1} n!.
\end{align*}
This shows $\bfN^\alpha \in \cO^{y,z,z^{-1}}_{\lambda} (\cD^c)$.
\end{proof}

\subsection{Estimate for $\AI$}

In this subsection, we will estimate the connection matrix $\AI$ of the connection $\nablaI$. 
Recall that there are two ways to construct $\AI$, that is,
\[
\AI = z \LI^{-1} d \LI = z \bfM^{-1} d \bfM.
\] 
Let $\corr{\cdot,\cdot}^E$ and $\corr{\cdot,\cdot}^B$ be the equivariant Poincar\'{e} pairings on the corresponding spaces. 
We define the pairing $\pairr{\cdot}{\cdot}$ by
\[
\pairr{\omega_1}{\omega_2} := \pair{\LI(-z)\omega_1}{\LI(z)\omega_2}^E \qquad \omega_1,\omega_2 \in H^*_T(E).
\]
As we will see below, this pairing does not contain negative powers of $z$ due to the Lagrangian property of the cone.
Furthermore, this pairing satisfies certain estimates. 

\begin{proposition}
\label{prop:pairing}
{\rm (1)} For any $\omega_1,\omega_2 \in H^*_T(E)$,
\[
\pairr{\omega_1}{\omega_2} 
= \sum_{\alpha\in F} \pair{\bfM_\alpha(-z)\omega_1}{\bfM_\alpha(z)\omega_2}^B.
\]
{\rm (2)} The pairing $\pairr{\cdot}{\cdot}$ is non-degenerate and takes values in the ring $\cO^{y,z}_\lambda(\cV)$ for some open neighborhood $\cV$ of $0 \in \C^N_\lambda$.
\end{proposition}

\begin{proof}

We see in the proof of Theorem \ref{thm:formaldecomp} that, for $\eta_1=(\eta^\alpha_1)_{\alpha\in F},\eta_2=(\eta^\alpha_2)_{\alpha\in F}\in \bigoplus_{\alpha\in F}H^*(B)$,
\[
\pair{\Xi^{-1} \bfGamma(-z)^{-1} \eta_1}{\Xi^{-1} \bfGamma(z)^{-1} \eta_2}^E = \sum_{\alpha\in F}\pair{\eta^\alpha_1}{\eta^\alpha_2}^B.
\]
Hence,
\begin{align*}
\pairr{\omega_1}{\omega_2}
&= \pair{\Xi^{-1} \bfGamma(-z)^{-1} \bfM(-z) \omega_1}{\Xi^{-1} \bfGamma(z)^{-1} \bfM(z) \omega_2}^E		\\
&= \sum_{\alpha \in F} \pair{\bfM_\alpha(-z) \omega_1}{\bfM_\alpha(z) \omega_2}^B.
\end{align*}

In order to prove (2), we first show that $\pairr{\omega_1}{\omega_2}$ is contained in the rings $\C[\lambda,z] [\![y]\!]$ and $\cO^{y,z,z^{-1}}_\lambda(\cD^c)$.
From Proposition \ref{prop:unitarity} and Proposition \ref{prop:BFforLI}, we have
\begin{align*}
\pairr{\omega_1}{\omega_2}
&=\pair{\bfL_E (-z) \hR (-z) \omega_1}{\bfL_E (z) \hR (z) \omega_2}^E		\\
&=\pair{\hR (-z) \omega_1}{\hR (z) \omega_2}^E.
\end{align*}
Since $\hR (z) \in \End (H^*(E)) \otimes \C[\lambda,z] [\![y]\!]$, $\pairr{\omega_1}{\omega_2}$ belongs to $\C[\lambda,z][\![y]\!]$.
On the other hand, it follows from (1) that
\begin{align*}
\pairr{\omega_1}{\omega_2}
&= \sum_{\alpha \in F} \pair{\Exc_\alpha(\lambda + \Lambda,-z) e^{- \tau_\alpha / z} \cdot \tbfM_\alpha(-z) \omega_1}{\Exc_\alpha(\lambda + \Lambda,z) e^{\tau_\alpha / z} \cdot \tbfM_\alpha(z) \omega_2}^B \\
&= \sum_{\alpha \in F} \pair{\frac{1}{e_T(N^\alpha)} \cdot \tbfM_\alpha(-z) \omega_1}{\tbfM_\alpha(z) \omega_2}^B.
\end{align*}
This calculation along with Proposition \ref{prop:estimateM} shows $\pairr{\omega_1}{\omega_2} \in \cO^{y,z,z^{-1}}_\lambda(\cD^c)$.

We may assume $\omega_1,\omega_2 \in H^*_T(E)$ to be homogeneous elements. 
Let 
\[
\pairr{\omega_1}{\omega_2} = \sum_{\bm,n\geq0} a_{\bm,n} (\lambda)  y^\bm z^n.
\]
Since $\pairr{\omega_1}{\omega_2} \in \C[\lambda,z][\![y]\!] \cap \cO^{y,z,z^{-1}}_\lambda(\cD^c)$, the coefficients $a_{\bm,n} (\lambda)$ are polynomials and satisfy
\[
| a_{\bm,n} (\lambda) | \leq C(\lambda)^{ | \bm | + n + 1 } n! \qquad \text{for } \lambda \in \cD^c
\]
for some real positive continuous function $C(\lambda)\geq1$ on $\cD^c$.
Since the degree of $\pairr{\omega_1}{\omega_2}$ equals $\deg \omega_1 + \deg \omega_2 - \dim E =: C_1$, we have $\deg a_{\bm,n}(\lambda) + \deg y^\bm + n = C_1$ for any $\bm,n$ with $a_{\bm,n}(\lambda) \neq 0$.
It follows from $a_{\bm,n} \in \C[\lambda]$ that $\deg a_{\bm,n}(\lambda) \geq 0$ and hence $n \leq C_1 - \deg y^\bm \leq | C_1 | + C_2 | \bm |$ for some constant $C_2 > 0$.
Combining these inequalities, we obtain 
\[
| a_{\bm,n} (\lambda) | \leq C(\lambda)^{ | \bm | + n + 1 } ( | C_1 | + C_2 | \bm | )^n \leq C'(\lambda)^{ | \bm | + n + 1 } | \bm |^n \qquad \text{for} \quad \bm \neq 0,\lambda \in \cD^c
\]
for $C'(\lambda) = ( | C_1 | + C_2 ) \cdot C(\lambda)$.
We claim that $a_{0,n}(\lambda) = 0$ for $n > 0$.
This implies that the above estimate holds for any $\bm,n$ and that $\pairr{\omega_1}{\omega_2}$ is contained in $\cO^{y,z}_\lambda(\cD^c)$.
Since $\LI(z)|_{q=Q=0} = e^{( t + \tau + \sigma ) / z} \cdot \Id$, it follows that
\[
\pairr{\omega_1}{\omega_2} |_{q=Q=0} 
= \pair{e^{- ( t + \tau + \sigma) / z} \omega_1}{e^{( t + \tau + \sigma) / z} \omega_2}^E 
= \pair{\omega_1}{\omega_2}^E.
\]
Therefore $\pairr{\omega_1}{\omega_2} |_{y=0}$ is independent of $z$, that is, $a_{0,n} = 0$ for $n>0$.
Note that this equation along with Lemma \ref{lem:invertible} implies non-degeneracy of the pairing.

Finally, we show that $\pairr{\omega_1}{\omega_2} \in \cO^{y,z}_\lambda(\cV)$ for some open neighborhood $\cV$ of $\lambda = 0$.
Choose a compact set $K \subset \cD^c$ such that the polynomially convex hull $\overline{K}^{\text{p}}$ of $K$ in $\C^N_\lambda$ contains the origin in its interior, 
and take $\cV \subset \text{int} \overline{K}^{\text{p}}$ to be a neighborhood of the origin.
Then we have the following estimate
\[
| a_{\bm,n} (\lambda) | \leq C'(\lambda)^{ | \bm | + n + 1 } | \bm |^n \qquad \text{for} \quad \lambda \in \cV,
\]
which implies $\pairr{\omega_1}{\omega_2} \in \cO^{y,z}_\lambda(\cV)$.
\end{proof}

We estimate the connection matrix $\AI$ by using this pairing.
To simplify notation, let us introduce the following temporary notation; for $0 \leq i \leq l, 0 \leq j \leq s$,  
\[
\sigma_{i,j} =
\begin{cases}
t_i		&\quad 1 \leq i \leq k, j=0, \\
\tau_j		&\quad i=0, 0 \leq j \leq s, \\
\sigma_{i,j}	&\quad \text{otherwise}.
\end{cases}
\]
This notation will be used only in the rest of this section.

\begin{proposition}
\label{prop:estimateAI}
There exists an open neighborhood $\cV \subset \C^N_\lambda$ of the origin such that all entries of $\AI_a$ belong to the ring $\cO^{y,z}_\lambda(\cV)$.
\end{proposition}

\begin{proof}
Let $\cV$ be the set introduced in the previous proposition.
Write
\[
\nablaI_\parfrac{}{\sigma_a} e_b = z^{-1} \sum_c (\AI_a)^c_b (\lambda,y,z) e_c
\]

Set $g_{ab} := \pairr{e_a}{e_b} \in \cO^{y,z}_\lambda(\cV)$.
Since the pairing $\pairr{\cdot}{\cdot}$ is perfect, the matrix $(g_{ab})$ is invertible in $\Mat((l+1)(s+1),\cO^{y,z}_\lambda(\cV))$.
We denote the entries of $(g_{ab})^{-1}$ by $g^{ab} \in \cO^{y,z}_\lambda(\cV)$. 
Then $(\AI_a)^c_b$ can be written as
\[
(\AI_a)^c_b = \pairr{e^c}{\AI_a e_b} = \sum_d g^{cd} \pairr{e_d}{\AI_a e_b}.	
\]
Therefore, it is sufficient to prove $\pairr{e_a}{\AI_b e_c} \in \cO^{y,z}_\lambda(\cV)$.
Since $\AI = z \LI^{-1} d \LI = z \bfM^{-1} d \bfM$, $\pairr{e_a}{\AI_b e_c}$ coincides with
\[
\pair{\LI(-z) e_a}{z \parfrac{\LI}{\sigma_b}(z)  e_c}^E = \sum_{\alpha \in F} \pair{\bfM_\alpha(-z) e_a}{z \parfrac{\bfM_{\alpha}}{\sigma_b}(z) e_c}^B.
\]
From Proposition \ref{prop:AI}, we can see that $\pairr{e_a}{\AI_b e_c}$ is contained in $\C[\lambda,z][\![y]\!]$.
On the other hand, from the right-hand side, we find that $\pairr{e_a}{\AI_b e_c}$ is in $\cO^{y,z,z^{-1}}_\lambda(\cD^c)$.
In the same way as the proof of Proposition \ref{prop:pairing}, we obtain the result.
\end{proof}

Expand the map $\hsigma$ and the matrix $\hR$ as follows.
\[
\hsigma = \sum_a \hsigma_a (\lambda,t,\tau,\sigma) e_a, \qquad \hR(\lambda,y,z) e_a = \sum_b \hR^b_a (\lambda,y,z) e_b.
\]
Proposition \ref{prop:BFforLI} says that $\hsigma_a (\lambda,t,\tau,\sigma) - \sigma_a \in \C[\lambda] [\![y]\!]$ and $\hR^b_a (\lambda,y,z) \in \C[\lambda,z] [\![y]\!]$.
In order to apply the gauge fixing theorem (Theorem \ref{thm:gaugefixing}) to the connection $\nablaI$, we rewrite $\nablaI$ as follows:
\begin{equation}
\label{eq:tAI}
\nabla^I = d + z^{-1} \sum_a \tbfA^I_a(\lambda,y) \frac{dy_a}{y_a}, \qquad
\tbfA^I_a :=
\begin{cases}
\bfA^I_a 	&\quad \text{if  } e_a \in H^2_T(E) \\
y_a \bfA^I_a	&\quad \text{if  } e_a \notin H^2_T(E).
\end{cases}
\end{equation}
We also introduce notation which is similar to \eqref{eq:Definition_y}:
\[
(\hy_{i,j} \mid 0 \leq i \leq l, 0 \leq j \leq s), \qquad \hy_{i,j} =
\begin{cases}
q_i e^{\hsigma_{i,0}}	&1\leq i\leq k,j=0,	\\
Q_j e^{\hsigma_{0,j}}	&i=0,1\leq j\leq r,	\\
\hsigma_{i,j}			&\text{otherwise}.	
\end{cases}
\]
It is easy to see that $\hy_a \in \C[\lambda] [\![y]\!]$.

\begin{theorem}
\label{thm:analyticity_hsigma}
{\rm (1)}
The coefficients $\hsigma_a - \sigma_a \in \C[\lambda] [\![y]\!]$ are convergent and analytic in a neighborhood of $\lambda = 0$ and $y=0$.
In particular, the map $(\lambda,y) \mapsto (\lambda,\hy)$ is biholomorphic near $(\lambda,y)=(0,0)$.

{\rm (2)}
The coefficients $\hR^b_a (\lambda,y,z)$ lie in the ring $\cO^{y,z}_\lambda (\cV)$ for some open neighborhood $\cV$ of $0 \in \C^N$.
In particular, they are formal power series in $z$ with coefficients in analytic functions of $(\lambda,y)$ defined in a uniform neighborhood of $\lambda = 0$ and $y=0$.
\end{theorem}

\begin{proof}

We first verify that the connection $\nablaI$ satisfies the assumption required by Theorem \ref{thm:gaugefixing}.
Proposition \ref{prop:estimateAI} implies that $(\tbfA^I_a)^c_b \in \cO^{y,z}_\lambda (\cV)$ for some open neighborhood $\cV$ of $0 \in \C^N_\lambda$, hence we need only to confirm that $\tbfA^I_a(0,0,z)$ are nilpotent for any $a$.
From \eqref{eq:tAI}, we see that $\tbfA^I_a(0,0,z) = 0$ if $e_a \notin H^2_T(E)$.
Since $\hR$ intertwines the connection $\nabla^I$ with $\hsigma^* \nabla^E$, it follows that
\begin{equation}
\label{eq:AI}
\sum_b \parfrac{\hsigma^b}{\sigma^a} \cdot \bfA^E_b (\lambda,\hy) = \hR \cdot \AI_a(\lambda,y,z) \cdot \hR^{-1} - z \parfrac{\hR}{\sigma^a} \cdot \hR^{-1}.
\end{equation}
From Proposition \ref{prop:BFforLI}, we have $\parfrac{\hsigma^b}{\sigma^a}|_{y=0} = \delta_{a,b}$, $\hy_a|_{y=0} = 0$ and $\hR(0,0,z) = \Id$.
Since $\bfA^E$ is the quantum connection of $E$, $\bfA^E_a(0,0)$ is equal to $e_a \cup$, the non-equivariant ordinary cup product.
If $e_a \in H^2_T(E)$, then $\parfrac{}{\sigma^a} = y_a \parfrac{}{y_a}$ and hence $\parfrac{\hR}{\sigma^a}(0,0,z) = 0$.
Therefore, when $e_a \in H^2_T(E)$ and $\lambda = y = 0$, the equation \eqref{eq:AI} becomes $\bfA^I_a(0,0,z) = e_a \cup$, which implies that $\tbfA^I_a(0,0,z)$ is nilpotent since $e_a \in H^2(E)$.

Consequently, we can apply the gauge fixing theorem to $\nablaI$.
Then $(2)$ follows immediately from the uniqueness of a gauge transformation.
Sending $1 \in H^*_T(E)$ by \eqref{eq:AI}, we have 
\[
\parfrac{\hsigma}{\sigma^a} = ( \hR \cdot \bfA^I_a \cdot \hR^{-1} ) (1) + O(z).
\]
The left-hand side is independent of $z$, and hence equals $( \hR \cdot \bfA^I_a \cdot \hR^{-1} ) (1) |_{z=0}$.
On the other hand, since all entries of the matrices $\hR$, $\bfA^I_a$ and $\hR^{-1}$ belong to $\cO^{y,z}_\lambda (\cV)$,  it follows that $( \hR \cdot \bfA^I_a \cdot \hR^{-1} ) (1) \in \cO^{y,z}_\lambda (\cV)$.
Therefore, $\parfrac{\hsigma^b}{\sigma^a}$ is analytic in $(\lambda,y)$ for all $a$ and $b$, which shows $(1)$.

\end{proof}

\begin{corollary}
\label{cor:convergence_of_QH(E)}
The big equivariant quantum cohomology $QH^*_T(E)$ has convergent structure constants.
\end{corollary}

\begin{proof}
By the discussion in the proof of Theorem \ref{thm:analyticity_hsigma}, the equation \eqref{eq:AI} becomes
\[
\sum_b \parfrac{\hsigma^b}{\sigma^a} \cdot \bfA^E_b (\lambda,\hy) = (\hR \cdot \AI_a(\lambda,y,z) \cdot \hR^{-1})|_{z=0},
\]
and this is analytic in a neighborhood of $(\lambda,y)=(0,0)$.
Since the matrix $( \parfrac{\hsigma^b}{\sigma^a} )$ is nonsingular near $(\lambda,y)=(0,0)$, all entries of $\bfA^E_b (\lambda,\hy)$, and therefore entries of $\bfA^E_b (\lambda,y)$, are analytic in a neighborhood of $(\lambda,y)=(0,0)$.
\end{proof}

\section{Convergence of $\bfR^\star$}
\label{sect:convergence2}

In this section, we will show the convergence of $\tau^\star_\alpha$ and $\bfR^\star$.
Recall that $\bfR^\star$ is the block matrix \eqref{eq:R*}.
By the proof of Proposition \ref{prop:BFforM}, $\bfR^\star$ can be computed from $\tau^\star$ and $\bfR^\star_{\alpha,0}$, which appear in the Birkhoff factorization of the matrix $\bfM_{\alpha,0}$ \eqref{eq:bfMai}:
\[
\bfM_{\alpha,0}(q,Q,\lambda,\tts,z) = \bfL_B(Q_\alpha,\tau^\star_\alpha) \bfR^\star_{\alpha,0}(\lambda,y,z).
\]

\subsection{Definitions}

We write
\[
\pi \colon \hcM \times \C^k \to \hcM \qquad (q,\lambda,t) \mapsto (qe^t,\lambda),
\]
and set $\cO^\prime$ to be the sheaf on $\hcM$ whose value on an open set $U$ is $\pi^* (\cO_{\hcM} (U)) \subset \cO_{\hcM \times \C^k} ( \pi^{-1} U )$.
We remark that, though $\cO^\prime$ is isomorphic to $\cO_{\hcM}$ as sheaves, they are different in the sense that we view sections of $\cO^\prime$ as functions over $\hcM \times \C^k$.

\begin{definition}
\label{def:rings2}
Let $U$ be an open subset of $\hcM$. 
Define $\ringA (U)$ to be the subset of the ring $\cO^\prime (U) (\!(z)\!)[\![x,\sigma]\!]$ consisting of series $\sum_{I,\bm,n}a_{I,\bm,n}(qe^t,\lambda)\sigma^I x^\bm z^n$ such that there exist a positive continuous function $C \colon U \to \R$ and a positive constant $C'$ satisfying
\begin{itemize} 
\item[(1)] $a_{I,\bm,n} (q,\lambda) = 0$ for $n<-C'(| I | + | \bm | + 1)$, and
\item[(2)] the following estimate for coeffitients:
\[
|a_{I,\bm,n}(q,\lambda)|\leq C(q,\lambda)^{| I | + | \bm | + | n | + 1} (| I | + | \bm |)^n
\]
for any $(q,\lambda) \in U$.
\end{itemize}
Here we set $(| I | + | \bm |)^n = 1$ if $| I | = | \bm | = 0$ and $n \leq 0$.
Define $\ringB (U)$ to be the intersection of the rings $\ringA (U)$ and $\pi^*(\cO_{\hcM} (U)) [\![x,\sigma,z]\!]$.
\end{definition}

Note that the estimates imposed on elements of $\ringB (U)$ are nothing but the estimates required by the gauge fixing theorem.

\begin{remark}
We impose different types of estimates on series in $\cO^{y,z,z^{-1}}_\lambda(U)$ and those in $\ringA(V)$. 
\end{remark}

\begin{lemma}
\label{lem:appendix1}
$\ringA (U)$ and $\ringB (U)$ are rings.
\end{lemma}

This lemma can be shown by the same strategy as the proof of Lemma \ref{lem:both_rings}, but it requires more complicated calculation.
A proof is given in Appendix \ref{app:1}.

\subsection{Critical Branches of Equivariant Toric Mirror}

In this subsection, we study the relative critical scheme of the equivariant phase function $W \colon Y \to \C$ appearing in Section \ref{subsect:equivmirror}.
We will follow the notation used there.
We refer the reader to \cite[Section 5.4]{Iritani-convergence}, \cite[Section 3.2]{Iritani-integral} for non-equivariant case.
\[
\xymatrix{
Y \ar[r]^{W} \ar[d]_{\pr} & \C \\
\cM &
}
\]
We fix $\alpha \in F$ and take the coordinate chart of $Y$ associated to $\alpha$ described below Proposition \ref{prop:rhoalpha}.
By this coordinate chart, we identify $Y$ with $\cM \times T_\alpha$, and regard $W(q,\lambda,\xi)$ as a function on the latter space.
We set $\Crit_{Y / \cM}(W)$ to be the relative critical scheme of $W$ with respect to $\pr \colon Y \to \cM$:
\[
\Crit_{Y / \cM}(W) = \Spec \C[q^\pm, \lambda, \{\xi^\pm_j\}_{j \notin \alpha} ] \left/ \corr{ \xi_j \parfrac{W}{\xi_j} \mathrel{}\middle|\mathrel{} j \notin \alpha} \right. .
\]

We introduce the \emph{non-degeneracy condition at infinity} \cite[1.19]{Kouchnirenko} for $N-k$ Laurent polynomials $F_1, \dots, F_{N-k} \in \C[ \, \{ \xi^\pm_j \}_{j \notin \alpha} \, ]$.

\begin{definition}
Let $F_1, \dots, F_{N-k}$ be Laurent polynomials in $\{ \xi_j \}_{j \notin \alpha}$.
Set 
\begin{align*}
F_i &= \sum_{n \in \Hom(T_\alpha,\C^\times)} a_{i,n} \xi^n, \\
S_F &= \{ n \in \Hom(T_\alpha,\C^\times) \mid a_{i,n} \neq 0 \quad 1 \leq {}^\exists i \leq N-k \}, \\
\hS_F &= ( \text{the convex hull of } S_F ) \subset \Hom(T_\alpha,\C^\times) \otimes \R.
\end{align*}
The Laurent polynomials $F_1, \dots, F_{N-k}$ are \emph{non-degenerate at infinity} if, for every proper face $\Delta$ of $\hS_F$, the Laurent polynomials $(F_i)_\Delta := \sum_{n \in \Delta} a_{i,n} \xi^n$ do not vanish at the same time on $T_\alpha$.
\end{definition}

The following proposition is an equivariant version of \cite[Proposition 5.11]{Iritani-convergence}, \cite[Proposition 3.10(ii)]{Iritani-integral}.

\begin{proposition}
There exists an open dense subset $\cM^{\operatorname{ss}}$ of $\cM$ such that
\begin{itemize}
\item $\cM^{\operatorname{ss}} \cap \{ \lambda = 0 \}$ is open and dense in $\cM \cap \{ \lambda = 0 \}$. 
\item the morphism $\pr \colon \Crit_{Y / \cM}(W) |_{\cM^{\operatorname{ss}}} \to \cM^{\operatorname{ss}}$ is finite and \'{e}tale of degree $(N-k)! \Vol (\hS_X)$ where $\hS_X$ is the convex hull of the primitive integral generators of the rays of the fan $\Sigma_X$ corresponding to $X$.
\end{itemize}
\end{proposition}

\begin{proof}

Set $F_j = \xi_j \cdot \partial W_{q,\lambda} / \partial \xi_j$ for $j \notin \alpha$.
We view $F_j$ as a family of Laurent polynomials in $\{ \xi_j \}_{j \notin \alpha}$ parametrized by $(q,\lambda) \in \cM$.
From the definition, the polyhedron $\hS_F$ coincides with $\hS_X$ and the origin $0$ of $\Hom(T_\alpha,\C^\times) \otimes \R$ belongs to the interior of $\hS_F$.
For every proper face $\Delta$ of $\hS_F$, $(F_j)_\Delta$ are independent of $\lambda$ and $(F_j)_\Delta = \xi_j \cdot \partial (W_{q,0})_\Delta / \partial \xi_j$.
Since $(W_{q,0})_\Delta$ is a Laurent polynomial in $\{ \xi_j \}_{j \notin \alpha}$, the same argument in \cite[6.3]{Kouchnirenko} implies that $\{ F_j \}_{j \notin \alpha}$ are non-degenerate at infinity for generic $q$ and every $\lambda$.
We write $\cM^\prime$ for the open dense subspace consisting of all $(q,\lambda) \in \cM$ at which $\{ F_j \}_{j \notin \alpha}$ are non-degenerate at infinity.
Due to Kouchnirenko's theorem \cite[1.18(ii)]{Kouchnirenko}, the fiber of $\Crit_{Y / \cM}(W) |_{\cM^\prime} \to \cM^\prime$ consists of $(N-k)! \Vol (\hS_X)$ points (counted with multiplicity).

It is easy to see that the coordinate ring of $\Crit_{Y / \cM}(W)$ is isomorphic to
\[
\C[ q^\pm, \lambda, \xi^\pm, \sfp_1,\dots,\sfp_k ] \left/ \corr{q_i - \prod_j \xi^{\bc_{ij}}_j, \xi_j + \lambda_j - \sum_i \bc_{ij} \sfp_i } \right. . 
\]
The isomorphism is given by $\sfp_i \mapsto q_i \cdot \partial W / \partial q_i = \sum_{n \in \alpha} (\bc^{-1}_\alpha)_{ni} ( \xi_n + \lambda_n )$.
Hence $\Crit_{Y / \cM}(W)$ is isomorphic to some (Zariski) open subset of $\Spec \C [ \lambda, \sfp ]$.
In particular, the scheme $\Crit_{Y / \cM}(W)$ is smooth.
By Sard-Beritini's theorem (generic smoothness theorem in algebraic geometry), there exists a dense subset $\cM^{\operatorname{ss}}$ of $\cM^\prime$ consisting of all regular values of $\pr \colon \Crit_{Y / \cM}(W) |_{\cM^\prime} \to \cM^\prime$.
Furthermore, the inverse function theorem implies that $\cM^{\operatorname{ss}}$ is open in $\cM^\prime$.
By construction, $\cM^{\operatorname{ss}}$ satisfies the second property.
To see the first property, we only have to show that $\cM^{\operatorname{ss}} \cap \{ \lambda = 0 \}$ is not empty.
This is true since $\Crit_{Y / \cM}(W) |_{ \{ \lambda = 0 \} }$ is also smooth and, by repeating the same argument, we can take a regular value for $\Crit_{Y / \cM}(W) |_{ \cM^\prime \cap \{ \lambda = 0 \} } \to \cM^\prime \cap \{ \lambda = 0 \}$, which is automatically a regular value for the original pr.

\end{proof}

By this proposition, we can extend the critical branch $\rho^\alpha \colon \cU \to \C^N$ and $\hcI_\alpha(q,\lambda,z)$ along a suitable path.
Take an embedded path $\gamma \colon [0,1] \to \cM^{\operatorname{ss}} \cup \cU$ satisfying $\gamma (0) \in \cU \cap \{ q = 0 \}$, $\gamma(t) \in \cM^{\operatorname{ss}}$ for $t \in (0,1]$ and $\gamma (1) \in \cM^{\operatorname{ss}} \cap \{ \lambda = 0 \}$, and choose an open and simply connected neighborhood $\cW$ of the image of $\gamma$ in $\cM^{\operatorname{ss}} \cup \cU$.
Then $\rho^\alpha$ can be extended to a critical branch over $\cW$.
We denote this critical branch also by $\rho^\alpha$.

Decompose $\hcI_\alpha(q,\lambda,z)$ as
\[
\hcI_\alpha(q,\lambda,z) 
= e^{W^\alpha (q,\lambda,\rho^\alpha) / z} \cdot \frac{ \sum_{n\geq0} a^\alpha_n (q,\lambda) z^n }{ \sqrt{ \det ( W_{ij} (\rho^\alpha) ) } } 
= e^{G^\alpha (q,\lambda) / z} \cdot \sum_{n\geq0} H^\alpha_n (q,\lambda) z^n
\]
where we set
\[
G^\alpha(q,\lambda) = W^\alpha (q,\lambda,\rho^\alpha(q,\lambda)), \quad H^\alpha_n(q,\lambda) = \frac{a^\alpha_n (q,\lambda)}{ \sqrt{ \det ( W_{ij} (\rho^\alpha) ) } }.
\]
Since $\rho^\alpha(q,\lambda)$ is an analytic function over $\cW$ and a non-degenerate critical point of $W_{q,\lambda}$, the analytic functions $G^\alpha$ and $H^\alpha_n$ are successfully extended to $\cW$.

\begin{remark}
This representation of $\hcI_\alpha(q,\lambda,z)$ is slightly different from that appearing in Remark \ref{rem:Brown}.
The functions $G^\alpha(q,\lambda)$ and $H^\alpha_n(q,\lambda)$ are multi-valued on $\cU$ whereas $F^\alpha(q,\lambda)$ and $A^\alpha_n(q,\lambda)$ are single-valued on $\cU$; also, $G^\alpha(q,\lambda)$ and $H^\alpha_n(q,\lambda)$ can be analytically continued along the path $\gamma$ to points with $\lambda = 0$ whereas $F^\alpha(q,\lambda)$ and $A^\alpha_n(q,\lambda)$ cannot be. 
\end{remark}

\subsection{Convergence of $\bfR^{\star\star}$}
\label{subsect:R**}

Recall that the matrix $\bfM_{\alpha,0}$ was obtained from $\bfL_B$ by applying the differential operator $\Delta(\sigma) \hcI_\alpha(qe^t,\lambda+z\partial_\Lambda)$.
In this subsection, we will consider the matrix $\bfL^\star$ obtained from $\bfL_B$ by applying $\Delta(\sigma) e^{G^\alpha(qe^t,\lambda+z\partial_\Lambda)/z}$, which is a factor of $\Delta(\sigma) \hcI_\alpha(qe^t,\lambda+z\partial_\Lambda)$, and introduce the associated gauge transformation $\bfR^{\star\star}$.
We will discuss the convergence of $\bfR^{\star\star}$.
This gives an intermediate step towards proving the convergence of $\bfR^\star$.

We fix an open set $\cW$ explained at the end of the previous subsection.
Define the matrix $\bfL^\star$ and the connection $\nabla^\star$ associated to $\bfL^\star$ as follows.
\begin{align*}
\bfL^\star &= \Delta(\sigma) e^{G^\alpha (qe^t,\lambda + z \partial_\Lambda) / z} \bfL_B (Q_\alpha,\tau_\alpha,z), \\
\nabla^\star &= d + z^{-1} \bfA^\star, \\
\bfA^\star &= \sum^s_{j=0} \bfA^\star_{(0,j)} d\tau_j + \sum_{(i,j) \in \sfS} \bfA^\star_{(i,j)} d\sigma_{i,j}, \\
\bfA^\star_{(0,j)} &= z (\bfL^\star)^{-1} \parfrac{\bfL^\star}{\tau_j} \qquad 0 \leq j \leq s, \\
\bfA^\star_{(i,j)} &= z (\bfL^\star)^{-1} \parfrac{\bfL^\star}{\sigma_{i,j}} \qquad (i,j) \in \sfS.
\end{align*}
We view the matrix $\tbfL^\star = \exp ( -(\tau_\alpha + G^\alpha(qe^t,\lambda)) / z ) \bfL^\star$ as an element of the ring
\[
\End ( H^*(B) ) \otimes \cO^\prime (\cW) (\!(z)\!) [\![x,\sigma]\!].
\]
Note that $\bfL^\star$ is related to $\bfM_{\alpha,0}$ as follows:
\begin{equation}
\label{eq:M=HL}
\bfM_{\alpha,0} = \Delta(\sigma) \biggl( \sum_{n\geq0} H^\alpha_n(qe^t,\lambda + z \partial_\Lambda) z^n \biggr) \Delta(-\sigma) \cdot \bfL^\star.
\end{equation}
We will study the matrix $\bfL^\star$ in the rest of this subsection, and the operator $\Delta(\sigma) ( \sum_{n\geq0} H^\alpha_n(qe^t,\lambda + z \partial_\Lambda) z^n ) \Delta(-\sigma)$ in the next subsection.
The following two propositions are analogues of Propositions \ref{prop:BFformodifiedJB} and \ref{prop:AI}.

\begin{proposition}
\label{prop:BFforL*}
There exists a unique Birkhoff factorization of the form
\[
\bfL^\star = \bfL_B(Q_\alpha, \tau^{\star\star}_\alpha, z) \bfR^{\star\star}_\alpha (\lambda, y, z)
\]
where 
\begin{gather*}
\tau^{\star\star}_\alpha = \tau^{\star\star}_\alpha (q,Q,\lambda,\tts), \qquad \tau^{\star\star}_\alpha - \tau_\alpha \in H^*(B) \otimes \cO^\prime (\cW) [\![x,\sigma]\!], \\
\bfR^{\star\star}_\alpha (\lambda, y, z) \in \End (H^*(B)) \otimes \cO^\prime (\cW) [\![x,\sigma,z]\!], \qquad \bfR^{\star\star}_\alpha|_{x=\sigma=0} = \Id.
\end{gather*}
Furthermore, $\tau^{\star\star}_\alpha$ satisfies that $(\tau^{\star\star}_\alpha - \tau_\alpha)|_{Q = \sigma = 0} = G^\alpha(qe^t,\lambda + \Lambda)$.
\end{proposition}

\begin{proposition}
\label{prop:A*}
All entries of the connection matrix $\bfA^\star_{(i,j)}$ belong to the ring $\cO^\prime (\cW) [\![x,\sigma,z]\!]$.
\end{proposition}

\begin{proof}[Proof of Propositions \ref{prop:BFforL*} and \ref{prop:A*}]

The argument is almost the same as in Propositions \ref{prop:BFformodifiedJB} and \ref{prop:AI}.
Here we only check the computation of $(\tau^{\star\star}_\alpha - \tau_\alpha)|_{Q = \sigma = 0}$.
The details are left to the reader.

When $Q = \sigma = 0$, the matrix $\tbfL^\star$ is equal to
\begin{align*}
& e^{-\tau_\alpha/z} \cdot e^{-G^\alpha(qe^t,\lambda)/z} \cdot \Delta(\sigma)|_{\sigma=0} \cdot e^{G^\alpha(qe^t,\lambda + z \partial_\Lambda)/z} \cdot \bfL_B(Q_\alpha,\tau_\alpha,z)|_{Q = 0} \\
=& e^{-\tau_\alpha/z} \cdot e^{-G^\alpha(qe^t,\lambda)/z} \cdot \Id \cdot \, e^{G^\alpha(qe^t,\lambda + z \partial_\Lambda)/z} \cdot e^{\tau_\alpha/z} \\
=& e^{-G^\alpha(qe^t,\lambda)/z} \cdot e^{G^\alpha(qe^t,\lambda + \Lambda)/z}
\end{align*}
and hence the Birkhoff factorization of $\bfL^\star$ is of the form $e^{\tau_\alpha + G^\alpha(qe^t,\lambda + \Lambda)/z} \cdot \Id$.
This shows that $(\tau^{\star\star}_\alpha - \tau_\alpha)|_{Q = \sigma = 0}$ equals $G^\alpha(qe^t,\lambda + \Lambda)$.

\end{proof}

\begin{remark}
We will later see that $\tau^{\star\star}_\alpha$ coincides with $\tau^\star_\alpha$.
\end{remark}

In order to estimate the matrix $\tbfL^\star$, we need the following lemma.

\begin{lemma}
\label{lem:Delta}
The operator $\Delta''(\sigma) := \Delta(\sigma)|_{z\parfrac{}{t} \to z\parfrac{}{t} + \parfrac{G^\alpha}{t} + P^\alpha, z\parfrac{}{\tau} \to z\parfrac{}{\tau} + \phi}$ preserves the ring $H^*_T(B) \otimes \ringA (\cW)$.
\end{lemma}

\begin{proof}

Let $a=\sum_{I,\bm,n} a_{I,\bm,n}(qe^t,\lambda) \sigma^I x^\bm z^n \in \ringA (\cW)$.
By definition, there exist a positive function $C \colon \cW \to [1,+\infty)$ and a positive constant $C'\geq1$ such that
\begin{itemize}
\item $a_{I,\bm,n}=0$ if $n < -C'(|I| + |\bm| + 1)$, and
\item $| a_{I,\bm,n} | < C(q,\lambda)^{|I| + |\bm| + |n| + 1} (|I| + |\bm|)^n$.
\end{itemize}
We fix $(i,j) \in \sfS$ and let $J$ to be the muli-degree of the monomial $T_i$, which is introduced in Section \ref{subsect:bigI}.
We prove that the series
\[
\ta := \exp \left( \sigma_{i,j} z^{-1} \left( z\parfrac{}{t} + \parfrac{G^\alpha}{t} + P^\alpha \right)^J \cdot \left( z\parfrac{}{\tau_j} + \phi_j \right) \right) \cdot a 
\]
belongs to the ring $H^*_T(B) \otimes \ringA (\cW)$.
Note that it follows from this fact that $\Delta''(\sigma) \cdot a$ belongs to $\ringA (\cW)$ since the operator $\Delta''(\sigma)$ is a product of such exponentials.
We will show it only for the case of $j=0$ or $j>r$, which is a condition equivalent to $\phi_j \notin H^2(B)$.
Note that it can be shown for the case of $1\leq j\leq r$ in a similar way.

Write $\ta = \sum \ta_{I,\bm,n}(qe^t,\lambda) \sigma^I x^\bm z^n$ and compute the coefficients $\ta_{I,\bm,n}$.
We first expand $\ta$$:$
\[
\ta = \sum^\infty_{k=0} \sum_{I,\bm,n} \frac{1}{k!} \left( z\parfrac{}{t} + \parfrac{G^\alpha}{t} + P^\alpha \right)^{k \cdot J} \left( z\parfrac{}{\tau_j} + \phi_j \right)^k a_{I,\bm,n}(qe^t,\lambda) \sigma^I \sigma^k_{i,j} x^\bm z^{n-k}.
\]
We remark that $z\parfrac{}{t}$ and $\parfrac{G^\alpha}{t}$ are non-commutative, and we can see that, for any holomorphic function $f(t)$,
\[
\left\| \left( \parfrac{}{t} + \parfrac{G^\alpha}{t} + P^\alpha \right)^J f(t) \right\| \leq \sum_{K,K';K+K'=J} \left\| \left( \parfrac{}{t} \right)^K \left[ \left( \parfrac{G^\alpha}{t} + P^\alpha \right)^{K'} f(t) \right] \right\|.
\]
By a similar observation, it follows that $\| \ta_{I,\bm,n}(qe^t,\lambda) \|$ is less than
\begin{align*}
&\sum_\sfK \frac{1}{k!} \cdot \frac{(k \cdot J)!}{K!K'!} \cdot \frac{k!}{k'!k''!} \cdot \frac{(\bm_j+k')!}{\bm_j!} \left\| \phi_j^{k''} \left(\parfrac{}{t}\right)^K \left[ \left( \parfrac{G^\alpha}{t} + P^\alpha \right)^{K'} a_{I - k \cdot e_{i,j}, \bm + k'e_j, n + k'' - |K|} \right] \right\| \\
&= \sum_\sfK \frac{1}{k''!} \cdot \frac{(k \cdot J)!}{K!K'!} \cdot \frac{(\bm_j+k')!}{\bm_j!k'!} \left\| \phi_j^{k''} \left(\parfrac{}{t}\right)^K \left[ \left( \parfrac{G^\alpha}{t} + P^\alpha \right)^{K'} a_{I - k \cdot e_{i,j}, \bm + k'e_j, n + k'' - |K|} \right] \right\|
\end{align*}
where the set $\sfK$ consists of all tuples $(k, k', k'', K, K')$ of non-negative integers such that $k\leq|I|$, $k''\leq\dim_\C B$, $k' + k'' = k$ and $K + K' = k \cdot J$.
It is easy to see that $\| \ta_{I,\bm,n}(qe^t,\lambda) \| = 0$ if $n < -C'(|I| + |\bm| + 1)$. 
Since $|\sfK|$ is finite and is upper-bounded by $(\dim_\C B + 1) \cdot 2^{|I| |J|}$, we have
\begin{multline*}
\| \ta_{I,\bm,n}(qe^t,\lambda) \| \leq (\dim_\C B + 1) \cdot 2^{|I| |J|} \cdot 2^{|I| |J|} 2^{|\bm| + |I|} \\
 \cdot \max_\sfK \left\{ \left\| \phi_j^{k''} \left(\parfrac{}{t}\right)^K \left[ \left( \parfrac{G^\alpha}{t} + P^\alpha \right)^{K'} a_{I - k \cdot e_{i,j}, \bm + k'e_j, n + k'' - |K|} \right] \right\| \right\}
\end{multline*}
We can see from the Cauchy integral formula and the estimate for $| a_{I,\bm,n} |$ that
\begin{multline*}
\left\| \phi_j^{k''} \left(\parfrac{}{t}\right)^K \left[ \left( \parfrac{G^\alpha}{t} + P^\alpha \right)^{K'} a_{I - k \cdot e_{i,j}, \bm + k'e_j, n + k'' - |K|} \right] \right\| \leq  C_1(qe^t,\lambda)^{ |I| + |\bm| + |n| + 1 }	\\
\cdot K! ( |I| + |\bm| - k'' )^{n + k'' - |K|}.
\end{multline*}
for some continuous function $C_1 \colon \cW \to [1,+\infty)$. 
It follows that $K! \leq |J|^{|I||J|} |I|^{|K|}$ from $|K| \leq |I||J|$, and 
\begin{align*}
( |I| + |\bm| - k'' )^{k''} &\leq (|I| + |\bm|)^{\dim_\C B}, \\
( |I| + |\bm| - k'' )^{n - |K|} &\leq  (\dim_\C B + 1)^{|n| + |K|} (|I| + |\bm|)^{n - |K|}
\end{align*}
from $k''\leq\dim_\C B$.
Using these inequalities, we can see that
\begin{align*}
&\max_\sfK \left\{ K! ( |I| + |\bm| - k'' )^{n + k'' - |K|} \right\} \\
\leq& |J|^{|I||J|} (|I| + |\bm|)^{\dim_\C B} (\dim_\C B + 1)^{|n| + |I||J|} \max_\sfK \left\{ |I|^{|K|} (|I| + |\bm|)^{n-|K|} \right\} \\
\leq& C_2^{|I| + |\bm| + |n| + 1} (|I| + |\bm|)^n
\end{align*}
for sufficiently large positive constant $C_2\geq1$.
Combining these estimate, we have
\begin{align*}
&\| \ta_{I,\bm,n}(qe^t,\lambda) \| \\
\leq& (\dim_\C B + 1) \cdot 2^{|I| |J|} 2^{|I| |J|} 2^{|\bm| + |I|} C_1(qe^t,\lambda)^{ |I| + |\bm| + |n| + 1 } \max_\sfK \left\{ K! ( |I| + |\bm| - k'' )^{n + k'' - |K|} \right\} \\
\leq& C_3(qe^t,\lambda)^{|I| + |\bm| + |n| + 1} (|I| + |\bm|)^n
\end{align*}
for some continuous function $C_3 \colon \cW \to [1,\infty)$.
These observations show that $\ta$ belongs to the ring $H^*_T(B) \otimes \ringA(\cW)$.

\end{proof}

\begin{proposition}
\label{prop:estimateL*}
All entries of the matrix $\tbfL^\star$ belong to the ring $\ringA (\cW)$.
\end{proposition}

\begin{proof}

Let $\bfL_B(Q,\tau,z) = e^{\tau/z} \sum_{\bm,n} \bfL^B_{\bm,n} x^\bm z^{-n}$.
We first estimate the matrix
\[
\bfL^\prime = e^{- ( \tau_\alpha + G^\alpha(qe^t,\lambda + \Lambda) ) / z} e^{G^\alpha(qe^t,\lambda + z\partial_\Lambda) / z} \bfL_B(Q_\alpha,\tau_\alpha,z).
\]
Set $\bfL^\prime = \sum_{\bm,n} \bfL^\prime_{\bm,n} (qe^t,\lambda) x^\bm z^n$.
We have
\[
\bfL^\prime |_{t=0} = \sum_{\bm,n} e^{( G^\alpha(q,\lambda + \Lambda + z \Lambda(\bm^\prime) - G^\alpha(q,\lambda + \Lambda) ) / z} \cdot \bfL^B_{\bm,n} \cdot q^{P^\alpha(\bm^\prime)}  x^\bm z^{-n}
\]
where $\bm^\prime \in \Z^r \cong H_2(B)$ is the degree of $x^\bm z^{-n}$ in $Q$.
We will estimate each term of this formula.
From Lemma \ref{lem:estimateLB}, it follows that $\| \bfL^B_{\bm,n} \| \leq C^{| \bm | + n+ 1}_1 / n!$.
The term $| q_i^{P_i^\alpha(\bm^\prime)} |$ satisfies the estimate $| q_i^{P_i^\alpha(\bm^\prime)} | \leq (\max\{ | q_i |, 1 \})^{C_2 | \bm |}$ for a sufficiently large constant $C_2\geq1$ satisfying $P^\alpha(\bm^\prime) \leq C_2 | \bm |$ for all $\bm$.
For the exponential $e^{( G^\alpha(q,\lambda + z \Lambda(\bm^\prime) - G^\alpha(q,\lambda) ) / z}$, we consider the expansion of the following form
\[
e^{( G^\alpha(q,\lambda + z \Lambda(\bm^\prime)) - G^\alpha(q,\lambda) ) / z} = \sum_{n \geq 0} G^\alpha_n (q,\lambda,\bm^\prime) z^n.
\]
We claim that the coefficients satisfy the estimate
\[
\| G^\alpha_n (q,\lambda + \Lambda,\bm^\prime) \| \leq C_3(q,\lambda)^{n + | \bm^\prime | + 1} | \bm^\prime |^n
\]
for some continuous function $C_3 \colon \cW \to [1,+\infty)$.
This follows from the calculation similar to the calculation for the estimate of $F^\alpha_{\bd,n}(\lambda,\bm^\prime)$ appearing in the proof of Proposition \ref{prop:estimateM}.
The details are left to the reader.
Combining these estimates, we obtain the estimate of $\bfL^\prime_{\bm,n}$:
\begin{align*}
\| \bfL^\prime_{\bm,n} (qe^t,\lambda) \| 
&= \left\| \sum_{n_1, n_2; n_1+n_2=n} G^\alpha_{n_1}(qe^t,\lambda,\bm') \cdot \bfL^B_{\bm,-n_2} \cdot (qe^t)^{P^\alpha(\bm')} \right\|		\\
&\leq C_4(qe^t,\lambda)^{|\bm| + |n| + 1} \sum_{\substack{n_1, n_2; n_1+n_2=n, n_1\geq0, \\ -C'(|\bm|+1)\leq n_2\leq0}} |\bm|^{n_1} \cdot n_2!	\\
&\leq C_4(qe^t,\lambda)^{|\bm| + | n | + 1} \cdot \left( C'(|\bm|+1) + 1 \right) e^{|\bm|} |\bm|^n									\\
&\leq C_5(qe^t,\lambda)^{|\bm| + |n| + 1} |\bm|^n.
\end{align*}
where $C_4 \colon \cW \to [1,+\infty)$ and $C_5 \colon \cW \to [1,+\infty)$ are sufficiently large continuous functions. 
Here we use the following inequality; $n_2! \leq e^{|\bm|} |\bm|^{n_2}$ if $n_2\leq0$.

Now we consider the matrix $\tbfL^\star$, which is related to $\bfL^\prime$ as follows.
\begin{align*}
\tbfL^\star &= e^{- ( \tau_\alpha + G^\alpha(qe^t,\lambda) ) / z} \Delta(\sigma) e^{( \tau_\alpha + G^\alpha(qe^t,\lambda) ) / z} \cdot e^{( G^\alpha(qe^t,\lambda + \Lambda) - G^\alpha(qe^t,\lambda) ) / z} \bfL^\prime \\
&= \Delta(\sigma)|_{z\parfrac{}{t} \to z\parfrac{}{t} + \parfrac{G^\alpha}{t} + P^\alpha, z\parfrac{}{\tau} \to z\parfrac{}{\tau} + \phi} \left( e^{( G^\alpha(qe^t,\lambda + \Lambda) - G^\alpha(qe^t,\lambda) ) / z} \bfL^\prime \right).
\end{align*}
The exponential $ \exp ( ( G^\alpha(qe^t,\lambda + \Lambda) - G^\alpha(qe^t,\lambda) ) / z )$ is expanded to a polynomial in $z^{-1}$ with coefficients in $\cO_\cM (\cW)$ since $\Lambda_j \in H^2(B)$ is nilpotent, and hence belongs to the ring $\cO^{x,\sigma,z,z^{-1}}_\cM (\cW)$.
From Lemma \ref{lem:Delta}, it follows that the entries of $\tbfL^\star$ belong to $\ringA(\cW)$.

\end{proof}

We define the pairing $\pairr{\cdot}{\cdot}^\star$ by
\[
\pairr{\phi_i}{\phi_j}^\star := \pair{\bfL^\star(-z) \phi_i}{\bfL^\star(z) \phi_j}^B \qquad \phi_i,\phi_j \in H^*(B)
\]

\begin{proposition}
The pairing $\pairr{\cdot}{\cdot}^\star$ is non-degenerate and takes values in the ring $\ringB (\cW)$.
\end{proposition}

\begin{proof}
Due to the previous proposition and the Lagrangian property of the cone, $\pairr{\phi_i}{\phi_j}^\star$ belongs to the ring $\ringA (\cW) \cap \cO^\prime (\cW) [\![x,\sigma,z]\!] = \ringB (\cW)$.
Since $\bfL^\star |_{Q = \sigma = 0}$ is equal to $e^{( \tau_\alpha + G^\alpha(qe^t,\lambda + \Lambda) ) / z} \cdot \Id$, it follows that
\[
\pairr{\phi_i}{\phi_j}^\star |_{Q = \sigma = 0} 
= \pair{e^{-( \tau_\alpha + G^\alpha(qe^t,\lambda + \Lambda) ) / z} \phi_i}{e^{( \tau_\alpha + G^\alpha(qe^t,\lambda + \Lambda) ) / z} \phi_j}^B
= \pair{\phi_i}{\phi_j}^B.
\]
Hence the pairing $\pairr{\cdot}{\cdot}^\star$ is non-degenerate.
\end{proof}

\begin{proposition}
All entries of $\bfA^\star_a$ belong to the ring $\ringB (\cW)$.
\end{proposition}

\begin{proof}
 
By setting $g_{ij} := \pairr{\phi_i}{\phi_j}^\star \in \ringB (\cW)$ and $g^{ij} \in \ringB (\cW)$ be the entries of $(g_{ij})^{-1}$, we can write $(\bfA^\star_a)^j_i$ as
\[
(\bfA^\star_a)^j_i = \pairr{\phi^j}{\bfA^\star_a \phi_i}^\star = \sum_k g^{jk} \pairr{\phi_k}{\bfA^\star_a \phi_i}^\star
\]
for $a \in \sfS \cup \{ (0,j) \mid 0 \leq j \leq s \}$.
Since $(\bfA^\star_a)^j_i \in \cO^\prime (\cW) [\![x,\sigma,z]\!]$ from Proposition \ref{prop:A*}, it suffices to show that $\pairr{\phi_i}{\bfA^\star_a \phi_j}^\star \in \ringA (\cW)$.
From the definition of $\bfA^\star$, we have
\[
\pairr{\phi_i}{\bfA^\star_a \phi_j}^\star = \pair{\bfL^\star(-z) \phi_i}{z \parfrac{\bfL^\star}{\sigma_a} (z) \phi_j}^B
\]
for $a \in \sfS$.
This equation and Proposition \ref{prop:estimateL*} imply that $\pairr{\phi_i}{\bfA^\star_a \phi_j}^\star \in \ringA (\cW)$.
The case $a \notin \sfS$ is exactly the same.
\end{proof}

Expand the map $\tau^{\star\star}_\alpha$ and the matrix $\bfR^{\star\star}_\alpha$ from Proposition \ref{prop:BFforL*} as follows.
\begin{equation}
\label{eq:expand**}
\tau^{\star\star}_\alpha = \sum_j \tau^{\star\star}_{\alpha,j} \phi_j, \qquad \bfR^{\star\star}_\alpha \phi_i = \sum_j (\bfR^{\star\star}_\alpha)^j_i (qe^t,\lambda,x,\sigma,z) \phi_j.
\end{equation}
In order to apply Theorem \ref{thm:gaugefixing} to the connection $\nabla^\star$, we rewrite $\nabla^\star$ as follows:
\begin{equation}
\label{eq:tA*}
\nabla^\star = d + z^{-1} \sum_{a \in \sfS^*} \tbfA^\star_a(\lambda,y,z) \frac{dy_a}{y_a}, \qquad
\tbfA^\star_a :=
\begin{cases}
\bfA^\star_a 	&\quad \text{if  } a = (0,j) \text{  with  } 1 \leq j \leq r, \\
y_a \bfA^\star_a	&\quad \text{otherwise},
\end{cases}
\end{equation}
where we set $\sfS^* = \sfS \cup \{ (0,j) \mid 0 \leq j \leq s \}$.
We also introduce the following notation:
\[
(x^{\star\star}_{\alpha,j} \mid 0 \leq j \leq s), \qquad x^{\star\star}_j =
\begin{cases}
(Q_\alpha)_j e^{\tau^{\star\star}_{\alpha,j}}	&1 \leq j \leq r,	\\
\tau^{\star\star}_{\alpha,j}			&\text{otherwise}.	
\end{cases}
\]
It is easy to see that $x^{\star\star}_j \in \cO^\prime (\cW) [\![x,\sigma]\!]$.
After shrinking the neighborhood $\cW$ of the path $\gamma$ if necessary, we have the following convergence result similar to Theorem \ref{thm:analyticity_hsigma}.

\begin{theorem}
\label{thm:analyticity_tau**}
{\rm (1)}
The coefficients $\tau^{\star\star}_{\alpha,j} - \tau_{\alpha,j} \in \cO^\prime (\cW) [\![x,\sigma]\!]$ are convergent and analytic on $\cW \times \cV$ where $\cV$ is a sufficiently small neighborhood of $(x,\sigma) = (0,0)$.

{\rm (2)}
The coefficients $(\bfR^{\star\star}_\alpha)^j_i$ lie in the ring $\ringB(\cW)$.
In particular, they are formal power series in $z$ with coefficients in analytic functions of $(qe^t,\lambda,x,\sigma)$ defined on $\cW \times \cV$ where $\cV$ is a sufficiently small neighborhood of $(x,\sigma) = (0,0)$.

\end{theorem}

\begin{proof}

We first verify that the connection $\nabla^\star$ has nilpotent residues at $x = \sigma = 0$, and then apply the gauge fixing theorem (Theorem \ref{thm:gaugefixing}); see also Remark \ref{rem:gaugefixing}.
From the definition, when $x=\sigma=0$, $\tbfA^\star_a = 0$ for $a \in \sfS$ and $a = (0,j)$ with $j = 0$ or $r +1 \leq j \leq s$.
Since $\bfR^{\star\star}_\alpha$ intertwines the connection $\nabla^\star$ with $(\tau^{\star\star}_\alpha)^* \nabla^B$, we have
\begin{equation}
\label{eq:tbfA*}
\sum_{0 \leq j \leq s} \parfrac{\tau^{\star\star}_{\alpha,j}}{\tau_i} \cdot \bfA^B_j (x^{\star\star}_\alpha) = \bfR^{\star\star}_\alpha \cdot \bfA^\star_{(0,i)}(qe^t,\lambda,x,\sigma,z) \cdot (\bfR^{\star\star}_\alpha)^{-1} - z \parfrac{\bfR^{\star\star}_\alpha}{\tau_i} \cdot (\bfR^{\star\star}_\alpha)^{-1}
\end{equation} 
for any $1 \leq i \leq r$.
From Proposition \ref{prop:BFforL*}, it follows that $\bfR^{\star\star}_\alpha|_{x = \sigma = 0} = \Id$ and
\begin{equation}
\label{eq:Q=sigma=0}
\left.\parfrac{\tau^{\star\star}_{\alpha,j}}{\tau_i}\right|_{Q = \sigma = 0} = \delta_{i,j}, \qquad x^{\star\star}_{\alpha,i}|_{Q = \sigma = 0} = 0, \qquad \bfA^B_j (x^{\star\star})|_{Q = \sigma = 0} = \phi_j \cup
\end{equation}
for $1 \leq i \leq r$ and $0 \leq j \leq s$.
Combining these along with \eqref{eq:tbfA*}, we obtain $\bfA^\star_{(0,i)} = \phi_i \cup$ for $1 \leq i \leq r$ when $x = \sigma = 0$, which is nilpotent.

We can now apply Theorem \ref{thm:gaugefixing} to the connection $\nabla^\star$.
As a result, we can say that, for any $p \in \cW$, there exists an open neighborhood $U_p$ of $p$ such that $(\bfR^{\star\star}_\alpha)^j_i$ belongs to $\ringB (U_p)$.
This implies that $(\bfR^{\star\star}_\alpha)^j_i \in \ringB (\cW)$.
Sending $1 \in H^*(B)$ by \eqref{eq:tbfA*}, we have
\[
\parfrac{\tau^{\star\star}_\alpha}{\tau_i} = [ \bfR^{\star\star}_\alpha \cdot \bfA^\star_{(0,i)} \cdot (\bfR^{\star\star}_\alpha)^{-1} ] (1) |_{z = 0}.
\]
Since all entries of the matrices $\bfR^{\star\star}_\alpha, \bfA^\star_{(0,i)}$ and  $(\bfR^{\star\star}_\alpha)^{-1}$ belong to  the ring $\ringB (\cW)$, the right-hand side also does.
Therefore, $\parfrac{\tau^{\star\star}_{\alpha,j}}{\tau_i}$ belongs to the ring $\ringB (\cW) \cap \cO^\prime (\cW) [\![x,\sigma]\!]$.
This shows $(1)$.

\end{proof}

\subsection{Convergence of $\bfR^\star$}

In this subsection, we study the operator $\Delta(\sigma) ( \sum_{n\geq0} H^\alpha_n(qe^t,\lambda + z \partial_\Lambda) z^n ) \Delta(-\sigma)$ and compute how this operator affects $\tau^{\star\star}_\alpha$ and $\bfR^{\star\star}_\alpha$.
As a result, we will derive the convergence of $\tau^\star_\alpha$ and $\bfR^\star_\alpha$.

We set
\begin{align*}
\sfs &= \sum_{(i,j) \in \sfS} \sigma_{i,j} \cdot T_i \left(z \parfrac{}{t} \right) \cdot \parfrac{}{\tau_j}, \\
 \cH^\alpha \left( qe^t,\lambda,\sigma,z,z \parfrac{}{t},z \parfrac{}{\tau} \right) &= e^\sfs \cdot \biggl[\sum_{n \geq 0} H^\alpha_n(qe^t,\lambda + z\partial_\Lambda) z^n \biggr] \cdot e^{-\sfs}.
\end{align*}
The operator $\cH^\alpha$ is interpreted as an element of the non-commutative ring $\cO^\prime (\cW) \corr{z \parfrac{}{t}} (\!(z)\!) [\![\sigma,z \parfrac{}{\tau}]\!]$ (later we see that $\cH^\alpha$ in fact belongs to $\cO^\prime (\cW) \corr{z \parfrac{}{t}} [\![ \sigma, z, z \parfrac{}{\tau} ]\!]$).
We will see that the conjugation by $e^\sfs$ is equivalent to \emph{shifting each $t_i$ by $[\sfs, t_i]$}.
Note that $[\sfs, t_i]$ may contain differential operators, and does not commute with $t$ in general.
We first explain the meaning of shifts by non-commutative elements.

Let $f(x_1, \dots, x_n)$ be a holomorphic function on $U \subset \C^n$, and $\sfa_1, \dots, \sfa_n$ be elements of $\C [ z \parfrac{}{x} ]$.
We assume the following two conditions:
\begin{align}
\label{eq:assumption_for_shifting1}
\text{The constant term of $\sfa_i$ as a polynomial in $z \parfrac{}{x}$ is $0$.} 
\end{align}
\begin{equation}
\label{eq:assumption_for_shifting2}
[ \sfa_i, \sfa_j ] = 0, \quad [ \sfa_i, x_j ] = [ \sfa_j, x_i] \qquad \text{for any  } i,j
\end{equation}
Note that the latter condition guarantees that $[ x_i + \sfa_i, x_j + \sfa_j ] = 0$. 
We define an infinite-order differential operator $f(x + \sfa)$ as follows.
Let $p \in U$ and take the Taylor expansion of $f(x)$ at $p \in U$:
\[
f(x) = \sum_{J \in \N^n} \frac{1}{J!} \parfrac{ {}^J f }{x^J} (p) (x-p)^J.  
\]
We then replace each $(x_i - p_i)$ in this formula by $(x_i - p_i + \sfa_i)$ and rewrite it in \emph{normal ordering}; the normal ordering is the ordering such that differential operators always appear to the right of functions:
\[
\sum_{J \in \N^n} \frac{1}{J!} \parfrac{ {}^J f }{x^J} (p) (x-p+\sfa)^J = \sum_{m \geq 0} \sum_{I \in \N^n} \left[ \sum_{J \in \N^n} \frac{1}{J!} \parfrac{ {}^J f }{x^J} (p) \sum_{K \in \N^n} c^J_{m,I,K} \cdot (x - p)^K \right] z^m \left( z \parfrac{}{x} \right)^I
\]
where $c^J_{m,I,K}$ is given by the coefficient of $(x-p+\sfa)^J$:
\begin{align}
\label{eq:polynomial_shift}
(x-p+\sfa)^J = \sum_{m \geq 0} \sum_{I \in \N^n} \sum_{K \in \N^n} c^J_{m,I,K} \cdot (x - p)^K z^m \left( z \parfrac{}{x} \right)^I
\end{align}
Note that $c^J_{m,I,K}$ do not depend on $p \in U$.
Since $( x_i + \sfa_i )$'s are commutative, the left-hand side of the formula is unambiguous regardless of the order of the product $(x-p)^I$.
Finally, we define $f(x + \sfa)$ by substituting $p = x$ in the above formula:
\[
f(x + \sfa) := \sum_{m \geq 0} \sum_{I \in \N^n} \left[ \sum_{J \in \N^n} \frac{c^J_{m,I,0}}{J!} \parfrac{ {}^J f }{x^J} (x)  \right] z^m \left( z \parfrac{}{x} \right)^I. 
\]
From the condition \eqref{eq:assumption_for_shifting1}, for fixed $m$ and $I$, $c^J_{m,I,0}$ equals $0$ when $| J |$ is sufficiently large. 
Hence the coefficient of $z^m (z \partial/\partial x)^I$ is a finite sum and becomes a holomorphic function on $U$. 

\begin{remark}
\label{rem:shifting}
(1)
The conditions \eqref{eq:assumption_for_shifting1} \eqref{eq:assumption_for_shifting2} are written in the case of shifting with respect to $x$.
The shifting we will consider is, however, for variables that are different from $x$.
Thus we need to replace the variables in \eqref{eq:assumption_for_shifting1} \eqref{eq:assumption_for_shifting2} appropriately when we refer to these conditions. 

(2)
More generally, we can define shifting by $\sfa_i \in \cO_{\C^n} (U) \corr{ z \parfrac{}{x} }$ without the assumption \eqref{eq:assumption_for_shifting1} in a similar way.
In this case, we need to shrink the domain $U$.
We let $\sfa_{0,i} (x) \in \cO_{\C^n} (U)$ be the constant term of $\sfa_i$ with respect to $z$.
If we take an open set $V \subset U$ so that $\phi (V) \subset U$ where the map $\phi$ is given by
\[
\phi \colon U \to \C^n \qquad x \mapsto ( x_1 + \sfa_{0,1}(x), \dots x_n + \sfa_{0,n}(x) ),
\]
then we can define the differential operator $f(x+\sfa)$ over $V$.
We can find such $V$ whenever there exists $p \in U$ such that $(\sfa_{0,1} (p),\dots,\sfa_{0,n} (p)) = 0$.

(3)
In the following, we will deal with shifting by more general elements (in the sense that $\sfa_i$ contains other variables commuting with $z \parfrac{}{x}$), but we do not repeat the definition since it should be clear from the context.

\end{remark}

We now return to examine $\cH^\alpha$.
Let $\sfa_i := [ \sfs, t_i ] \in \C [ \sigma, z \parfrac{}{t}, z \parfrac{}{\tau} ]$.
Since $\sfa_i$'s satisfy the condition \eqref{eq:assumption_for_shifting1} \eqref{eq:assumption_for_shifting2} (for the variables $t$), we can define $f(qe^{t + \sfa},\lambda)$ for $f(qe^t,\lambda) \in \cO^\prime (\cW)$ in a similar way.
Here we work in the $\sigma$-adic topology, and regard $f(qe^{t + \sfa},\lambda)$ as an element of the ring $\cO^\prime (\cW) \corr{ z \parfrac{}{t} } [\![ \sigma, z, z \parfrac{}{\tau} ]\!]$.
In this setting, we can state the equivalence between the conjugation by $e^\sfs$ and the shifting $t_i$'s by $\sfa_i$'s as follows:

\begin{proposition}
\label{prop:conjugation}
Let $f(qe^t,\lambda) \in \cO^\prime (\cW)$.
Then the equality
\[
e^\sfs \cdot f(qe^t,\lambda) \cdot e^{-\sfs} = f(qe^{t + \sfa},\lambda)
\]
holds in the space $\cO^\prime (\cW) \corr{z \parfrac{}{t}} [\![ \sigma, z, z \parfrac{}{\tau} ]\!]$.
\end{proposition}

This is an immediate consequence of the following lemma.

\begin{lemma}
\[
e^\sfs \cdot t^n_i \cdot e^{-\sfs} = ( t_i + \sfa_i )^n
\]
\end{lemma}

\begin{proof}

It suffices to prove the equality when $n=1$.
Since $\sfs$ commutes with $\sfa_i$, we have
\[
\sfs^n t_i = t_i \sfs^n + n \sfs^{n-1} \sfa_i,
\] 
and hence
\[
e^\sfs \cdot t_i = \sum_{n \geq 0} \frac{1}{n!} \sfs^n t_i = t_i \cdot e^\sfs + \sum_{n \geq 1} \frac{1}{(n-1)!} \sfs^{n-1} \sfa_i = ( t_i + \sfa_i ) \cdot e^\sfs.
\]
\end{proof}
 
Applying Proposition \ref{prop:conjugation} to the operator $\cH^\alpha$, we obtain
\begin{equation}
\label{eq:cHa}
\cH^\alpha \left( qe^t,\lambda,\sigma,z,z \parfrac{}{t},z \parfrac{}{\tau} \right) = \sum_{n \geq 0} H^\alpha_n (qe^{t + \sfa}, \lambda + z \partial_\Lambda) z^n.
\end{equation}
In particular, $\cH^\alpha$ belongs to $\cO^\prime (\cW) \corr{z \parfrac{}{t}} [\![ \sigma, z, z \parfrac{}{\tau} ]\!]$.

We now consider the effect of $\cH^\alpha$ on $\tau^{\star\star}_\alpha$ and $\bfR^{\star\star}_\alpha$.
Recall that $\tau^\star_\alpha$ and $\bfR^\star_\alpha$ appeared in the Birkhoff factorization of $\bfM_{\alpha,0}$ (Proposition \ref{prop:BFformodifiedJB}), which is related to $\bfL^\star$ as \eqref{eq:M=HL}, that is,
\[
\bfM_{\alpha,0} = \cH^\alpha \left( qe^t,\lambda,\sigma,z,z \parfrac{}{t},z \parfrac{}{\tau} \right) \cdot \bfL^\star.
\]
By using the Birkhoff factorization of $\bfL^\star$ (Proposition \ref{prop:BFforL*}) and the property of the fundamental solution \eqref{eq:fund_sol}, the right-hand side can be written as
\[
\cH^\alpha \left( qe^t,\lambda,\sigma,z,z \parfrac{}{t},z \parfrac{}{\tau} \right) \cdot \bfL_B (Q_\alpha,\tau^{\star\star}_\alpha,z) \bfR^{\star\star}_\alpha = \bfL_B (Q_\alpha,\tau^{\star\star}_\alpha,z) \cdot \biggl[ \sum_{n \geq 0} H^\alpha_n ( qe^{t + \tsfa}, \lambda + \tsfb ) z^n \biggr] \bfR^{\star\star}_\alpha
\]
where $\tsfa_i = \sfa_i |_{z \partial \mapsto ((\tau^{\star\star}_\alpha)^* \nabla^B)_{z \partial}}$ and $\tsfb_j = ((\tau^{\star\star}_\alpha)^* \nabla^B)_{z \partial_{\Lambda_j}}$.
Here we use the equation \eqref{eq:cHa} and shifting explained in Remark \ref{rem:shifting} (2).
Since $\cH^\alpha$ belongs to $\cO^\prime (\cW) \corr{ z \parfrac{}{t} } [\![ \sigma,z,z \parfrac{}{\tau} ]\!]$ and in particular contains no negative powers of $z$, this also gives the Birkhoff factorization of $\bfM_{\alpha,0}$.
From the uniqueness of the factorization, $\tau^\star_\alpha$ and $\bfR^\star_\alpha$ can be written as
\begin{equation}
\label{eq:R*=HR**}
\tau^\star_\alpha = \tau^{\star\star}_\alpha, \qquad \bfR^\star_\alpha = \biggl[ \sum_{n \geq 0} H^\alpha_n (qe^{t + \tsfa}, \lambda + \tsfb ) z^n \biggr] \cdot \bfR^{\star\star}_\alpha.
\end{equation}
We expand $\tau^\star_\alpha$ and $\bfR^\star_\alpha$ in the same way as we did for $\tau^{\star\star}_\alpha$ and $\bfR^{\star\star}_\alpha$ (see \eqref{eq:expand**}):
\[
\tau^\star_\alpha = \sum_j \tau^\star_{\alpha,j} \phi_j, \qquad \bfR^\star_\alpha \phi_i = \sum_j (\bfR^\star_\alpha)^j_i(qe^t,\lambda,x,\sigma,z) \phi_j. 
\]
Then we can state the convergence result for $\tau^\star_\alpha$ and $\bfR^\star_\alpha$.

\begin{theorem}
\label{thm:analyticity_star}
{\rm (1)}
The coefficients $\tau^\star_{\alpha,j} - \tau_{\alpha,j} \in \cO^\prime (\cW) [\![x,\sigma]\!]$ are convergent and analytic on $\cW \times \cV$ where $\cV$ is a sufficiently small neighborhood of $(x,\sigma) = (0,0)$.

{\rm (2)}
The coefficients $(\bfR^\star_\alpha)^j_i$ lie in the ring $\cO^\prime (\cW \times \cV) [\![ z ]\!]$.
\end{theorem}

\begin{proof}

Since $\tau^\star_\alpha = \tau^{\star\star}_\alpha$, $(1)$ follows from Theorem \ref{thm:analyticity_tau**} $(1)$.
By Theorem \ref{thm:analyticity_tau**} $(2)$, $\bfR^{\star\star}_\alpha$ belongs to $\End (H^*(B)) \otimes \cO^\prime (\cW \times \cV) [\![ z ]\!]$.
On the other hand, the operator $\sum_{n \geq 0} H^\alpha_n ( qe^{t + \tsfa}, \lambda + \tsfb ) z^n$ belongs to the space $\End (H^*(B)) \otimes \cO^\prime (\cW \times \cV)  [\![ z, z \parfrac{}{t}, z \parfrac{}{\tau} ]\!]$ when written in normal ordering.
These observations imply that $\bfR^\star_\alpha \in \End (H^*(B)) \otimes \cO^\prime (\cW \times \cV) [\![ z ]\!]$.
\end{proof}

\subsection{Analytic Decompositions}

Combining Theorem \ref{thm:analyticity_hsigma} and Theorem \ref{thm:analyticity_star}, we will prove the main result.
Let $\hU \subset \hcM = \Spec \C[ \lambda, q ]$ and $\hV \subset \Spec \C [ Q, \{\hsigma_a\}_{a\in\sfS} ]$ be analytic open neighborhoods of the origins such that the equivariant quantum product $\star_E$ converges over $\hU \times \hV$; this is possible thanks to Theorem \ref{thm:analyticity_hsigma}.
Recall that the mirror map $(\lambda,y)\mapsto(\lambda,\hy)$ in Theorem \ref{thm:analyticity_hsigma} is biholomorphic.
By setting $t=0$ and $\tau=0$, the mirror map can be also regarded as the isomorphism from an open neighborhood $W_0$ of $0 \in \Spec\C[\lambda,q,Q,\sigma]$ to $\hU \times \hV$, which we will write as $\Phi$.
Let $\pi_1$ be the projection map $\hU \times \hV \to \hU$, and $\hU'$ be an image by the map $\pi_1\circ\Phi$ of $( W_0 \cap ((Q,\sigma)=0) ) \cap (\cM^{\operatorname{ss}} \cup \cU)$.
Then, from Theorem \ref{thm:analyticity_star}, there exists an open neighborhood $W$ of $\hU' \times \{0\}$ in $\hU \times \hV$ such that the map $\tau^\star_\alpha$ and the matrix $\bfR^\star$ can be extended analytically over $W$. 
Denote by $\widetilde{W}$ the universal cover of $W$.

\begin{theorem}
\label{thm:decompQDM}
The isomorphism in Theorem \ref{thm:formaldecomp} can be extended analytically over $\widetilde{W}$$:$ 
\[
\chbfR \colon \oanQDMT(E) \cong \bigoplus_{\alpha \in F} \chtau_\alpha^* \oanQDMT(B)
\]
Furthermore, we can take the non-equivariant limit of this isomorphism $:$
\[
\chbfR |_{\lambda = 0} \colon \oanQDM (E) \cong \bigoplus_{\alpha \in F} \chtau_\alpha^* \oanQDM (B)
\]
\end{theorem}

\begin{corollary}
\label{cor:decompQH}
Let $J_{\chtau}$ be the Jacobian matrix of $\chtau$.

{\rm (1)}
For $(\lambda,q,Q,\hsigma) \in \widetilde{W}$, the map $J_{\chtau}$ gives the ring isomorphism
\[
J_{\chtau} \colon QH^*_T(E)_{\lambda,q,Q,\hsigma} \cong \bigoplus_{\alpha \in F} QH^*_T(B)_{\lambda,Q_\alpha,\chtau_\alpha}.
\]

{\rm (2)}
For $(0,q,Q,\hsigma) \in \widetilde{W} \cap ( \olambda=0 )$, the map $J_{\chtau}$ gives the ring isomorphism
\[
J_{\chtau} |_{\lambda = 0} \colon QH^*(E)_{q,Q,\hsigma} \cong \bigoplus_{\alpha \in F} QH^*(B)_{Q_\alpha,\chtau_\alpha}.
\]

{\rm (3)}
The map $\chtau$ gives an local isomorphism between $QH^*(E)$ and $\bigoplus_{\alpha \in F} \chtau^*_\alpha QH^*(B)$ as $F$-manifolds with Euler vector fields.
\end{corollary}

\begin{remark}
Even if the fiber $X$ is not compact, our argument still works in some situation except for the argument that takes the non-equivariant limit.
We assume that the total space $E$ is semiprojective and the base space $B$ is compact.
We can define the equivariant Gromov-Witten invariants for $E$ via the virtual localization theorem.
Since $E$ is semiprojective, the evaluation maps $\ev_i$ for the moduli space $\overline{\cM}_{0,n}(E,\cD)$ become proper and hence we can consider the non-equivariant limit of $\QDMT(E)$ and $QH^*_T(E)$ \cite{Coates-Corti-Iritani-Tseng-hodge}.
Brown's mirror theorem holds in this setup, with the same proof.
The argument up to Section \ref{sect:convergence1} holds true and in particular the structure constants of $QH^*_T(E)$ are convergent.
The discussion in Section \ref{sect:convergence2} is true for non-compact $E$ if we restrict all open sets appearing there to the set $(\lambda \neq 0)$. 
As a result, we can obtain the decompositions of $\oanQDMT(E)$ and $QH^*_T(E)$ similar to Theorem \ref{thm:decompQDM} and Corollary \ref{cor:decompQH}.
These decompositions do not admit non-equivariant limits in general. 
\end{remark}

\subsection{Examples}

Finally, we give three examples of the analytic decomposition of the small quantum cohomology algebra $QH^*_T(E)$.
These examples may be helpful to understand our main theorem and why the decomposition for $QH^*_T(E)$ becomes multi-valued.

\begin{example}
We first consider the case that the base space $B$ is just a point.
In this case, the total space $E = X$ is a toric variety and the torus action on $E$ is a usual one.
Quantum cohomology of toric varieties has already been studied extensively by many reserchers. 
We study it from the perspective of the decomposition.

Let $X (=E)$ be a projective space $\PP^{N-1}$.
Then the small quantum cohomology can be described as
\[
QH^*_T(E) \cong \C[ p, \lambda, q ] \left/ \biggl( \prod^N_{j=1} (p - \lambda_j) - q \biggr) \right. 
\]
where $p$ corresponds to a standard generator of $H^2(\PP^{N-1})$.
If the point $(\lambda,q)$ does not belong to the discriminant locus of the equation $f(x)=\prod^N_{j=1} (x - \lambda_j) - q = 0$, the ring $QH^*_T(E)_{\lambda,q}$ is isomorphic to $\C^N$, that is, the direct sum of $QH^*(\pt)$.
This decomposition corresponds to the semisimplicity of the quantum cohomology. 
It is known that an arbitrary smooth projective toric manifold has generically semisimple quantum cohomology \cite[Corollary 5.12]{Iritani-convergence}.
\end{example}

\begin{example}
Let $B = \PP^n$, $X = \PP^{N-1}$ and $V$ be a vector bundle $\cO_{\PP^n}^{\oplus N -1} \oplus \cO_{\PP^n} (-1)$ over $B$.
By definition, the total space $E$ is a projectivized bundle of $V$, which is a $\PP^{N-1}$-bundle over $\PP^n$. 
(Note that $E$ is isomorphic to a blow up of $\PP^{n+N-1}$ along $\PP^{N-2}$.)
We remark that $E$ is itself a toric manifold, and the torus $T=(\C^\times)^N$ acts on $E$ fiberwise.

The $T$-equivariant cohomology of $E$ can be described as follows:
\[
H^*_T(E) \cong \C[ \phi, P, \lambda ] \left/ \biggl( \phi^{n+1}, \prod^N_{j=1} U_j \biggr) \right.
\]
where $\phi \in H^2(B)$ be a standard generator, $P \in H^2_T(E)$ be the class defined in Section 3 and $U_1, \dots, U_N \in H^2_T(E)$ be the toric divisors:
\begin{equation*}
U_j =
\begin{cases}
P - \lambda_j		& 1\leq j\leq N-1,	\\
P - \phi - \lambda_N	& j = N.
\end{cases}
\end{equation*}
Since Brown's $I$-function for $E$ is equal to the $J$-function, we can use $I_E(t,\tau,z)$ to compute quantum differential equations and the quantum cohomology explicitly.
We have
\begin{align*}
QH^*_T(B) &\cong \C[ \phi, \lambda, Q ] \left/ (\phi^{n+1} - Q) \right. ,	\\
QH^*_T(E) &\cong \C[ \phi, P, \lambda, q, Q ] \left/ \biggl( \phi^{n+1} - Q \cdot U_N, q - \prod^N_{j=1} U_j \biggr) \right. .
\end{align*}

We focus on the ring $QH^*_T(E)_{\lambda, q, Q}$.
We first consider the case that $Q=0$.
For generic $(\lambda,q)$, the equation $\prod^N_{j=1} U_j = q$ for $P$ has $N$ solutions over the ring $H^*(B)$ whose constant terms are different from each other. 
Then from the Chinese remainder theorem, we have
\[
QH^*_T(E)_{\lambda,q,0} \cong \bigoplus^N_{j=1} H^*(\PP^n).
\]
We now proceed to the case where $0 < |Q| \ll 1$.
The above decomposition can be extended to a decomposition over the formal power series ring $\C[\![Q]\!]$, and we have a map $\chtau$ satisfying
\[
QH^*_T(E)_{\lambda,q,Q} \cong \bigoplus^N_{j=1} QH^*(\PP^n)_{Q_j,\chtau_j}.
\]
A computation of the map $\chtau$ for $(n,N)=(1,2)$ will be given in the next example.
From Theorem \ref{thm:analyticity_star}, we know that the formal power series $\chtau$ is convergent near $Q=0$.
The figure below illustrates how $\Spec(QH^*_T(E)_{\lambda,q,Q})$ is deformed in $\Spec\C[\phi,P]$ for the case of $(n,N)=(2,5)$.

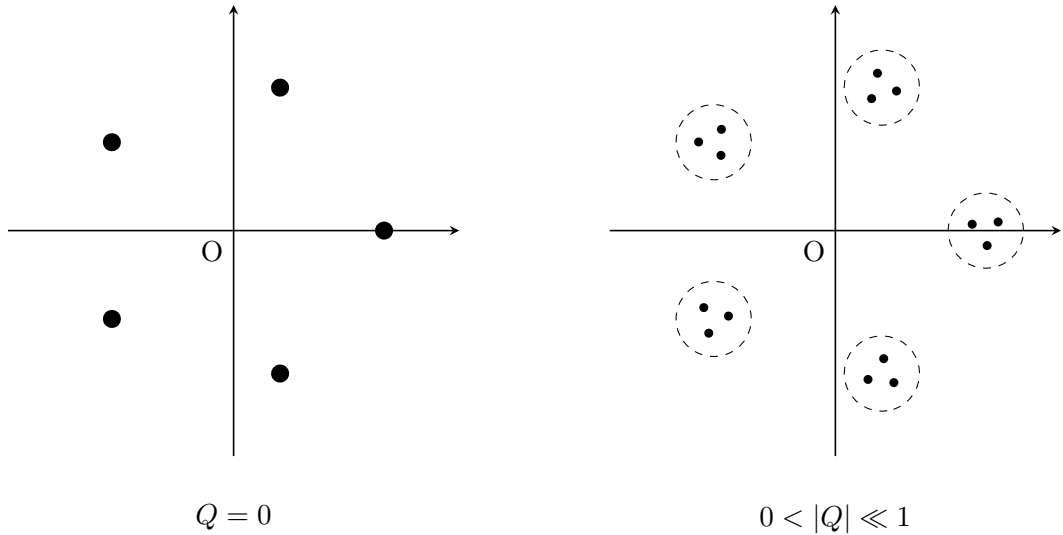
\begin{figure}[h]
\centering
\begin{tikzpicture} 
\coordinate[label=below left:O] (O) at (0,0); 
\coordinate (XS) at (-3,0); 
\coordinate (XL) at (3,0); 
\coordinate (YS) at (0,-3); 
\coordinate (YL) at (0,3); 

\draw[semithick,->,>=stealth] (XS)--(XL); 
\draw[semithick,->,>=stealth] (YS)--(YL); 

\foreach \n in {0,72,144,216,288}{
	\fill[black,rotate=\n] (2,0) circle (0.12);
};

\draw (0,-3.5) node[below] {$Q=0$};

\begin{scope}[shift={(8,0)}]
\coordinate[label=below left:O] (O) at (0,0); 
\coordinate (XS) at (-3,0); 
\coordinate (XL) at (3,0); 
\coordinate (YS) at (0,-3); 
\coordinate (YL) at (0,3); 

\draw[semithick,->,>=stealth] (XS)--(XL); 
\draw[semithick,->,>=stealth] (YS)--(YL); 

\foreach \n in {0,72,144,216,288}{
	\begin{scope}[rotate=\n]
		\draw[dashed] (2,0) circle (0.5); 
		\foreach \m in {0,120,240}{
			\fill[black] (2,0)+(\m+35:0.2) circle (0.06); 
		}
	\end{scope}
}

\draw (0,-3.5) node[below] {$0<|Q|\ll 1$};

\end{scope}

\end{tikzpicture}

\caption{The deformation of the decomposition of $\Spec (QH^*(\PP(\cO_{\PP^2}^{\oplus 4}\oplus\cO_{\PP^2}(-1)))_{\lambda,q,Q})$.}

\end{figure}

\end{example}

\begin{example}
Let $E$ be the toric bundle considered in the previous example with $(n,N)$ set to $(1,2)$.
(This is the Hirzebruch surface $F_1=\PP(\cO_{\PP^1}\oplus\cO_{\PP^1}(-1))$.)
We want to describe explicitly the map $\chtau$ and the decomposition of $QH^*_T(E)$.
For simplicity, we will perform all calculations over $K = \C[Q]/(Q^2)$ in this example.
Let $D$ be a discriminant locus of $(x-\lambda_1)(x-\lambda_2)-q=0$, that is,
\[
D = \{(\lambda,q)\in\cM \mid (\lambda_1-\lambda_2)^2 + 4q = 0  \}.
\]
We label the fixed points of $\PP^1$, $[1,0]$ and $[0,1]$, as $\alpha$ and $\beta$ respectively.
Let $s_\alpha(\lambda,qe^t)$ and $s_\beta(\lambda,qe^t)$ be solutions of $(x-\lambda_1)(x-\lambda_2)-qe^t=0$ for $(\lambda,qe^t)\in\cM\setminus D$.
We remark that these are locally single-valued functions away from $D$.
By direct calculation, we can see that the map $\chtau$ has the following form$:$
\begin{align*}
\chtau_\alpha =& \left(\tau_0 + \lambda_1 t + 2s_\alpha - \lambda_1 - \lambda_2 - (\lambda_1 - \lambda_2) \log(s_\alpha-\lambda_2)\right) + \left(\tau_1 + \log(s_\alpha-\lambda_2)\right)\phi \\
&- \frac{s_\alpha-\lambda_2}{2(2s_\alpha - \lambda_1 - \lambda_2)}Qe^{\tau_1} + \frac{(s_\alpha-\lambda_2)(6s_\alpha-5\lambda_1-\lambda_2)}{12(2s_\alpha - \lambda_1 - \lambda_2)^3}Qe^{\tau_1}\phi, \\
\chtau_\beta =& \left(\tau_0 + \lambda_2 t + 2s_\beta - \lambda_1 - \lambda_2 - (\lambda_2 - \lambda_1) \log(s_\beta-\lambda_1)\right) + \left(\tau_1 + t - \log(s_\beta-\lambda_1)\right)\phi \\
&- \frac{s_\beta-\lambda_2}{2(2s_\beta - \lambda_1 - \lambda_2)}Qe^{\tau_1} + \frac{(s_\beta-\lambda_2)(6s_\beta-5\lambda_1-\lambda_2)}{12(2s_\beta - \lambda_1 - \lambda_2)^3}Qe^{\tau_1}\phi.
\end{align*}
To simplify notation, we set $t=\tau_1=0$ from now on.
Note that we can recover the original $J_{\chtau}$ from $J_{\chtau}|_{t=\tau_1=0}$ since $t$ and $\tau_1$ always appear in the map $J_{\chtau}$ in the form of $qe^t$ and $Qe^{\tau_1}$.
As written in the previous example, small quantum cohomology $QH^*_T(F_1)$ and $QH^*_T(\PP^1)$ can be described as
\begin{align*}
QH^*_T(F_1)_{\lambda,q,Q} &\cong K[ \phi, P ] \left/ \Bigl( \phi^2 - Q(P-\phi-\lambda_2), (P-\lambda_1)(P-\phi-\lambda_2) - q \Bigr) \right. , 	\\
QH^*_T(\PP^1)_{\lambda,Q_\alpha,\chtau_\alpha} &\cong K[ \phi ] \left/ (\phi^2 - Q(s_\alpha-\lambda_2)) \right. ,\qquad Q_\alpha=Q, \\
QH^*_T(\PP^1)_{\lambda,Q_\beta,\chtau_\beta} &\cong K[ \phi ] \left/ (\phi^2 - Q(s_\beta-\lambda_2)) \right. ,\qquad Q_\beta=qQ.
\end{align*}
Via these identifications, the decomposition
\[
J_{\chtau} \colon QH^*_T(F_1)_{\lambda,q,Q} \to QH^*_T(\PP^1)_{\lambda,Q_\alpha,\chtau_\alpha}\oplus QH^*_T(\PP^1)_{\lambda,Q_\beta,\chtau_\beta}
\]
can be represented as follows$:$
\begin{align*}
J_{\chtau_\gamma}(\phi) = \parfrac{\chtau_\gamma}{\tau_1} =& \phi - \frac{s_\gamma-\lambda_2}{2(2s_\gamma - \lambda_1 - \lambda_2)}Q + \frac{(s_\gamma-\lambda_2)(6s_\gamma-5\lambda_1-\lambda_2)}{12(2s_\gamma - \lambda_1 - \lambda_2)^3}Q\phi	\\
J_{\chtau_\gamma}(P) = \parfrac{\chtau_\gamma}{t} =& s_\gamma + \frac{s_\gamma-\lambda_1}{2s_\gamma - \lambda_1 - \lambda_2}\phi + \frac{q(\lambda_1-\lambda_2)}{2(2s_\gamma - \lambda_1 - \lambda_2)^3}Q \\
&- \frac{q(12q + (\lambda_1-\lambda_2)(4s_\gamma-5\lambda_1+\lambda_2))}{12(2s_\gamma - \lambda_1 - \lambda_2)^5}Q\phi
\end{align*}
for $\gamma \in \{ \alpha, \beta \}$.
Note that we can define $J_{\chtau}$ when $q=0$, though we may not be able to define $\chtau$ at $q=0$.
It is easy to check that the decomposition $J_{\chtau}$ coincides with the localization isomorphism for $H^*_T(F_1)$ when we set the Novikov variables $q=Q=0$.

Finally, we describe the non-equivariant limit of the decomposition $J_{\chtau}$.
In the non-equivariant case, $s_+(q):=s_\alpha(0,q)$ and $s_-(q):=s_\beta(0,q)$ become the solutions of $x^2-q=0$, and we write $s_\pm(q)=\pm q^{1/2}$.
By setting $\lambda$ to be zero in the above discussion, we obtain the following description for non-equivariant quantum cohomology$:$
\begin{align*}
QH^*(F_1)_{q,Q} &\cong K[ \phi, P ] \left/ \Bigl( \phi^2 - Q(P-\phi), P(P-\phi) - q \Bigr) \right. , 	\\
QH^*(\PP^1)_{Q_\pm,\chtau_\pm} &\cong K[ \phi ] \left/ (\phi^2 \mp q^{\frac{1}{2}} Q) \right. .
\end{align*}
We also obtain the decomposition for the non-equivariant quantum cohomology$:$
\[
J_{\chtau} \colon QH^*(F_1)_{q,Q} \to QH^*(\PP^1)_{Q_+,\chtau_+}\oplus QH^*(\PP^1)_{Q_-,\chtau_-}
\]
\[
J_{\chtau_\pm}(\phi) = \phi - \frac{1}{4}Q \pm \frac{1}{16}q^{-\frac{1}{2}}Q\phi, \quad J_{\chtau_\pm}(P) = \pm q^{\frac{1}{2}} + \frac{1}{2}\phi \mp\frac{1}{32}Q\phi. 
\]
Again, we see from Theorem \ref{thm:analyticity_star} that the map $\chtau$ and the decomposition $J_{\chtau}$ are convergent with respect $Q$.

\end{example}

\appendix
\section{Proof of Lemma \ref{lem:appendix1}}
\label{app:1}

In this appendix, we will prove that $\ringA (U)$ and $\ringB (U)$, which are defined in Definition \ref{def:rings2}, are rings.
Here $U$ is an open subset of $\hcM = \C^k_q \times \C^N_\lambda$.
It is enough to confirm that the former is a ring since $\ringB (U)$ consists of elements of $\ringA (U)$ which are non-negative power series for $z$.

Let $a = \sum_{I,\bm,n} a_{I,\bm,n} \sigma^I x^\bm z^n$ and $b = \sum_{I,\bm,n} b_{I,\bm,n} \sigma^I x^\bm z^n$ belong to $\ringA (U)$.
It is easy to see that $a + b \in \ringA (U)$ and thus we only need to check that $a \cdot b$ also belongs to $\ringA (U)$.
Since $a,b \in \ringA (U)$, there exist a positive continuous function $C \colon U \to [1,+\infty)$ and a positive constant $C'\geq1$ satisfying
\begin{itemize} 
\item $a_{I,\bm,n} (q,\lambda) = b_{I,\bm,n} (q,\lambda) = 0$ for $n<-C'(| I | + | \bm | + 1)$, and
\item the following estimate for coeffitients:
\[
\max \{ |a_{I,\bm,n}(q,\lambda)|, |b_{I,\bm,n}(q,\lambda)| \} \leq C(q,\lambda)^{| I | + | \bm | + | n | + 1} (| I | + | \bm |)^n
\]
for any $(q,\lambda) \in U$.
\end{itemize}
If we write $a \cdot b = \sum_{I,\bm,n} (a \cdot b)_{I,\bm,n}(q,\lambda) \sigma^I x^\bm z^n$, we have
\[
(a \cdot b)_{I,\bm,n}(q,\lambda) = \sum_{I_1,I _2; I_1+I_2=I} \sum_{\bm_1, \bm_2; \bm_1+\bm_2=\bm} \sum_{n_1, n_2; n_1+n_2=n} a_{I_1,\bm_1,n_1} (q,\lambda) \cdot b_{I_2,\bm_2,n_2} (q,\lambda).
\] 

\begin{lemma}
\label{lem:C'}
The coefficients $(a \cdot b)_{I,\bm,n}$ equal zero if $n < - 2C' (|I| + |\bm| + 1)$.
\end{lemma}

\begin{proof}

If $n$ is less than $- C' (|I| + |\bm| + 2)$, it follows that $n_1 < - C' (|I_1| + |\bm_1| + 1)$ or $n_2 < - C' (|I_2| + |\bm_2| + 1)$ and hence $(a \cdot b)_{I,\bm,n} = 0$.
In particular, $(a \cdot b)_{I,\bm,n} = 0$ if $n < - 2C' (|I| + |\bm| + 1)$.

\end{proof}

Now we estimate the coefficients $| (a \cdot b)_{I,\bm,n} |$.
We have
\[
| (a \cdot b)_{I,\bm,n} | \leq \sum_{I_1,I _2; I_1+I_2=I} \sum_{\bm_1, \bm_2; \bm_1+\bm_2=\bm} \sum_{n_1, n_2; n_1+n_2=n} | a_{I_1,\bm_1,n_1} (q,\lambda) \cdot b_{I_2,\bm_2,n_2} (q,\lambda) |.
\]
Note that $| a_{I_1,\bm_1,n_1} (q,\lambda) \cdot b_{I_2,\bm_2,n_2} (q,\lambda) | = 0$ if $n_1 < - C' (|I_1| + |\bm_1| + 1)$ or $n_2 < - C' (|I_2| + |\bm_2| + 1)$.
The number of non-zero terms of the right-hand side of this inequality is not greater than $\max\{ 2^{|I|+|\bm|} ( n + 2 C' (|I| + |\bm| + 1 ) ), 0 \}$.
Therefore, if we write
\[
M(q,\lambda) = \max\{ | a_{I_1,\bm_1,n_1} (q,\lambda) \cdot b_{I_2,\bm_2,n_2} (q,\lambda) | \mid I_1+I_2=I, \bm_1+\bm_2=\bm, n_1+n_2=n \},
\]
then we have
\begin{align*}
| (a \cdot b)_{I,\bm,n}(q,\lambda) | 
&\leq \max\{ 2^{|I|+|\bm|} ( n + 2 C' (|I| + |\bm| + 1 ) ), 0 \} \cdot M(q,\lambda) \\
&\leq C_1^{|I|+|\bm|+|n|+1} \cdot M(q,\lambda)
\end{align*}
for some positive constant $C_1\geq1$.

Using the properties of $|a_{I,\bm,n}(q,\lambda)|$ and $|b_{I,\bm,n}(q,\lambda)|$, we have
\begin{align*}
| a_{I_1,\bm_1,n_1} (q,\lambda) \cdot b_{I_2,\bm_2,n_2} (q,\lambda) | 
&\leq C(q,\lambda)^{|I|+|\bm|+|n_1|+|n_2|+2} (|I_1|+|\bm_1|)^{n_1} (|I_2|+|\bm_2|)^{n_2} \\
&\leq C_2(q,\lambda)^{|I|+|\bm|+|n|+1} (|I_1|+|\bm_1|)^{n_1} (|I_2|+|\bm_2|)^{n_2}
\end{align*}
for some positive function $C_2(q,\lambda)$.
To obtain an estimate for $M(q,\lambda)$, we calculate the maximum of $(|I_1|+|\bm_1|)^{n_1} (|I_2|+|\bm_2|)^{n_2}$.

\begin{lemma}
Let $m$ be a non-negative integer and $n$ be an integer.
Define a subset $A_{m,n} \subset \Z^4$ consisting of quadruples $(m_1,m_2,n_1,n_2)$ which satisfy the following condition$:$
\begin{itemize}
\item $m_1 \geq 0$, $m_2 \geq 0$,
\item $m_1 + m_2 = m$, $n_1 + n_2 = n$,
\item $n_1 \geq -C'( m_1 + 1 )$, $n_2 \geq -C'( m_2 + 1 )$.
\end{itemize}
Then, there exists a constant $C_0\geq1$ such that
\[
\max_{A_{m,n}} (m_1^{n_1} m_2^{n_2}) \leq C_0^{m+|n|+1} m^n.
\]
\end{lemma}

\begin{proof}

By symmetry, we can assume $m_1 \geq m_2$, or equivalently $m_1 \geq m/2 \geq m_2$.
In this case, the larger $n_1$ is, the larger $m_1^{n_1} m_2^{n_2}$ is.
From the assumption, $n_1$ is no more than $n + C'(m_2 + 1)$ and we have
\[
m_1^{n_1} m_2^{n_2} \leq \frac{ m_1^{n + C'(m_2 + 1)} }{ m_2^{C'(m_2 + 1)} }.
\] 

If $m=0$, then $(m_1,m_2)=(0,0)$ and $m_1^{n_1} m_2^{n_2} \leq 1$.
In the case that $m=1$, we have $(m_1,m_2)=(1,0)$ and $m_1^{n_1} m_2^{n_2} \leq 1$.
Hereafter, we assume that $m\geq 2$.
We define the real function $f(x)$ to be
\[
f(x) = (n + C'x + C') \log(m-x) - (C'x + C') \log x, \qquad 0< x\leq \frac{m}{2}
\]
and set $f(0) := (n + C') \log m$.
This function satisfies that
\[
\frac{ m_1^{n + C'(m_2 + 1)} }{ m_2^{C'(m_2 + 1)} } = e^{f(m_2)},
\]
thus we only have to estimate $M_f := \max\{ f(x) \mid x\in\Z, 0\leq x\leq m/2 \}$.

We have
\[
M_f \leq \max \left( \left\{ f(0), f(1), f\left(\frac{m}{2}\right) \right\} \cup \left\{ f(\alpha) \;\middle|\; f'(\alpha)=0, 1\leq\alpha<\frac{m}{2} \right\} \right).
\]
Since $f(0) = (n + C') \log m$, $f(1) = (n + 2C') \log(m-1)$ and $f(m/2) = n \log (m/2)$, we have
\begin{align*}
e^{f(0)} &= m^{n + C'} \leq C_3^m m^n, \\
e^{f(1)} &= (m-1)^{n + 2C'} \leq C_3^{m + |n|} m^n, \\
e^{f\left(\frac{m}{2}\right)} &= \left( \frac{m}{2} \right)^n \leq 2^{|n|} m^n
\end{align*}
for some positive constant $C_3\geq1$.

We proceed to estimate the critical values of the function $f(x)$.
Let $\alpha$ be a critical point of $f$.
Since
\[
f'(x) = C'\log\frac{m-x}{x} - \frac{1}{x(m-x)} [ (n+C'm)x + C'm ],
\]
we have
\[
C'\log\frac{m-\alpha}{\alpha} = \frac{1}{\alpha(m-\alpha)} [ (n+C'm)\alpha + C'm ].
\]
Using this equation, the critical value can be written as follows$:$
\begin{align*}
f(\alpha)
&= n \log(m-\alpha) + (\alpha + 1) \cdot C'\log\frac{m-\alpha}{\alpha} \\
&= n \log(m-\alpha) + \frac{\alpha + 1}{\alpha} \cdot \frac{(n+C'm)\alpha + C'm}{m-\alpha}
\end{align*}
Since $1\leq\alpha\leq m/2$, we have
\begin{align*}
n\log(m-\alpha) &\leq |n|\log2 + n\log m,								\\
\frac{\alpha + 1}{\alpha} &\leq 2,									\\
\frac{(n+C'm)\alpha + C'm}{m-\alpha} &\leq \frac{ (|n| + C'm)\cdot\frac{m}{2} + C'm }{\frac{m}{2}}	\\
&= |n| + C'm + 2C'.
\end{align*}
Combining these inequalities, we have
\[
f(\alpha) \leq |n|\log2 + 2(|n| + C'm + 2C') + n\log m,
\]
and therefore we obtain the following estimate$:$
\[
e^{f(\alpha)} \leq 2^{|n|} e^{2(|n| + C'm + 2C')} m^n.
\]

Finally, we estimate $\max_{A_{m,n}} (m_1^{n_1} m_2^{n_2})$ as follows$:$
\begin{align*}
\max_{A_{m,n}} (m_1^{n_1} m_2^{n_2})
&\leq e^{M_f}															\\
&\leq \max \left\{ e^{f(0)}, e^{f(1)}, e^{f\left(\frac{m}{2}\right)}, 2^{|n|} e^{2(|n| + C'm + 2C')} m^n \right\}	\\
&\leq C_0^{m+|n|+1} m^n
\end{align*}
for some positive constant $C_0\geq1$.

\end{proof}

By this proposition, there exists a positive constant $C_4\geq1$ such that
\[
(|I_1|+|\bm_1|)^{n_1} (|I_2|+|\bm_2|)^{n_2} \leq C_4^{|I|+|\bm|+|n|+1} (|I|+|\bm|)^n.
\]
Using this inequality, we obtain the following estimates$:$
\begin{align*}
M(q,\lambda) &\leq (C_4 \cdot C_2(q,\lambda))^{|I|+|\bm|+|n|+1} (|I|+|\bm|)^n,					\\
| (a \cdot b)_{I,\bm,n}(q,\lambda) | &\leq (C_1 C_4 \cdot C_2(q,\lambda))^{|I|+|\bm|+|n|+1} (|I|+|\bm|)^n.
\end{align*}
The latter inequality and Lemma \ref{lem:C'} indicate that $a \cdot b \in \ringA (U)$.

\end{document}